\documentclass[10pt]{amsart}

\usepackage[hmargin=1in, vmargin=1in]{geometry}

\usepackage{graphicx,amssymb,epstopdf,subcaption,mathtools,tikz-cd,pb-diagram,thmtools, thm-restate,pinlabel,pifont,enumitem,stmaryrd}
\usepackage[figurewithin=section]{caption}

\newcommand{\nc}{\newcommand}
\nc{\dmo}{\DeclareMathOperator}
\nc{\nt}{\newtheorem}

\nt{theorem}{Theorem}[section]
\nt{corollary}[theorem]{Corollary}
\nt{lemma}[theorem]{Lemma}
\nt{criterion}[theorem]{Criterion}
\nt*{theorem*}{Theorem}
\nt{question}[theorem]{Question}
\nt{problem}[question]{Problem}
\nt{conjecture}[question]{Conjecture}
\nt{proposition}[theorem]{Proposition}

\newcommand{\cmark}{\ding{51}}
\newcommand{\xmark}{\ding{55}}

\nc{\M}{\mathcal{M}}
\nc{\C}{\mathcal{C}}

\nc{\cut}{\!\ssearrow\!}

\dmo{\Mod}{Mod}
\dmo{\SMod}{SMod}
\dmo{\PMod}{PMod}
\dmo{\AMod}{AMod}
\dmo{\I}{\mathcal{I}}
\dmo{\Sp}{Sp}
\dmo{\SL}{SL}
\dmo{\PSp}{PSp}
\dmo{\PSL}{PSL}
\dmo{\Homeo}{Homeo}

\nc{\Z}{\mathbb Z}
\nc{\N}{\mathcal N}
\nc{\R}{\mathbb R}
\nc{\F}{\mathcal F}

\nc{\ga}{\gamma}
\nc{\de}{\delta}
\nc{\ep}{\epsilon}

\nc{\flm}{\lambda_{2}}

\nc{\normalclosure}[1]{\ensuremath{\left \langle \left \langle #1 \right \rangle \right \rangle}}

\nc{\margin}[1]{\marginpar{\scriptsize #1}}

\title{Normal generators for mapping class groups are abundant}
\author{Justin Lanier}
\author{Dan Margalit}

\address{Justin Lanier \\ School of Mathematics\\ Georgia Institute of Technology \\ 686 Cherry St. \\ Atlanta, GA 30332 \\  jlanier8@gatech.edu}

\address{Dan Margalit \\ School of Mathematics\\ Georgia Institute of Technology \\ 686 Cherry St. \\ Atlanta, GA 30332 \\  margalit@math.gatech.edu}

\begin{document}

\maketitle

\vspace*{-4ex}

\begin{abstract}
We provide a simple criterion for an element of the mapping class group of a closed surface to be a normal generator for the mapping class group.  We apply this to show that every nontrivial periodic mapping class that is not a hyperelliptic involution is a normal generator for the mapping class group when the genus is at least 3.  We also give many examples of pseudo-Anosov normal generators, answering a question of D.~D. Long.  In fact we show that every pseudo-Anosov mapping class with stretch factor less than $\sqrt{2}$ is a normal generator.  Even more, we give pseudo-Anosov normal generators with arbitrarily large stretch factors and arbitrarily large translation lengths on the curve graph, disproving a conjecture of Ivanov.
\end{abstract}


\vspace*{0in}

\section{Introduction}

Let $S_g$ denote a connected, closed, orientable surface of genus $g$. The mapping class group $\Mod(S_g)$ is the group of homotopy classes of orientation-preserving homeomorphisms of $S_g$.  The goal of this paper is to give new examples of elements of $\Mod(S_g)$ that have normal closure equal to the whole group; in this case we say that the element \emph{normally generates} $\Mod(S_g)$.

In the 1960s Lickorish \cite{Lickorish} and Mumford \cite{mumford} proved that $\Mod(S_g)$ is normally generated by a Dehn twist about a nonseparating curve in $S_g$.  On the other hand for $k > 1$ the $k$th power of a Dehn twist is not a normal generator since it acts trivially on the mod $k$ homology of $S_g$.

The Nielsen--Thurston classification theorem for mapping class groups categorizes elements of $\Mod(S_g)$ as either periodic, reducible, or pseudo-Anosov; see \cite[Chapter 13]{Primer}.  Dehn twists and their powers are examples of reducible elements.  Our primary focus in this paper is to find normal generators for $\Mod(S_g)$ among the periodic and pseudo-Anosov elements.  

\subsection*{Periodic elements} By the work of Harvey, Korkmaz, McCarthy, Papadopoulos, and Yoshihara \cite{HK,Korkmaz,McPap,YoshiharaTalk}, there are several specific examples of periodic mapping classes that normally generate $\Mod(S_g)$.  The first author recently showed \cite{Lanier} that for $k \geq 5$ and $g \geq (k-1)^2+1$, there is a mapping class of order $k$ that normally generates $\Mod(S_g)$.  

Our first theorem completely answers the question of which periodic elements normally generate.  In the statement, the hyperelliptic involution is the element (or, conjugacy class) of $\Mod(S_g)$ depicted in Figure~\ref{fig:hi}.

\begin{theorem}
\label{main:periodic}
For $g \geq 3$, every nontrivial periodic mapping class that is not a hyperelliptic involution normally generates $\Mod(S_g)$.
\end{theorem}

Additionally, we show in Proposition~\ref{etorelli} that for $g \geq 3$ the normal closure of the hyperelliptic involution is the preimage of $\{\pm I\}$ under the standard symplectic representation of $\Mod(S_g)$.  The Torelli group $\I(S_g)$ is the kernel of this representation, so the normal closure of the hyperelliptic involution contains $\I(S_g)$ as a subgroup of index 2.  We also give in Section~\ref{sec:periodic} a complete classification of the normal closures of periodic elements of $\Mod(S_g)$ when $g \leq 2$.  

In Section~\ref{sec:periodic} we give an extension of Theorem~\ref{main:periodic} to surfaces $S_{g,n}$ of genus $g \geq 3$ with $n$ marked points (equivalently, punctures).  Specifically, we show in Theorem~\ref{thm:punc} that the normal closure of every nontrivial periodic mapping class that is not a hyperelliptic involution contains the pure mapping class group $\PMod(S_{g,n})$.  Here, a hyperelliptic involution is a periodic element that maps to a hyperelliptic involution in $\Mod(S_g)$ under the forgetful map $\Mod(S_{g,n}) \to \Mod(S_g)$.  Note that $\PMod(S_{g,1})$ is equal to the full mapping class group $\Mod(S_{g,1})$.  

Several well-known facts about mapping class groups follow from Theorem~\ref{main:periodic}.   For instance, the\ Torelli group $\mathcal{I}(S_g)$ is torsion free.  Also, the level $m$ congruence subgroup $\Mod(S_g)[m]$, the kernel of the action of $\Mod(S_g)$ on $H_1(S_g;\Z/m\Z)$, is torsion free for $m \geq 3$.  Additionally, it follows that there are no finite nontrivial normal subgroups of $\Mod(S_g)$ except for the cyclic subgroups of $\Mod(S_1)$ and $\Mod(S_2)$ generated by the hyperelliptic involution.  One new consequence is that any normal subgroup of $\Mod(S_g)$ not containing $\I(S_g)$ is torsion free.

Another consequence of Theorem~\ref{main:periodic} involves the subgroup of $\Mod(S_g)$ generated by $n$th powers of all elements.  Let $L$ be the least common multiple of the orders of the periodic elements of $\Mod(S_g)$.  For $n$ not divisible by $L/2$ we show that the $n$th power subgroup of $\Mod(S_g)$ is equal to the whole group $\Mod(S_g)$; see Corollary~\ref{cor:funar}.  This improves on a result of Funar \cite{funar}, who proved the analogous result with $L/2$ replaced by $4g+2$.  Funar's theorem answers in the negative a question of Ivanov, who asked in his problem list \cite[Problem 13]{nvi15} if the $n$th power subgroup of $\Mod(S_g)$ has infinite index for $n$ sufficiently large.

\begin{figure}[h!]
\centering
\includegraphics[scale=.5]{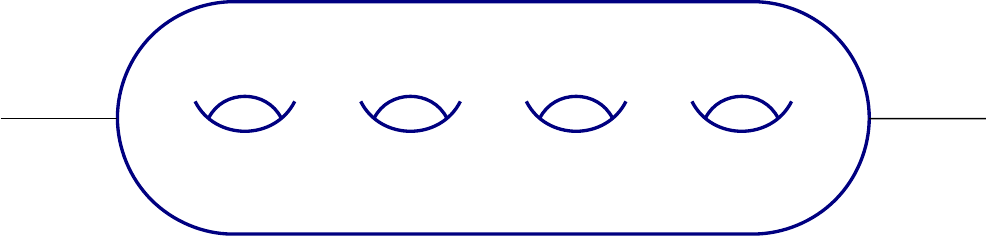}
\caption{Rotation by $\pi$ about the indicated axis is a hyperelliptic involution}
\label{fig:hi}
\end{figure}

It follows from the fact that $\Mod(S_g)$ has a periodic normal generator that any homomorphism from $\Mod(S_g)$ to a torsion-free group is trivial.  Theorem~\ref{main:periodic} gives a strengthening: any homomorphism from $\Mod(S_g)$ to a group without the ``right'' torsion must be trivial.  Theorem~\ref{main:periodic}, its extension Theorem~\ref{thm:punc} for punctured surfaces, and the surrounding ideas have already been leveraged in this way to prove several results about homomorphisms of mapping class groups:
\begin{enumerate}
\item Mann--Wolf showed that for $g \geq 3$ any homomorphism $\Mod(S_{g,1}) \to \textrm{Homeo}^+(S^1)$ is either trivial or equivalent to the standard Gromov boundary action \cite{mann}.
\item Chen and the first author proved that for $g \geq 3$ and $h < 2g-1$ with $h \neq g$, any homomorphism $\Mod(S_g) \to \Mod(S_h)$ is trivial \cite{chenlanier}.
\item Chen, Kordek, and the second author showed that any homomorphism from the braid group $B_n$ to the braid group $B_{2n}$ is either cyclic or is equivalent to one of the standard inclusions \cite{ckm}.
\end{enumerate}
The first result answers a special case of a question of Farb \cite[Questions 6.2]{FarbProblems}.  The second result  extends a special case of a result of Aramayona--Souto; they proved an analogous statement for surfaces of genus $g \geq 6$, possibly with punctures or boundary \cite[Theorem 1.1]{AS}.

The analogue of Theorem~\ref{main:periodic} for non-orientable surfaces was recently proved by Le\'sniak \cite{lesniak}.


\subsection*{Pseudo-Anosov elements with small stretch factor} Having addressed the periodic elements, we turn to the case of pseudo-Anosov mapping classes.  In a 1986 paper \cite{Long}, Darren Long asked: 
\begin{quote}
\emph{Can the normal closure of a pseudo-Anosov map ever be all of $\Mod(S_g)$?}
\end{quote}
Long answered the question in the affirmative for $g = 1$.  In Proposition~\ref{prop:long} below we give a flexible construction that gives many pseudo-Anosov normal generators for each $g \geq 1$, thus answering Long's question.

Penner constructed a family of pseudo-Anosov mapping classes, one for each $g$, with the property that the stretch factors tend to 1 as $g$ tends to infinity \cite{Penner}.  We show in Proposition~\ref{prop:penner} that each of these small stretch factor pseudo-Anosov mapping classes is a normal generator.

Our second main theorem shows that in fact every pseudo-Anosov mapping class with sufficiently small stretch factor is a normal generator.

\begin{theorem}
\label{main:pa}
\label{pA}
If a pseudo-Anosov element of $\Mod(S_g)$ has stretch factor less than $\sqrt{2}$ then it normally generates $\Mod(S_g)$.  
\end{theorem}

For each $g \geq 3$ there are pseudo-Anosov elements of $\Mod(S_g)$ that satisfy the hypothesis of Theorem~\ref{main:pa}.  Indeed, for $g \geq 4$ we may use the fact that there are pseudo-Anosov mapping classes with stretch factor less than $\phi^{2/(g-1)}$ (see \cite[Proposition A.1]{ALM}) and for $g=3$ we may appeal to the example given by Hironaka \cite[Theorem 1.5]{eko}; see Table 1 in her paper.  On the other hand, for $g \leq 2$ it is known that there are no pseudo-Anosov mapping classes that satisfy the hypothesis of Theorem~\ref{main:pa}.  Indeed in these cases the smallest stretch factors are known and they are greater than $\sqrt{2}$; for $g=1$ this is classical and for $g=2$ this is due to Cho and Ham \cite{CH}.

We can make precise one sense in which the mapping classes with stretch factor less than $\sqrt{2}$ are abundant.  A theorem of Leininger and the second author of this paper \cite[Theorem 1.3]{LM} has the following consequence: for any $d$ there is a polynomial $Q(g)$ with degree $d$ and with positive leading term so that the number of pseudo-Anosov elements of $\Mod(S_g)$ with stretch factor less than $\sqrt{2}$ is bounded below by $Q(g)$ for $g \gg 0$.  

Another way to state Theorem~\ref{main:pa} is:
\begin{quote}
\emph{
If a pseudo-Anosov mapping class lies in any proper normal subgroup of $\Mod(S_g)$, its stretch factor is greater than $\sqrt{2}$.
}
\end{quote}
As such, Theorem~\ref{main:pa} generalizes work of Agol, Farb, Leininger, and the second author of this paper.  Farb, Leininger, and the second author proved that if a pseudo-Anosov mapping class is contained in $\Mod(S_g)[m]$ with $m \geq 3$ then the stretch factor is greater than $1.218$ \cite[Theorem 1.7]{FLM}.  Agol, Leininger, and the second author proved that if a pseudo-Anosov mapping class is contained in $\Mod(S_g)[2]$ then the stretch factor is greater than $1.00031$ \cite[Theorem 1.1]{ALM}.  Our Theorem~\ref{main:pa} improves upon these results in two ways: from congruence subgroups to arbitrary normal subgroups and from 1.218 and 1.00031 to $\sqrt{2}$.  

With regard to the first improvement, there do indeed exist examples of normal subgroups that are not contained in congruence subgroups (and hence not covered by the theorems of Agol, Farb, Leininger, and the second author).  Such examples were recently constructed by Clay, Mangahas, and the second author \cite{cmm}, solving a problem posed in an earlier version of this paper \cite[Problem 1.5]{lanierm}.

With regard to the second improvement, we note that when using the bound $\phi^{2/(g-1)}$, we only obtain examples of pseudo-Anosov mapping classes with stretch factor less than 1.00031 when $g$ is at least 3,106.

The bound $\sqrt{2}$ in the statement of Theorem~\ref{main:pa} comes from the following considerations.  We show in Section~\ref{sec:int} that if the stretch factor of $f$ is bounded above by $\sqrt{2}$ then there is a curve $c$ so that the pairwise geometric intersection numbers of $c$, $f(c)$, and $f^2(c)$ are bounded above by 2.  In other words, the constraint on the stretch factor of $f$ is converted into a constraint on the action of $f$ on curves in $S_g$.  The proof of Theorem~\ref{main:pa} proceeds by classifying all possible triples of curves $(c,f(c),f^2(c))$ with pairwise geometric intersection number at most 2 and treating each possibility on a case-by-case basis.  In order to relax the bound $\sqrt{2}$ to a slightly larger number, our approach would require a consideration of triples of curves $(c,f(c),f^2(c))$ with pairwise geometric intersection number at most $n$, where $n > 2$.  Even for $n=3$, a classification of such triples would be daunting.  And even if this were accomplished, it is possible that there is a pseudo-Anosov $f$ that is not a normal generator and a curve $c$ so that $c$, $f(c)$, and $f^2(c)$ have pairwise intersection at most 3.

That said, it is natural to ask how sharp the bound $\sqrt{2}$ is in Theorem~\ref{main:pa}.  In other words, what is the infimum of the stretch factors of all pseudo-Anosov mapping classes lying in any proper normal subgroup of any $\Mod(S_g)$?  Farb, Leininger, and the second author \cite{FLM} showed that for each $g \geq 2$ there are pseudo-Anosov elements of $\I(S_g)$ with stretch factor at most 62.  Thus the infimum lies between $\sqrt{2}$ and 62.  On the other hand, Farb, Leininger, and the second author proved that for the $k$th term of the Johnson filtration of $\Mod(S_g)$ the smallest stretch factor tends to infinity as $k$ does (independently of $g$), so this gives a sequence of normal subgroups for which the bound $\sqrt{2}$ of Theorem~\ref{main:pa} becomes decreasingly sharp.  It is an interesting problem, already raised by Farb, Leininger, and the second author \cite{FLM}, to understand the smallest stretch factors in various specific normal subgroups, such as the level 2 congruence subgroup $\Mod(S_g)[2]$.

\subsection*{Pseudo-Anosov elements with large stretch factor} Having given many examples of pseudo-Anosov normal generators with small stretch factor, we turn to the question of what other kinds of pseudo-Anosov normal generators may exist.  In this vein, Ivanov \cite[Problem 13]{nvi15} made the following conjecture in his 2006 problems paper: 
\begin{quote}
\emph{{\bf Conjecture.}  If $f$ is a pseudo-Anosov element of a mapping class group $\Mod(S)$ with sufficiently large dilatation coefficient, then the subgroup of $\Mod(S)$ normally generated by $f$ is a free group having as generators the conjugates of $f$. More cautiously, one may conjecture that the above holds for a sufficiently high power $g = f^N$ of a given pseudo-Anosov element $f$.}
\end{quote}
(The term ``dilatation coefficient'' is interchangeable with the term ``stretch factor.'')  In Proposition~\ref{prop:arbi} below we give for each $g \geq 1$ a flexible construction of pseudo-Anosov normal generators for $\Mod(S_g)$ with arbitrarily large stretch factor.  Since $\Mod(S_g)$ is not a free group, this in particular disproves the first part of Ivanov's conjecture.

Much more than this, we have the following theorem.  In the statement, note that $\Mod(S_g)[m]$ is $\Mod(S_g)$ when $m=1$ and it is $\I(S_g)$ when $m=0$.

\begin{theorem}
\label{level}
Let $g \geq 3$.  For each $m \geq 0$ there are pseudo-Anosov mapping classes with arbitrarily large stretch factors whose normal closures in $\Mod(S_g)$ are equal to $\Mod(S_g)[m]$.
\end{theorem}

It is natural to ask if other normal subgroups, such as the Johnson kernel and the other terms of the Johnson filtration, can be obtained as the normal closure in $\Mod(S_g)$ of a single pseudo-Anosov mapping class.  In fact it is an open question whether any of these groups can be obtained as the normal closure in $\Mod(S_g)$ of a single element. 

The second part of Ivanov's conjecture is (perhaps intentionally) ambiguous: the word ``a'' can be interpreted as either ``some'' or ``any.''  Dahmani, Guirardel, and Osin proved that every pseudo-Anosov mapping class has a large power (depending only on $g$) whose normal closure is an all pseudo-Anosov infinitely generated free group \cite{DGO}.  This confirms the second part of Ivanov's conjecture with the ``some'' interpretation.  Their theorem also answers another question in Ivanov's problem list \cite[Problem 3]{nvi15}, which asks if there are any normal, all pseudo-Anosov subgroups of $\Mod(S_g)$.

Our next theorem disproves the second part of Ivanov's conjecture with the ``any'' interpretation.  In the statement, the curve graph $\C(S_g)$ is the graph whose vertices are homotopy classes of simple closed curves in $S_g$ and whose edges are pairs of vertices with disjoint representatives in $S_g$.  Masur and Minsky proved that the asymptotic translation length for a pseudo-Anosov element of $\Mod(S_g)$ is a positive real number \cite[Proposition 4.6]{MM}.

\begin{theorem}
\label{translation}
For each $g \geq 3$ there are pseudo-Anosov mapping classes with the property that all of their odd powers normally generate $\Mod(S_g)$.  Consequently, there are pseudo-Anosov mapping classes with arbitrarily large asymptotic translation lengths on $\C(S_g)$ that normally generate $\Mod(S_g)$. 
\end{theorem}

The pseudo-Anosov mapping classes used to prove Theorem~\ref{translation} are not generic, as their invariant foliations have nontrivial symmetry groups (cf. \cite{masai}).   Clay, Mangahas, and the second author proved that if a pseudo-Anosov mapping class has invariant foliations without symmetries then the normal closure of any sufficiently large power is a free group of infinite rank \cite{cmm}.  Their result confirms a suspicion raised in an earlier version of this paper \cite{lanierm}.  In a similar direction, Maher--Tiozzo \cite{mt} further proved that the typical mapping class (in the sense of random walks) has normal closure isomorphic to an infinitely generated free group (Maher had previously proved that the typical mapping class in this sense in pseudo-Anosov \cite{maher}).  We are also led to ask: if the asymptotic translation length of a pseudo-Anosov mapping class is large, can its normal closure be anything other than the mapping class group, a free group, or perhaps the extended Torelli group?

\subsection*{Moduli spaces} The group $\Mod(S_g)$ can be identified with the orbifold fundamental group of $\M_g$, the moduli space of Riemann surfaces of genus $g$.  This means that Theorems~\ref{main:periodic} and~\ref{main:pa} can both be recast in terms of the geometry of normal covers of $\M_g$.

Periodic elements of $\Mod(S_g)$ correspond to orbifold points in $\M_g$ and so Theorem~\ref{main:periodic} can be interpreted as saying that the only orbifold points in a proper normal cover of $\M_g$ are those arising from the hyperelliptic involution.  In particular, these orbifold points all have order 2 and lie along the hyperelliptic locus.   

Torelli space is the normal cover of $\M_g$ corresponding to the Torelli group.  This space can be described as the space of Riemann surfaces with homology framings.  A further consequence of Theorem~\ref{main:periodic} is that every normal cover of $\M_g$ not covered by Torelli space is a manifold.  In fact a normal cover of $\M_g$ is a manifold if and only if it is not covered by the quotient of Torelli space given by the action of $-I \in \Sp_{2g}(\Z)$.

If we endow $\M_g$ with the Teichm\"uller metric, then pseudo-Anosov elements of $\Mod(S_g)$ correspond exactly to geodesic loops in $\M_g$.  The length of a geodesic loop is the logarithm of the corresponding stretch factor.   Theorem~\ref{main:pa} can thus be interpreted as saying that geodesic loops in $\M_g$ whose lengths are less than $\log \sqrt{2}$ do not lift to loops in any proper normal cover of $\M_g$.  The existence of pseudo-Anosov normal generators for $\Mod(S_g)$ with arbitrarily large stretch factor, shown in Proposition~\ref{prop:arbi}, implies the existence of arbitrarily long geodesic loops in $\M_g$ that do not lift to loops in any proper normal cover.  

\subsection*{The well-suited curve criterion} All of the results about normal generators in this paper are derived from a simple, general principle for determining when a mapping class $f$ normally generates the mapping class group.  We call this principle the well-suited curve criterion.  The principle is that if we can find a curve $c$ so that the configuration $c \cup f(c)$ is ``simple'' enough, then $f$ normally generates the mapping class group.  We give many concrete manifestations of this principle in the paper, namely, as Lemmas~\ref{wscca}, \ref{wsccb}, \ref{wsccsep}, \ref{chen}, \ref{lemma:nonsep lantern}, \ref{good pair}, \ref{gpbp}, and \ref{gp} and Proposition~\ref{prop:wscc}.

Our first example of the well-suited curve criterion, Lemma~\ref{wscca} below, takes the following form for $g \geq 3$:
\begin{quote}
\emph{If $f$ lies in $\Mod(S_g)$ with $g \geq 3$ and if $c$ is a nonseparating curve in $S_g$ with $i(c,f(c)) = 1$ then $f$ is a normal generator for $\Mod(S_g)$.}
\end{quote}
The well-suited curve criterion is a very general and widely applicable principle, and there are many variations besides the ones introduced in this paper.  We expect that this principle can be leveraged to address other problems about mapping class groups and related groups.  For example, we use a variant to prove our result about normal generators for congruence subgroups, Theorem~\ref{level} below.  We also use the principle to give normal generators for certain linear groups.

While having a simple configuration for $c$ and $f(c)$ is a powerful sufficient criterion for a mapping class to be a normal generator, it is not a necessary condition.  This point is underscored by Theorem~\ref{translation}, which gives examples of normal generators with large translation length on the curve graph $\C(S_g)$.  For these examples, the distance between every curve and its image is large, which means that each curve forms a complicated configuration with its image.  Still there is a way to formulate the well-suited curve criterion as a necessary, as well as sufficient, condition for a mapping class to be a normal generator.  We state in Proposition~\ref{prop:wscc} below, the most general version of our well-suited curve criterion, which says that a mapping class is a normal generator if and only if a certain associated curve graph is connected.

Subsequent to our work, a version of the well-suited curve criterion for the braid group was given by Chen, Kordek, and the second author of this paper.  They used the criterion to show that the normal closure of any nontrivial periodic element contains the commutator subgroup of the braid group \cite[Lemma 4.2]{ckm}.  Also, Baik--Kin--Shin--Wu \cite{bksw} applied the well-suited curve criterion to show that normal generators are abundant among the monodromies associated to a hyperbolic fibered 3-manifold with first Betti number equal to 2.

\subsection*{Overview of the paper} The remainder of the paper is divided into two parts.  In the first part, Sections~\ref{sec:criterion}--\ref{sec:other}, we give special cases of the well-suited curve criterion and then use them to prove all of our main results aside from Theorem~\ref{main:pa}.  In the second part, Sections~\ref{sec:int}--\ref{sec:pA}, we prove Theorem~\ref{main:pa}, which is by far our most technical result.

We begin the first part of the paper by proving in Section~\ref{sec:criterion} several special cases of the well-suited curve criterion.   In Section~\ref{sec:periodic}, we use the special cases of the criterion to prove our Theorem~\ref{main:periodic} about periodic elements and to give our extension of Funar's theorem.   In Section~\ref{sec:long} we again use the special cases of the criterion  to answer Long's question in the affirmative, to disprove Ivanov's conjecture, and to prove Theorems~\ref{level} and~\ref{translation}.   In Section~\ref{sec:other} we apply our results about mapping class groups to give normal generators for certain linear groups.  

The second part of the paper begins with Section~\ref{sec:int}, in which we relate small stretch factor to geometric intersection numbers for curves and lay out the plan for the proof of Theorem~\ref{main:pa}.  In Section~\ref{sec:general} we state and prove the general well-suited curve criterion, Proposition~\ref{prop:wscc}, and use it to prove one case of Theorem~\ref{main:pa}.  Then in Section~\ref{sec:configs} we address the main technical obstacle of the paper, by giving a detailed analysis of all possible configurations of certain triples of mod 2 homologous curves that intersect in at most two points pairwise on any closed surface (there are 36 configurations up to stabilization).  In Section~\ref{sec:gp} we give a variant of the well-suited curve criterion that applies to many of the configurations from Section~\ref{sec:configs}.  Finally, in Section~\ref{sec:pA} we use the results of Sections~\ref{sec:int}--\ref{sec:gp} in order to prove Theorem~\ref{main:pa}.

\subsection*{Acknowledgments}
The authors are supported by NSF Grants DGE - 1650044 and DMS - 1510556.  We would like to thank Sebastian Baader, Mladen Bestvina, Lei Chen, Benson Farb, S{\o}ren Galatius, Asaf Hadari, Chris Leininger, Marissa Loving, Curtis McMullen, Gregor Masbaum, Andrew Putman, Nick Salter, Bal\'azs Strenner, Nick Vlamis, and an anonymous referee for helpful comment and conversations.  We are also grateful to Mehdi Yazdi for bringing Long's question to our attention at the 2017 Georgia International Topology Conference.


\section{The well-suited curve criterion: special cases}
\label{sec:criterion}

As discussed in the introduction, the well-suited curve criterion is the principle that if $f$ is a mapping class and $c$ is a curve in $S_g$ so that the configuration $c \cup f(c)$ is simple enough, then the normal closure of $f$ is equal to the mapping class group.  We will give several examples of this phenomenon in this section, in Lemmas~\ref{wscca}, \ref{wsccb}, and~\ref{wsccsep}.  Besides serving as a warmup for the full version of the well-suited curve criterion, these special cases also suffice to prove Theorem~\ref{main:periodic}, to answer Long's question, and to resolve Ivanov's conjecture.

\subsection*{Curves and intersection number} In what follows we refer to a homotopy class of essential simple closed curves in $S_g$ as a ``curve'' and we write $i(c,d)$ for the geometric intersection number between curves $c$ and $d$. 

We will write $[c]$ for the element of $H_1(S_g;\Z)/\{\pm 1\}$ represented by a curve $c$ (the ambiguity comes from the two choices of orientation).  We will write $[c] \mod 2$ for the corresponding element of $H_1(S_g;\Z/2)$.  A useful fact is that whenever $i(c,d)=0$ we have $[c]=[d]$ if and only if $[c] = [d] \mod 2$.

Finally, we write $|\hat\imath|(c,d)$ for the absolute value of the algebraic intersection number between two elements of $H_1(S_g;\Z)$ corresponding to $c$ and $d$.  We will refer this number as simply the algebraic intersection number, as we will have no need to discuss the signed algebraic intersection number.

\subsection*{Normal generators for the commutator subgroup} The following lemma, along with Lemmas~\ref{wscca} and \ref{wsccb}, already appears in the paper by Harvey--Korkmaz \cite[Lemma 3]{HK}.  The ideas also appeared in the earlier works of McCarthy--Papadopoulos \cite{McPap} and Luo \cite{Luo}.  All of our well-suited curve criteria will be derived from this lemma.

The conclusions of Lemmas~\ref{wscca} and \ref{wsccb} only give that the normal closure of a given element $f$ contains the commutator subgroup of $\Mod(S_g)$.  For $g \geq 3$ it is well-known that $\Mod(S_g)$ is perfect \cite[Theorem 5.2]{Primer} and so in these cases the lemmas imply that $f$ is a normal generator.

\begin{lemma}
\label{commutator}
Suppose $c$ and $d$ are nonseparating curves in $S_g$ with $i(c,d)=1$.  Then the normal closure of $T_cT_d^{-1}$ is equal to the commutator subgroup of $\Mod(S_g)$.
\end{lemma}

\begin{proof}

We will show the two inclusions in turn.  For $g \geq 0$, it is known that $\Mod(S_g)$ is generated by Dehn twists about nonseparating curves.  It is also known that the abelianization of $\Mod(S_g)$ is cyclic.  Since the Dehn twists about any two nonseparating curves are conjugate it follows that $T_cT_d^{-1}$, and hence its normal closure, lies in the commutator subgroup of $\Mod(S_g)$.  

It remains to show that the commutator subgroup is contained in the normal closure of $T_cT_d^{-1}$.  Lickorish proved that there is a generating set for $\Mod(S_g)$ where the generators are Dehn twists about nonseparating curves $\{c_1,\dots,c_{3g-3}\}$ in $S_g$ and where each $i(c_i,c_j)$ is at most 1 \cite[Theorem 4.13]{Primer}.  The commutator subgroup of $\Mod(S_g)$ is thus normally generated by the various $[T_{c_i},T_{c_j}]$.  When $i(c_i,c_j) = 0$, the corresponding commutator $[T_{c_i},T_{c_j}]$ is trivial.  The nontrivial commutators $[T_{c_i},T_{c_j}]$ are all conjugate in $\Mod(S_g)$ to $[T_c,T_d]$, where $c$ and $d$ are the curves in the statement of the lemma.  Therefore it suffices to show that the single commutator $[T_c,T_d]$ is contained in the normal closure of $T_cT_d^{-1}$.  

It is a general fact that if $a$ and $b$ are elements of a group $G$ and $N$ is a normal subgroup of $G$ containing $ab^{-1}$ then $[a,b]$ is contained in $N$.  Indeed, if we consider the quotient homomorphism $G \to G/N$ then $a$ and $b$ map to the same element, and so $[a,b]$ maps to the identity.  Applying this general fact to our situation, we have that $[T_c,T_d]$ is contained in the normal closure of $T_cT_d^{-1}$, as desired.
\end{proof}

\subsection*{Two well-suited curve criteria for nonseparating curves} The next two lemmas are special cases of the well-suited curve criterion.

\begin{lemma}
\label{wscca}
Let $g \geq 0$ and let $f \in \Mod(S_g)$.  Suppose that there is a nonseparating curve $c$ in $S_g$ so that $i(c,f(c)) = 1$.  Then the normal closure of $f$ contains the commutator subgroup of $\Mod(S_g)$.
\end{lemma}

\begin{proof}

Consider the commutator $[T_c,f]$.  Since this commutator is equal to the product of $T_cfT_c^{-1}$ and $f^{-1}$, it lies in the normal closure of $f$.  Since $fT_cf^{-1}$ is equal to $T_{f(c)}$ the commutator $[T_c,f]$ is also equal to $T_cT_{f(c)}^{-1}$.  Since $i(c,f(c))=1$, the lemma now follows from Lemma~\ref{commutator}.
\end{proof}

\begin{figure}
\centering
\includegraphics[scale=.4]{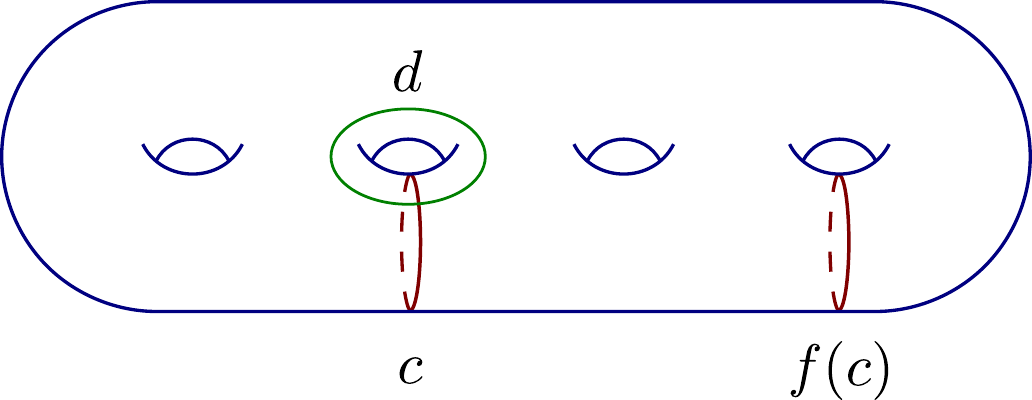}
\caption{The curves $c$, $f(c)$, and $d$ in the proof of Lemma~\ref{wsccb}}
\label{fig:wsccb}
\end{figure}

\begin{lemma}
\label{wsccb}
Let $g \geq 0$ and let $f \in \Mod(S_g)$.  Suppose that there is a nonseparating curve $c$ in $S_g$ so that $i(c,f(c)) = 0$ and $[c] \neq [f(c)]$.  Then the normal closure of $f$ contains the commutator subgroup of $\Mod(S_g)$.
\end{lemma}

\begin{proof}

It follows from the hypotheses on $c$ and $f(c)$ that there is an additional curve $d$ with $i(c,d)=1$, $i(f(c),d)=0$, and $[d] \neq [f(c)]$; see Figure~\ref{fig:wsccb}.  The commutator $[T_c,f] = T_cT_{f(c)}^{-1}$ again lies in the normal closure of $f$.  Since $\Mod(S_g)$ acts transitively on pairs of disjoint, non-homologous curves in $S_g$, there is an $h \in \Mod(S_g)$ taking the pair $(c,f(c))$ to the pair $(f(c),d)$.  The conjugate $h[T_c,f]h^{-1}$, which also lies in the normal closure of $f$, is equal to $T_{f(c)}T_d^{-1}$.  Thus the product
\[
\Big(T_cT_{f(c)}^{-1}\Big)\Big(T_{f(c)}T_d^{-1}\Big) = T_cT_d^{-1}
\]
lies in the normal closure of $f$.  An application Lemma~\ref{wscca} completes the proof.
\end{proof}

\begin{figure}
\centering
\includegraphics[scale=.25]{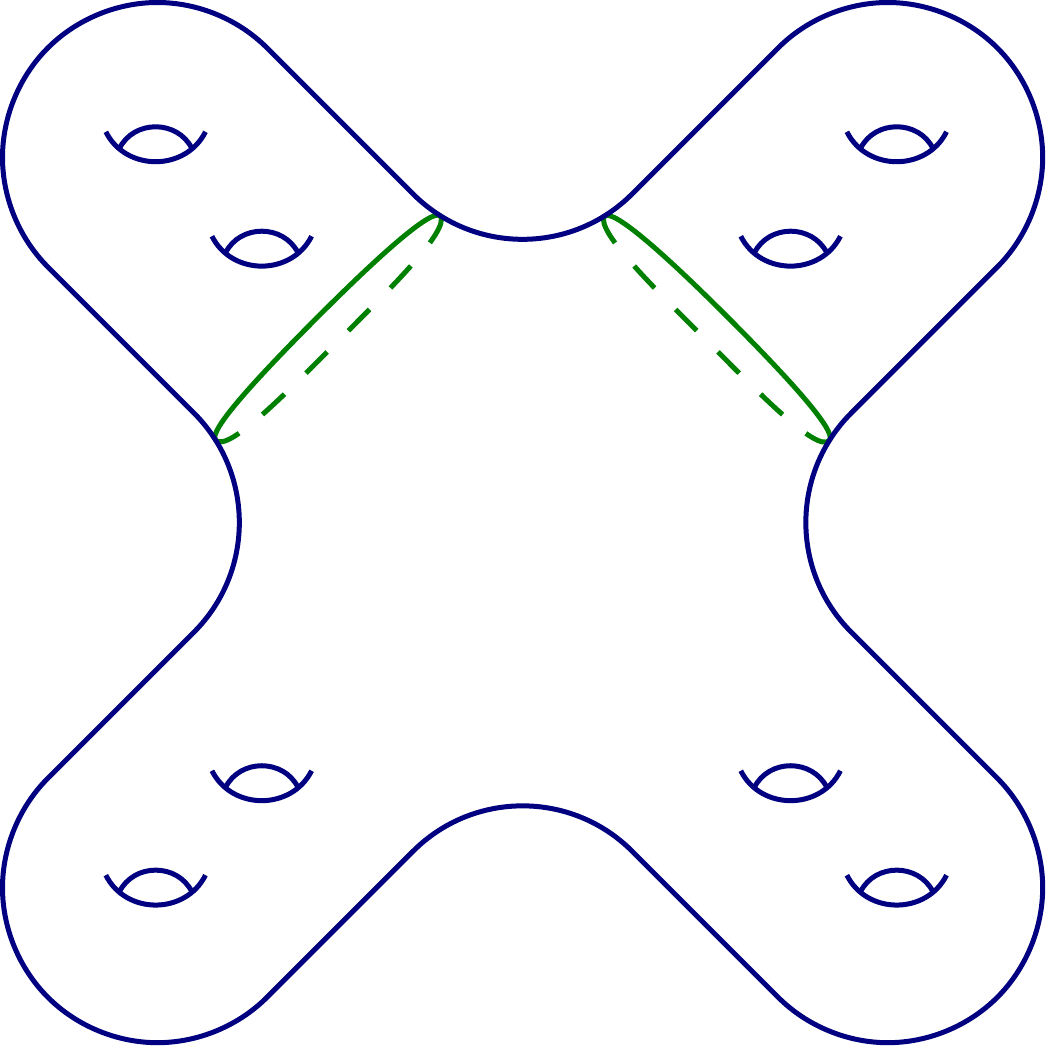}
\qquad \includegraphics[scale=.25]{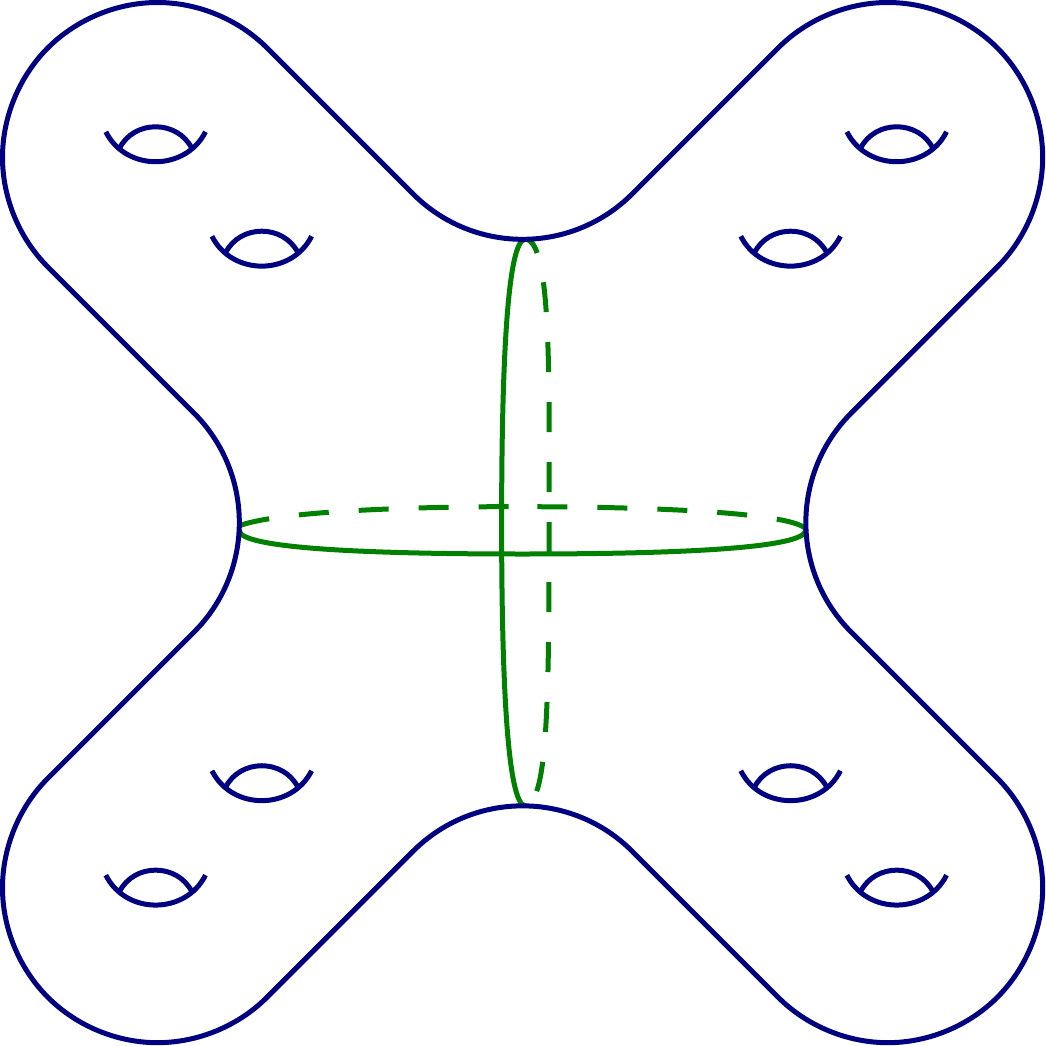}
\caption{The configurations where $c$ is separating and $i(c,f(c)) \leq 2$}
\label{fig:sep}
\end{figure}

\subsection*{A well-suited curve criterion for separating curves} We will now use Lemma~\ref{wsccb} to obtain an instance of the well-suited curve criterion for separating curves, as follows.  

\begin{lemma}
\label{wsccsep}
Let $g \geq 0$ and let $f \in \Mod(S_g)$.  Suppose that there is a separating curve $d$ in $S_g$ with $i(d,f(d)) \leq 2$.  Then the normal closure of $f$ contains the commutator subgroup of $\Mod(S_g)$.
\end{lemma}

\begin{proof}

Let $d$ be a separating curve in $S_g$ with $i(d,f(d)) \leq 2$.  Since the intersection number between two separating curves is even, we have that $i(d,f(d))$ is either 0 or 2.  In each case there is only one possible configuration up to homeomorphism and the genera of the complementary regions; see Figure~\ref{fig:sep}.

In each case we may find nonseparating curves $a$ and $b$ so that $a$ and $b$ lie on different sides of $d$ and on the same side of $f(d)$.  Since $a$ and $b$ lie on different sides of $d$, it follows that $f(a)$ and $f(b)$ lie on different sides of $f(d)$.  Therefore it must be that either $a$ and $f(a)$ lie on different sides of $f(d)$ or $b$ and $f(b)$ do (or both).  Without loss of generality, suppose $a$ and $f(a)$ lie on different sides of $f(d)$.  Then clearly $i(a,f(a))=0$ and $[a] \neq [f(a)]$.  An application of Lemma~\ref{wsccb} completes the proof.
\end{proof}


\section{Application: periodic elements}
\label{sec:periodic}

In this section we apply the special cases of the well-suited curve criterion from Section~\ref{sec:criterion} to determine the normal closure of each periodic element of each $\Mod(S_g)$.  The main technical result is the following.

\begin{proposition}
\label{per comm}
Let $g \geq 0$ and let $f$ be a nontrivial periodic element of $\Mod(S_g)$ that is not a hyperelliptic involution.  Then the normal closure of $f$ contains the commutator subgroup of $\Mod(S_g)$.
\end{proposition}

Since $\Mod(S_g)$ is perfect when $g \geq 3$, Proposition~\ref{per comm} immediately implies Theorem~\ref{main:periodic}.  Later in the section we will use Proposition~\ref{per comm} to determine the normal closures of all periodic elements of $\Mod(S_1)$ and $\Mod(S_2)$.  At the end of the section, we prove Corollary~\ref{cor:funar}, which gives a condition on $n$ so that the $n$th power subgroup of $\Mod(S_g)$ is the whole group.

Our proof of Proposition~\ref{per comm} requires a lemma about roots of the hyperelliptic involution, Lemma~\ref{no hyp} below.  Before giving this lemma, we begin with some preliminaries.  

\subsection*{Standard representatives} It is a classical theorem of Fenchel and Nielsen \cite[Theorem 7.1]{Primer} that a periodic mapping class $f \in \Mod(S_g)$ is represented by a homeomorphism $\phi$ whose order is equal to that of $f$.  Moreover, $\phi$ is unique up to conjugacy in the group of homeomorphisms of $S_g$.  We refer to any such $\phi$ as a \emph{standard representative} of $f$.  

\subsection*{The Birman--Hilden theorem} We now recall a theorem of Birman and Hilden.  Let $i : S_g \to S_g$ be a hyperelliptic involution of $S_g$ (a homeomorphism as in Figure~\ref{fig:hi}) and let $\iota$ be the resulting hyperelliptic involution in $\Mod(S_g)$.  Birman and Hilden proved \cite[Theorem 1]{BH} that for $g \geq 2$ there is a short exact sequence
\[
1 \to \langle \iota \rangle \to \SMod(S_g) \stackrel{\Theta}{\to} \Mod(S_{0,2g+2}) \to 1
\]
where $\SMod(S_g)$ is the centralizer in $\Mod(S_g)$ of $\iota$ and $S_{0,2g+2}$ is a sphere with $2g+2$ marked points.  The map $\Theta : \SMod(S_g) \to \Mod(S_{0,2g+2})$ is defined as follows: it is proved \cite[Theorem 4]{BH2} that each element $h$ of $\SMod(S_g)$ has a representative $\psi$ that commutes with $i$, and so $\psi$ can be pushed down to a homeomorphism $\bar \psi$ of the quotient $S_g/\langle i \rangle$, which is a sphere with $2g+2$ marked points, namely, the images of the fixed points of $\iota$.  We note that the above exact sequence is not correct as stated for $g \leq 2$, as the given map $\SMod(S_g) \to \Mod(S_{0,2g+2})$ is not well-defined. 

\subsection*{Roots of the hyperelliptic involution} The next lemma describes a property of roots of the hyperelliptic involution that will be used in the proof of Proposition~\ref{per comm}.  

\begin{lemma}
\label{no hyp}
Let $g \geq 1$.  Suppose $f \in \Mod(S_g)$ is an $n$th root of a hyperelliptic involution (with $n>1$).  Then there is a power of $f$ that is not the identity or the hyperelliptic involution and that has a standard representative with a fixed point.
\end{lemma}

\begin{proof}

We first dispense with the case $g=1$.  In this case all periodic elements are given by rotations of either the hexagon or square.  In particular, they all have fixed points.  Any nontrivial root $f$ of the hyperelliptic involution is a periodic element, and so the first power of $f$ satisfies the conclusion.

Now assume that $g \geq 2$.  Fix a hyperelliptic involution $i : S_g \to S_g$ and corresponding mapping class $\iota$ as above.  Suppose that $f$ is the $n$th root of $\iota$.  It follows that $f$ lies in $\SMod(S_g)$, the centralizer of $\iota$.  If $\Theta$ is the map from the aforementioned exact sequence of Birman and Hilden, then $\Theta(f)$ is a periodic element of $\Mod(S_{0,2g+2})$.  

Let $\psi$ be a standard representative of $\Theta(f)$.  By ignoring the marked points, we may regard $\psi$ as a finite-order homeomorphism of $S^2$.  A theorem of Brouwer, Eilenberg, and de Ker\'ekj\'art\'o \cite{Brouwer,K,Eilenberg} states that every finite-order orientation-preserving homeomorphism of $S^2$ is conjugate to a rotation.  In particular, $\psi$ has two fixed points in $S^2$.

There are two homeomorphisms of $S_g$ that are lifts of $\psi$; they differ by the hyperelliptic involution $i$.  One of these lifts is a standard representative $\phi$ for $f$ and one is a standard representative for $f\iota$.  If at least one of the fixed points of $\psi$ is marked point, then both lifts of $\psi$ to $S_g$ have a fixed point.  In particular $\phi$ has a fixed point, as desired.

Now suppose that neither of the fixed points of $\psi$ are marked points.  In this case we will show that $f^2$ satisfies the conclusion of the lemma.  

Let $p$ be one of these fixed points and let $\tilde p$ be one point of the preimage in $S_g$.  One of the lifts of $\psi$ fixes $\tilde p$ and one interchanges it with $i(\tilde p)$.  The one that fixes $\tilde p$ cannot be a representative of $f$ since no power of this homeomorphism is $i$ (the only fixed points of $i$ are the preimages of the marked points in $S_{0,2g+2}$).  So the lift of $\psi$ that interchanges $\tilde p$ and $i(\tilde p)$ is a standard representative $\phi$ of $f$.

The homeomorphism $\phi^2$ fixes $\tilde p$.  Clearly then $\phi^2$ is also not equal to $i$ or the identity, since $i$ does not fix $\tilde p$.  Thus $f^2$ is not the identity or a hyperelliptic involution and its standard representative $\phi^2$ fixes a point in $S_g$, as desired.
\end{proof}

\begin{figure}[h]
\centering
\includegraphics[scale=.33]{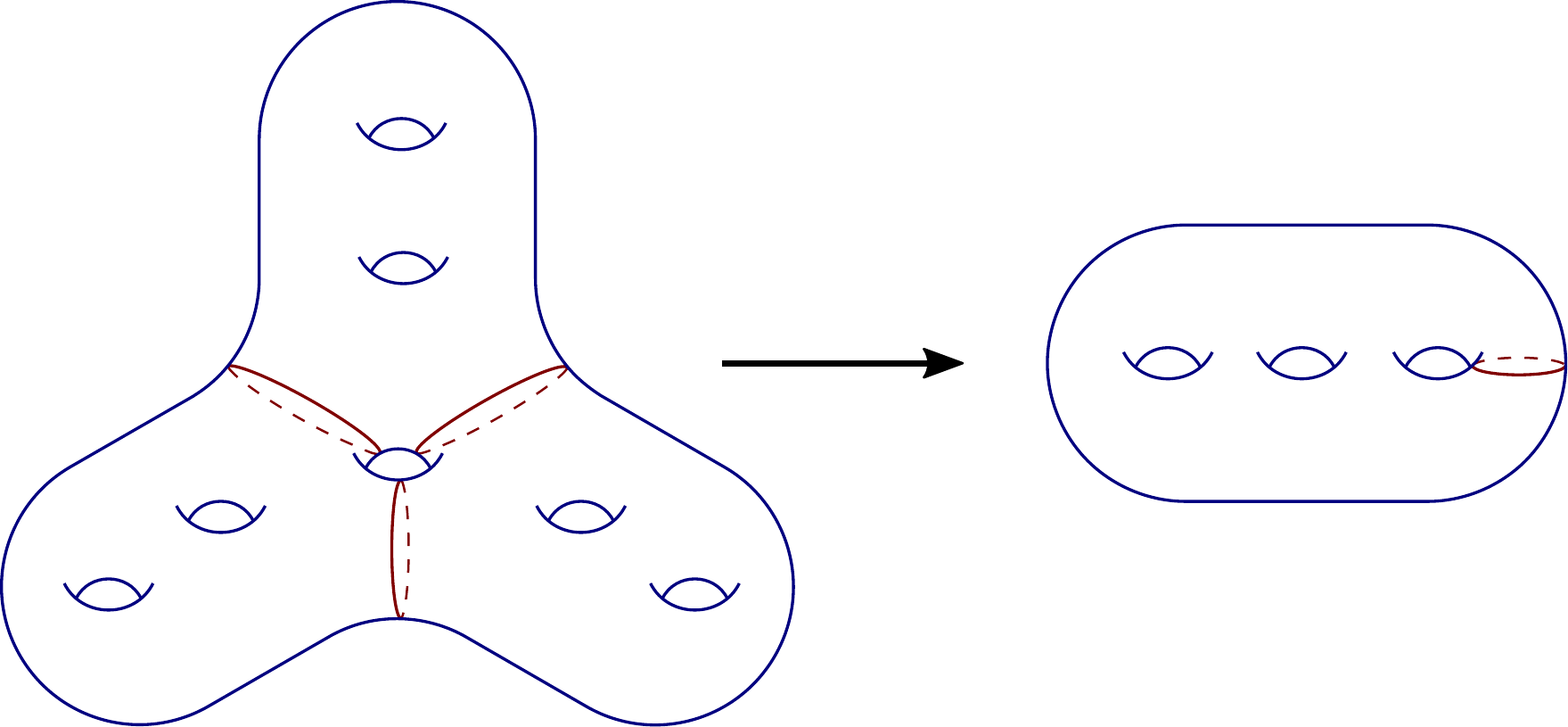}
\caption{A cyclic cover of surfaces}
\label{fig:cover}
\end{figure}

\subsection*{Proof of the theorem} We now prove Proposition~\ref{per comm}, which, as described above, implies Theorem~\ref{main:periodic}.

\begin{proof}[Proof of Proposition~\ref{per comm}]

Let $f$ be a nontrivial periodic mapping class and assume that $f$ is not a hyperelliptic involution.  Let $\phi$ be a standard representative of $f$ and let $\langle \phi \rangle$ denote the cyclic group of homeomorphisms of $S_g$ generated by $\phi$.  

We treat separately three cases:

\smallskip

\hspace*{0ex}\emph{Case 1.} the action of $\langle \phi \rangle$ is free,

\hspace*{0ex}\emph{Case 2.} the action of $\langle \phi \rangle$ is not free and $\phi$ has order 2, and

\hspace*{0ex}\emph{Case 3.} the action of $\langle \phi \rangle$ is not free and $\phi$ has order greater than 2.

\smallskip

We begin with the first case.  If the action of $\langle \phi \rangle$ is free then this action is a covering space action.  Every cyclic covering map $S_g \to S$ is equivalent to one of the covering maps indicated in Figure~\ref{fig:cover}.  In particular, we can find a curve $c$ so that $f$ and $c$ satisfy the hypotheses of Lemma~\ref{wsccb}.  It follows from Lemma~\ref{wsccb} that the normal closure of $f$ contains the commutator subgroup of $\Mod(S_g)$, as desired.

\begin{figure}[h]
\centering
\includegraphics[width=200px]{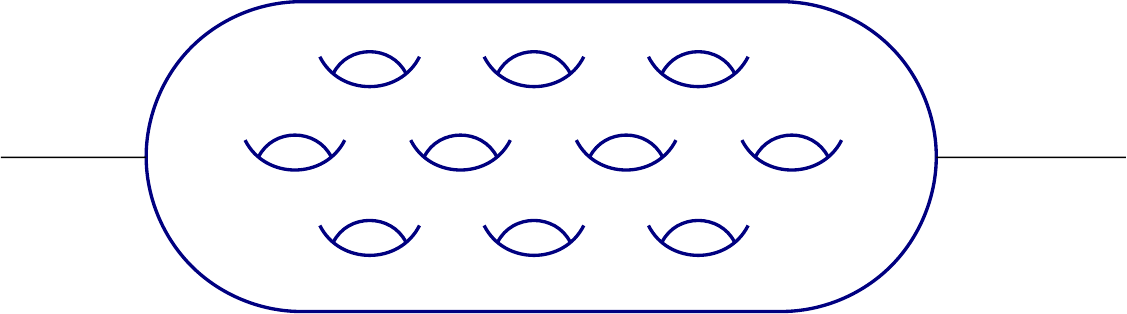}
\caption{Rotation by $\pi$ gives a mapping class of order 2}
\label{fig:skewer}
\end{figure}

We now treat the second case.  There is a classification of homeomorphisms of $S_g$ of order 2 that goes back to the work of Klein \cite{Klein}; see Dugger's paper \cite{Dugger} for a modern treatment.  In the cases where there is a fixed point in $S_g$, such a homeomorphism is conjugate to one of the ones indicated by Figure~\ref{fig:skewer}; there is a 1-parameter family, according to the number of handles above the axis of rotation.  In particular, the conjugacy class of one of these homeomorphisms is completely determined by the genus of the quotient surface $S_g/\langle \phi \rangle$.  When the genus of the quotient is 0, the homeomorphism is a hyperelliptic involution, which is ruled out by hypothesis.  When the genus of the quotient is positive, we can again find a curve $c$ as in Lemma~\ref{wsccb}.  Again by Lemma~\ref{wsccb} the normal closure of $f$ contains the commutator subgroup of $\Mod(S_g)$.  

Finally we treat the third case, where the order of $\phi$ is greater than 2 and the action of $\langle \phi \rangle$ on $S_g$ is not free.  Since the action of $\langle \phi \rangle$ on $S_g$ is not free, some power of $\phi$ has a fixed point.  We may choose this power so that $\phi$ has a fixed point and is not a hyperelliptic involution or the identity.  This is obvious if $f$ is not a root of a hyperelliptic involution and it follows from Lemma~\ref{no hyp} otherwise.  Without loss of generality we may replace $\phi$ by this power with a fixed point, since the normal closure of a power of $f$ is contained in the normal closure of $f$; we may continue to assume that the order of the new $\phi$ is greater than two, for if not we may apply Cases 1 and 2 above.

A theorem of Kulkarni \cite[Theorem 2]{Kulkarni} states that if $\psi$ is a finite order homeomorphism of $S_g$ that has a fixed point, then there is a way to represent $S_g$ as a quotient space of some regular $n$-gon in such a way that $\psi$ is given as rotation of the polygon by some multiple of $2\pi/n$.  We apply this theorem to $\phi$. Let $P$ be the resulting $n$-gon (so $S_g$ is obtained from $P$ by identifying the sides of $P$ in pairs in some way).

Let $c$ be a line segment in $P$ that connects the midpoints of two edges that are identified in $S_g$.  Then $c$ represents a curve in $S_g$.  We may assume that $c$ represents a nontrivial curve in $S_g$, for if all curves in $S_g$ coming from these line segments were trivial, then we would have $g=0$, in which case $\Mod(S_g)$ is trivial.

We may regard $\phi$ as a rotation of $P$ by some multiple of $2\pi/n$.  The image of the line segment $c$ under $\phi$ is another line segment $\phi(c)$.  Since $\phi$ has order greater than 2, it follows that $\phi(c)$ is not equal to $c$, and hence the number of intersections between these line segments is either 0 or 1.  Moreover, this intersection number is equal to $i(c,f(c))$ (here we are regarding $c$ as a curve in $S_g$). 

If $c$ is a separating curve in $S_g$, then the proposition follows from an application of the well-suited curve criterion for separating  curves (Lemma~\ref{wsccsep}).  We may henceforth assume that $c$ is nonseparating. 

If $i(c,f(c))=1$, then it follows from Lemma~\ref{wscca} that the normal closure of $f$ contains the commutator subgroup of $\Mod(S_g)$, as desired.  Finally suppose that $i(c,f(c)) = 0$.  We would like to show that $[c] \neq [f(c)]$.  To do this, it is enough to show that $[c]$, $[f(c)]$, and $[f^2(c)]$ are not all equal (if $[f(c)] = [c]$ then it would follow that $[f^2(c)] = [c]$).

\begin{figure}[h]
\centering
\includegraphics[scale=.225]{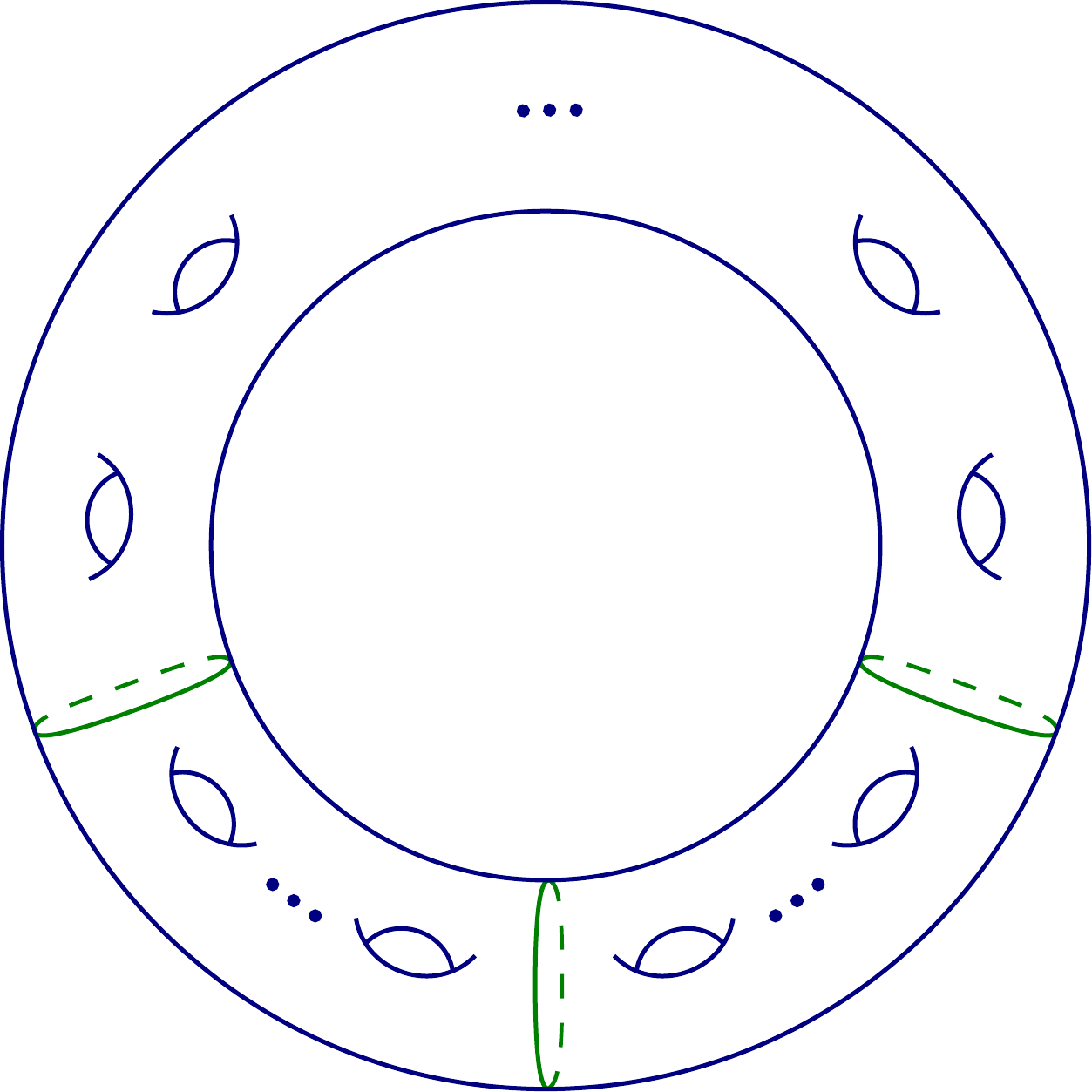}
\caption{Three disjoint homologous curves in $S_g$}
\label{fig:triple}
\end{figure}

Since the order of $\phi$ is greater than 2 the line segments $c$, $\phi(c)$, and $\phi^2(c)$ are all distinct.  If $i(c,f^2(c))=1$ we may again Lemma~\ref{wscca} to $f^2$ to conclude that the normal closure of $f^2$, hence $f$ contains the commutator subgroup.  So we may assume that $c$, $\phi(c)$, and $\phi^2(c)$ are all disjoint.  If we cut $S_g$ along the three curves corresponding to $c$, $\phi(c)$, and $\phi^2(c)$ then there is a region bordering all three curves, namely, the region of $S_g$ containing the image of the center of $P$.  On the other hand, if we have three disjoint curves in $S_g$ that are all homologous, then they must separate $S_g$ into three components, and each region abuts two of the three curves (refer to Figure~\ref{fig:triple}).  Thus, $[c]$, $[f(c)]$, and $[f^2(c)]$ are not all equal and we are done.
\end{proof}

\subsection*{Normal closure of the hyperelliptic involution} To complete the classification of normal closures of periodic elements in $\Mod(S_g)$ it remains to deal with the case of the hyperelliptic involution and also the cases of $g=1$ and $g=2$.  The following proposition already appears in the paper by Harvey and Korkmaz \cite[Theorem 4(b)]{HK}.  In the statement, let $\Psi$ denote the standard symplectic representation $\Mod(S_g) \to \Sp_{2g}(\Z)$.

\begin{figure}[h]
 \labellist
 \small\hair 2pt
 \pinlabel {$d$} [ ] at 135 27
 \pinlabel {$c$} [ ] at 135 90
 \endlabellist  
\centering
\includegraphics[scale=.4]{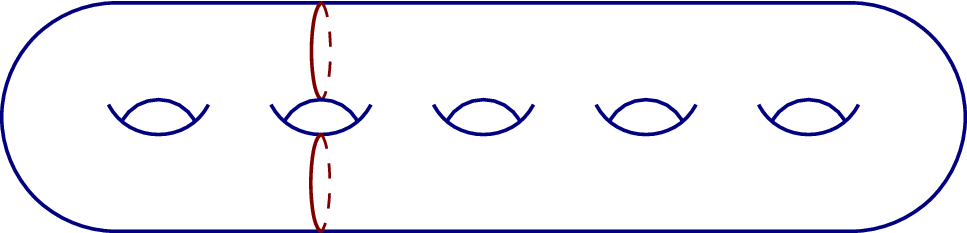}
\caption{The normal closure of $T_cT_d^{-1}$ in $\Mod(S_g)$ is $\I(S_g)$}
\label{fig:bp}
\end{figure}

\begin{proposition}
\label{etorelli}
Let $g \geq 3$ and let $\iota \in \Mod(S_g)$ be a hyperelliptic involution.  Then the normal closure of $\iota$ in $\Mod(S_g)$ is the preimage of $\{\pm I\}$ under $\Psi$.
\end{proposition}

\begin{proof}

The image of $\iota$ under $\Psi$ is $-I$.  Since the latter is central in $\Sp_{2g}(\Z)$ it follows that the normal closure of $\iota$ is contained in the preimage of $\{\pm I\}$.  It remains to show that the normal closure of $\iota$ contains the Torelli group $\I(S_g)$.  

It is a theorem of D. Johnson \cite[Theorem 2]{DJ} that $\I(S_g)$ is equal to the normal closure in $\Mod(S_g)$ of the mapping class $T_cT_d^{-1}$, where $c$ and $d$ are the curves in $S_g$ indicated in Figure~\ref{fig:bp} (this mapping class is called a bounding pair map of genus 1).  Since $\iota(c)=d$ we have that 
\[
T_cT_d^{-1} = [T_c,\iota].
\]
Thus $T_cT_d^{-1}$ is equal to a product of two conjugates of $\iota$.  In particular it is contained in the normal closure of $\iota$.  By Johnson's result, the normal closure of $\iota$ contains $\I(S_g)$ and we are done.
\end{proof}

\subsection*{Low genus cases} We now discuss the cases of $g=1$ and $g=2$.  In both cases the hyperelliptic involution $\iota$ is unique and is central in $\Mod(S_g)$.  Thus the normal closure of $\iota$ in both cases is the cyclic group of order 2 generated by $\iota$.  By Proposition~\ref{per comm} the normal closure of any other periodic element of $\Mod(S_g)$ contains the commutator subgroup, and hence the normal closure of an element $f$ is completely determined by the image of $f$ in the abelianization of $\Mod(S_g)$.

The abelianizations of $\Mod(S_1)$ and $\Mod(S_2)$ are isomorphic to $\Z/12\Z$ and $\Z/10\Z$, respectively; see \cite[Section 5.1.3]{Primer}.  The image of any Dehn twist about a nonseparating curve in either case is equal to 1.  Therefore, if $f$ is a periodic element of $\Mod(S_g)$ that is not a hyperelliptic involution and $f$ is equal to a product of $k$ Dehn twists about nonseparating curves, then the normal closure of $f$ in $\Mod(S_g)$ is equal to the preimage under the abelianization map $\Mod(S_g) \to \Z/n\Z$ of the group generated by $k$ (here $n$ is 12 when $g=1$ and 10 when $g=2$).  In particular if we consider any periodic element that is not a hyperelliptic involution then the normal closure has finite index in $\Mod(S_g)$. For a complete list of periodic elements in $\Mod(S_1)$ and $\Mod(S_2)$ and realizations of these elements as products of Dehn twists about nonseparating curves, see the paper of Hirose \cite[Theorem 3.2]{hirose2010} (we note that the second $f_{2,3}$ in Hirose's theorem should be $f_{2,4}$).  In each case the image of the periodic element is not a generator for the abelianization.  So for $g$ equal to 1 or 2, the normal closure of any periodic element is a proper subgroup of $\Mod(S_g)$.

\subsection*{Finite generating sets} We already mentioned results of Harvey, Korkmaz, McCarthy, Papadopoulos, Yoshihara, and the first author of this paper that give periodic normal generators for the mapping class group.  Much more than this, Korkmaz proved that only two conjugate elements of order $4g+2$ are needed to generate $\Mod(S_g)$.  Yoshihara proved that three conjugate elements of order 6 are required for $g \geq 10$.  Finally the first author proved that for $k \geq 6$ and $g \geq (k-1)^2+1$ only three conjugate elements of order $k$ are needed, and for $k=5$ and $g \geq 8$ only four conjugate elements are needed.  Based on these results and Theorem~\ref{main:periodic} we are led to the following question.

\begin{question}
Is there a number $N$, independent of $g$, so that if $f$ is a periodic normal generator of $\Mod(S_g)$ then $\Mod(S_g)$ is generated by $N$ conjugates of $f$?
\end{question}

We emphasize that in the question $N$ is independent of both $g$ and $f$.  

\subsection*{Power subgroups} As in the introduction, let $L=L(g)$ denote the least common multiple of the orders of the periodic elements of $\Mod(S_g)$.  We have the following corollary of Theorem~\ref{main:periodic}.

\begin{corollary}
\label{cor:funar}
Let $g \geq 3$.  Suppose that $n$ is not divisible by $L/2$.  Then the subgroup of $\Mod(S_g)$ generated by the $n$th powers of all elements is equal to the whole group $\Mod(S_g)$.
\end{corollary}

This result improves on a theorem of Funar \cite[Theorem 1.16(2)]{funar} who showed the analogous result with $L/2$ replaced by $4g+2$.  Note that $4g+2$ is a factor of $L/2$ since $L$ has 4 and $2g+1$ as factors.  Thus our result indeed recovers the theorem of Funar.  For comparison, when $g=4$ the number $L/2$ is 360 and $4g+2$ is 18.

\begin{proof}[Proof of Corollary~\ref{cor:funar}]

Since $L/2$ does not divide $n$, then for some prime $p$ we have that $L/2$ has a factor $p^j$ whereas $n$ only has a factor $p^k$ with $j>k \geq 0$.

If this $p$ is an odd prime, then $\Mod(S_g)$ contains a periodic element $f$ whose order has $p^j$ as a factor.  The element $f^n$ is nontrivial and its order has $p^{j-k}$ as a factor. Since $p$ is odd it follows that $f^n$ is not a hyperelliptic involution.  By Theorem~\ref{main:periodic}, $f^n$ is a normal generator for $\Mod(S_g)$.  Since $f^n$ is contained in the $n$th power subgroup we are done in this case.

If instead $p$ is 2, then $L/2$ has a factor $2^j$ whereas $n$ only has a factor $2^k$ with $j>k \geq 0$. But then $L$ has a factor $2^{j+1}$. This implies that $\Mod(S_g)$ contains a periodic element $f$ whose order has $2^{j+1}$ as a factor, and we have that $f^n$ is nontrivial and its order has $2^{j+1-k}$ as a factor, which is at least 4. Hence $f^n$ is nontrivial and not the hyperelliptic involution.  Again an application of Theorem~\ref{main:periodic} completes the proof.
\end{proof}

If we only consider periodic elements when analyzing power subgroups, we could at most hope to replace $L/2$ by $L$ in the corollary (as the $L$th power of every periodic element is trivial). However, for some $g$ the analysis fails with this replacement. For instance, when $g$ is a power of 2, the element of order $4g$ is the only element that is nontrivial when raised to the power $L/2$, and this power is the hyperelliptic involution.

Of course there do exist values of $n$ so that the $n$th power subgroup of $\Mod(S_g)$ is not $\Mod(S_g)$.  Indeed for any group $G$ with a proper subgroup $H$ of finite index, there is an $n$ so that the $n$th power group of $G$ is a subgroup of $H$ and hence is not equal to $G$.

\subsection*{Surfaces with marked points} We now give an extension of Theorem~\ref{main:periodic} to the case of surfaces with marked points.  As in the introduction, by a hyperelliptic involution in $\Mod(S_{g,n})$ we mean a periodic element that maps to a hyperelliptic involution in $\Mod(S_g)$ under the forgetful map $\Mod(S_{g,n}) \to \Mod(S_g)$.  

\begin{theorem}
\label{thm:punc}
For $g \geq 3$ and $n \geq 1$, the normal closure of any nontrivial periodic mapping class that is not a hyperelliptic involution contains $\PMod(S_{g,n})$.  In particular, when $n=1$, the normal closure is equal to $\Mod(S_{g,n})$.
\end{theorem}

Let $\Sigma_n$ denote the symmetric group on $n$ letters.  There is a short exact sequence
\[
1 \to \PMod(S_{g,n}) \to \Mod(S_{g,n}) \to \Sigma_n \to 1.
\]
Let $\AMod(S_{g,n})$ denote the preimage in $\Mod(S_{g,n})$ of the alternating group $A_n$; this is a subgroup of $\Mod(S_{g,n})$ of index 2.  Because the only normal subgroups of $\Sigma_n$ are the trivial group, $A_n$, $\Sigma_n$, and (when $n=4$) the Klein four group, it follows from Theorem~\ref{thm:punc} that the only possibilities for the normal closure of a nontrivial, non-hyperelliptic, periodic mapping class $f \in \Mod(S_{g,n})$ with $g \geq 3$ and $n \neq 4$ are $\PMod(S_{g,n})$, $\AMod(S_{g,n})$, and $\Mod(S_{g,n})$, according to whether the image of $f$ in $\Sigma_n$ is trivial, nontrivial even, or odd.

\begin{proof}[Proof of Theorem~\ref{thm:punc}]

Let $f$ be a nontrivial periodic mapping class in $\Mod(S_{g,n})$ that is not a hyperelliptic involution.  Since $\PMod(S_{g,n})$ is perfect, it suffices to show that the normal closure of $f$ contains the commutator subgroup of $\PMod(S_{g,n})$.  As in the proof of Proposition~\ref{per comm}, we accomplish this by establishing well-suited curve criteria for $\Mod(S_{g,n})$ and then verifying that $f$ satisfies one of these criteria.

The statements and proofs of the well-suited curve criteria established in Section~\ref{sec:criterion} carry over to the case of surfaces with marked points, with $S_g$ replaced by $S_{g,n}$, with $\Mod(S_g)$ replaced by $\PMod(S_{g,n})$, and with the following caveats.  In the statement of Lemma~\ref{wsccb}, we should interpret $[c] \neq [f(c)]$ to mean that the images of $c$ and $f(c)$ in $S_g$ are not homologous mod 2.  Similarly, in the statement of Lemma~\ref{wsccsep}, we should interpret the assumption that $d$ is separating to mean that the image of $d$ in $S_g$ is an essential separating curve.  A key point is that, like in the case $n=0$, the group $\PMod(S_{g,n})$ has a generating set consisting of Dehn twists about nonseparating curves, where the pairwise intersection numbers of the curves is at most 1; see \cite[Section 4.4.4]{Primer}.

As in the proof of Theorem~\ref{main:periodic}, we take $\phi$ to be a standard representative of $f$.  As in the case $n=0$, this is a homeomorphism of $S_{g,n}$ that represents $f$ and has the same order.  An essential point is that $\phi$ may be regarded as a homeomorphism of $S_g$.  Because of this, we may apply the same classification results and theorems used in the proof of Proposition~\ref{per comm}.    As in the proof of Theorem~\ref{main:periodic} we consider three cases, according to whether the action of $\langle \phi \rangle$ is free, the action is not free and the order of $\phi$ is 2, or the action is not free and the order of $\phi$ is greater than 2.

In the first two cases, the proof is the same; we simply choose the curve $c$ to avoid the marked points.  In the third case, we may need to make a small perturbation of the line segment $c$ in order to avoid the marked points; the proof then applies as stated. 
\end{proof}


\section{Application: Long's question and Ivanov's conjecture}
\label{sec:long}

In this section we construct, for each $g \geq 1$, a pseudo-Anosov mapping class that normally generates $\Mod(S_g)$ (for $g=1$ pseudo-Anosov mapping classes are usually called Anosov, but we will not make this distinction).  As discussed in the introduction, it follows from our Theorem~\ref{main:pa} that such mapping classes exist when $g \geq 3$.  The examples in this section are simple and explicit and only use the special cases of the well-suited curve criterion given in Section~\ref{sec:criterion}, whereas the proof of Theorem~\ref{main:pa} is more involved and requires the more general version of the well-suited curve criterion given in Section~\ref{sec:general}.  

We begin with a very simple family of pseudo-Anosov mapping classes that normally generate, then we give examples that have large stretch factor, examples that have small stretch factor, and examples that have large translation length on the curve graph.  Finally we give examples with large stretch factor with normal closure equal to any given congruence subgroup.  As described in the introduction, the examples in this section answer Long's question in the affirmative and resolve Ivanov's conjecture in the negative.

\begin{figure}[h]
 \labellist
 \small\hair 2pt
 \pinlabel {$a_1$} [ ] at 83 112
 \pinlabel {$a_2$} [ ] at 193 112
 \pinlabel {$a_3$} [ ] at 303 112
 \pinlabel {$b_1$} [ ] at 28 53
 \pinlabel {$b_2$} [ ] at 138 53
 \pinlabel {$b_3$} [ ] at 248 53
 \pinlabel {$b_4$} [ ] at 358 53
 \endlabellist  
\centering
\includegraphics[width=175px]{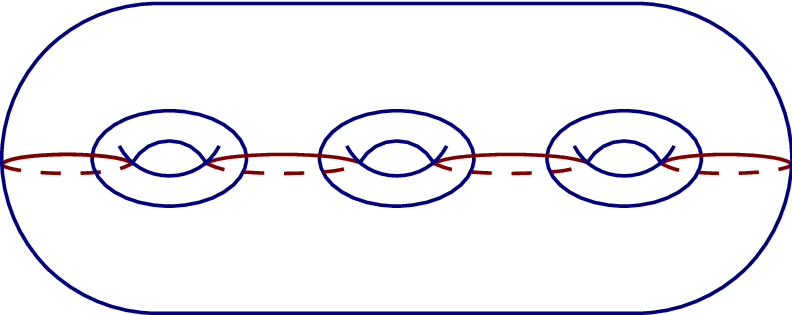}
\caption{The $a_i$ and $b_i$ used in the definition of $T_A$ and $T_B$}
\label{fig:thurston}
\end{figure}
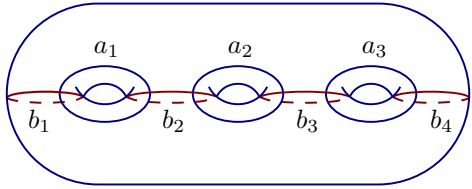

\subsection*{First examples} For each $g \geq 1$, let $T_A  = T_{a_1} \cdots T_{a_g}$ and $T_B=T_{b_1} \cdots T_{b_{g+1}}$ where the $a_i$ and $b_i$ are the curves in $S_g$ indicated in Figure~\ref{fig:thurston}.  The figure shows the case $g=3$ but there is an obvious generalization for all other $g \geq 1$; when $g=1$ the curves $b_1$ and $b_2$ are parallel.  

Consider the mapping class $f = T_A^{-1}T_B$ (note $f$ depends on $g$).   The following proposition answers in the affirmative the question of Long from the introduction.

\begin{proposition}
\label{prop:long}
For each $g \geq 1$ the mapping class $f$ is pseudo-Anosov and it normally generates $\Mod(S_g)$. 
\end{proposition}

\begin{proof}

There are well-known constructions of pseudo-Anosov mapping classes due to Thurston and to Penner where certain products of Dehn twists are shown to be pseudo-Anosov; see \cite[Theorems 14.1 and 14.4]{Primer}.  The product $f = T_A^{-1}T_B$ is both an example of the Thurston construction and an example of the Penner construction.  In particular, $f$ is pseudo-Anosov for all $g \geq 1$.

Consider the action of $f$ on $b_1$.  Since $i(b_1,b_i)$ is equal to 0 for all $i$ and $i(b_1,a_i)$ is equal to 0 for $i > 1$ we have that $f(b_1) = T_{a_1}^{-1}(b_1)$.  It follows that $i(f(b_1),b_1)$ is equal to 1.  Thus by Lemma~\ref{wscca} the normal closure of $f$ contains the commutator subgroup of $\Mod(S_g)$.  For $g \geq 3$ it follows that $f$ normally generates $\Mod(S_g)$.  For $g$ equal to 1 or 2, it is enough now to observe that since the sum of the exponents of the Dehn twists in the definition of $f$ is 1.  As such the image of $f$ in the abelianization of $\Mod(S_g)$ is 1, and hence a generator for the abelianization.  Thus in these cases $f$ is a normal generator as well.
\end{proof}

\subsection*{Examples with large stretch factor} There is great flexibility in the construction of $f$.  Indeed, we can alter any of the exponents, except on $T_{a_1}$, and we obtain another normal generator for $\Mod(S_g)$ when $g \geq 3$.  If we also preserve the condition that the exponent sum is relatively prime to 12 or 10, we obtain a normal generator for $\Mod(S_1)$ or $\Mod(S_2)$, respectively.  We take advantage of this flexibility to prove the following.

\begin{proposition}
\label{prop:arbi}
For each $g \geq 1$ there are pseudo-Anosov mapping classes with arbitrarily large stretch factors that normally generate $\Mod(S_g)$. 
\end{proposition}

\begin{proof}

Fix some $g \geq 1$ and consider again the mapping classes $T_A$ and $T_B$ as above.  Much more than proving $T_A^{-1}T_B$ is pseudo-Anosov, Thurston proved that there is a positive number $\mu$ and a homomorphism 
\[ 
\langle T_A , T_B \rangle \to \PSL_2(\R)
\]
with
\[ 
T_A \mapsto \begin{bmatrix}  1 & -\mu^{1/2} \\ 0 & 1\end{bmatrix} \quad \text{and} \quad T_B \mapsto \begin{bmatrix}  1 & 0 \\ \mu^{1/2} & 1\end{bmatrix}
\]
and so that $f \in \langle T_A , T_B \rangle$ is pseudo-Anosov if and only if its image in $\PSL_2(\R)$ (or, rather, a representative in $\SL_2(\R)$) has an eigenvalue $\lambda > 1$; moreover in this case $\lambda$ is the stretch factor of $f$ (the number $\mu$ is the Perron--Frobenius eigenvalue of $NN^T$, where $N$ is the intersection matrix between the $a_i$ and $b_i$; see \cite{Primer}).

Consider then the mapping class $f_n = T_A^{-1}T_B^n$.  Its image in $\PSL_2(\R)$ is
\[
\begin{bmatrix}  1 & \mu^{1/2} \\ 0 & 1\end{bmatrix}\begin{bmatrix}  1 & 0 \\ n\mu^{1/2} & 1\end{bmatrix} = \begin{bmatrix}  1+n\mu & \mu^{1/2} \\ n\mu^{1/2} & 1\end{bmatrix}
\]
The trace of this matrix is larger than 2, and so it has two real eigenvalues.  The larger of these eigenvalues is strictly larger than the trace, which is $2+n\mu$.  In particular, each mapping class $T_A^{-1}T_B^n$ is pseudo-Anosov, and the stretch factors tend to infinity as $n$ tends to infinity.  As in the proof of Proposition~\ref{prop:long}, we have $i(f_n(b_1),b_1)=1$ and so by Lemma~\ref{wscca}  the normal closure of each of these mapping classes contains the commutator subgroup of $\Mod(S_g)$.  This completes the proof for $g \geq 3$.  

For $g=2$ the image of $f_n$ in the abelianization of $\Mod(S_2)$ is $3n-2$.  For each $n$ equal to 1, 3, 5, or 7 mod 10 the image of $f_n$ is therefore relatively prime to 10 and so the proposition is proven for $g=2$.  For $g=1$ the image of $f_n$ is $2n-1$, which is relatively prime to 12 for $n$ equal to 0, 1, 3, or 4 mod 6.
\end{proof}

\begin{figure}[h]
 \labellist
 \small\hair 2pt
 \pinlabel {$a$} [ ] at 177 294
 \pinlabel {$b$} [ ] at 195 225
 \pinlabel {$c$} [ ] at 125 240
 \pinlabel {$d$} [ ] at 160 80
 \pinlabel {$\rho$} [ ] at 65 280
 \endlabellist  
\centering
\includegraphics[width=150px]{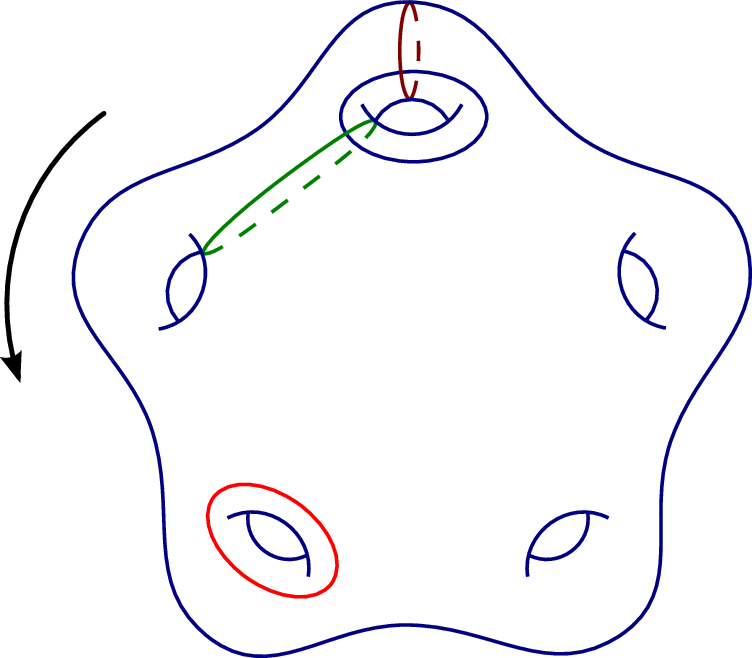}
\caption{The curves and the rotation used in Penner's construction of pseudo-Anosov mapping classes with small stretch factor}
\label{fig:penner}
\end{figure}
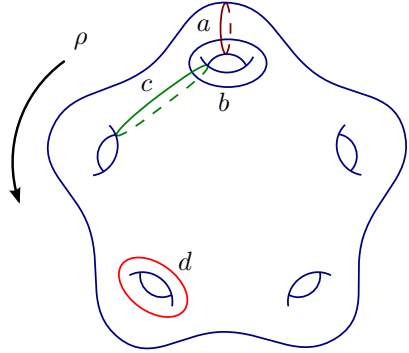

\subsection*{Examples with small stretch factor} There are many other explicit examples of pseudo-Anosov mapping classes that satisfy the special cases of the well-suited curve criterion given in Section~\ref{sec:criterion} and hence normally generate $\Mod(S_g)$.  Here we discuss one famous family of examples due to Penner.

Consider the curves $a$, $b$, and $c$ in $S_g$ indicated in Figure~\ref{fig:penner} and consider the order $g$ rotation $\rho$ of $S_g$ indicated in the same figure.  Penner proved that the product $\rho T_c T_b^{-1} T_a$ is pseudo-Anosov (to prove this he notes that the $g$th power is an example of his Dehn twist construction for pseudo-Anosov mapping classes).  By inspection we observe that Penner's mapping class, together with the curve $d$ in Figure~\ref{fig:penner}, satisfies Lemma~\ref{wsccb}.  We thus have the following proposition.

\begin{proposition}
\label{prop:penner}
For each $g \geq 3$, Penner's mapping class $\rho T_c T_b^{-1} T_a$ is a normal generator for $\Mod(S_g)$.
\end{proposition}

For $g=2$ the mapping class $\rho$ maps to 5 in $\Z/10\Z$ and so Penner's mapping class maps to 6.  Thus the image of Penner's mapping class has index 2 in $\Z/10\Z$.  On the other hand, $\rho T_c^2 T_b^{-2} T_a^2$ is pseudo-Anosov and normally generates.
  
Penner proved that the stretch factors of his mapping classes are bounded above by $11^{1/g}$ for all $g \geq 2$.  Hence they form a sequence of pseudo-Anosov mapping classes whose stretch factors tend to 1 as $g$ tends to infinity.  As such, these examples served as the initial inspiration for our Theorem~\ref{main:pa}, which says that all pseudo-Anosov mapping classes with sufficiently small stretch factor are normal generators.

\subsection*{Smallest stretch factors} For $g \geq 3$ it is an open problem to determine which mapping classes realize the smallest stretch factor.  Nevertheless, we do know in these cases that the smallest stretch factors are less than $\sqrt{2}$ and so by Theorem~\ref{main:pa} the minimizers are normal generators.

For $g$ equal to 1 or 2, we know precisely which mapping classes realize the smallest stretch factor in $\Mod(S_g)$.  For $g = 1$ the minimum is given uniquely by the conjugacy class of $f_1 = T_{a_1}T_{b_1}^{-1}$ where $a_1$ and $b_1$ are as shown in Figure~\ref{fig:thurston}.  The stretch factor of $f_1$ is $\phi^2 = (3 + \sqrt{5})/2$.  The mapping class $f_1$ is not a normal generator since the image of $f_1$ in $\Z/12\Z$, the abelianization of $\Mod(S_1)$, is 0.  On the other hand, the normal closure of $f_1$ does contain the commutator subgroup of $\Mod(S_1)$.  This is because $i(f_1(b_1),b_1)=1$ and so $f_1$ satisfies Lemma~\ref{wscca}.

For $g=2$, Cho and Ham \cite{CH} proved that the minimum stretch factor is the largest real root of $x^4 - x^3-x^2-x+1$, which is approximately 1.72208.  Lanneau and Thiffeault \cite{LT} gave an independent proof and also classified the conjugacy classes of all of the minimizers and gave explicit representatives of these conjugacy classes in terms of Dehn twists.  Again using the labels in Figure~\ref{fig:thurston}, these are:
\[
f_2 = T_{a_2}T_{b_3}^{-1}T_{a_1}T_{b_1}^2T_{b_2} \ \ \text{and} \ \ f_2' = T_{a_2}^{-1}T_{b_2}^{-1}T_{b_3}^{-1}T_{a_1}T_{b_1}^2.
\]
The images of these mapping classes in the abelianization $\Z/10\Z$ are 4 and 0, respectively, and so neither is a normal generator.  As in the $g=1$ case the normal closures of both $f_2$ and $f_2'$ contain the commutator subgroup of $\Mod(S_2)$ since $i(f_2(b_3),b_3)=1$ and $i(f_2'(b_3),b_3)=1$.

\subsection*{Examples with large translation length} In this section, we have already given examples of pseudo-Anosov normal generators for $\Mod(S_g)$ with either small or large stretch factor.  All of these examples have translation length in $\C(S_g)$ at most 2 since they each satisfy one of the well-suited curve criteria given in Lemmas~\ref{wscca} and~\ref{wsccb}.  It follows that the asymptotic translation lengths are also bounded above by 2.

We will now prove Theorem~\ref{translation}, which states that for each $g \geq 3$ there are pseudo-Anosov mapping classes with the property that all of their odd powers are normal generators.  Since the asymptotic translation length on the curve graph is multiplicative, and since the asymptotic translation length of every pseudo-Anosov mapping class is positive, it follows that there are pseudo-Anosov normal generators with arbitrarily large asymptotic translation lengths.  Because large translation length implies large stretch factor, Theorem~\ref{translation} implies Proposition~\ref{prop:arbi} in the cases when $g \geq 3$.

\begin{proof}[Proof of Theorem~\ref{translation}]

Fix $g \geq 3$.   Let $D_{2g}$ denote the dihedral group of order $2g$.  There is a standard action of $D_{2g}$ on $S_g$ by orientation-preserving homeomorphisms.  In this way, we identify $D_{2g}$ as a subgroup of $\Homeo^+(S_g)$.  We denote the quotient $S_g/D_{2g}$ by $X$.  As an orbifold, $X$ is a sphere with five orbifold points.  

Consider any pseudo-Anosov homeomorphism of $X$, thought of as a sphere with five marked points.  Up to taking a power, we may assume that this homeomorphism lifts to a pseudo-Anosov homeomorphism $\psi$ of $S_g$; see \cite[Section 14.1.1]{Primer}.

The homeomorphism $\psi$ lies in the normalizer of $D_{2g}$, and so we may identify $\psi$ with an element of the automorphism group of $D_{2g}$; it acts by conjugation.  Let $r$ be the element of $D_{2g}$ corresponding to rotation by $2\pi/g$ under the standard action of $D_{2g}$ on the $g$-gon.  Since $r$ and $r^{-1}$ are the only elements of $D_{2g}$ that are conjugate in $\Homeo^+(S_g)$, it must be that the automorphism of $D_{2g}$ induced by $\psi$ maps $r$ to either $r$ or $r^{-1}$.  Let $k \in D_{2g}$ correspond to some reflection of the $g$-gon.  Up to replacing $\psi$ by $k\psi$ we may assume that the automorphism of $D_{2g}$ induced by $\psi$ maps $r$ to $r^{-1}$.  Since $k$ fixes the invariant foliations for $\psi$ and does not change the stretch factor, the new $\psi$ is a pseudo-Anosov element $f$ of $\Mod(S_g)$.  If we identify $r$ as an element of $\Mod(S_g)$ then $frf^{-1}=r^{-1}$.  By construction $f^nrf^{-n} = r^{-1}$ for all odd $n$.

Consider the mapping class $[r,f^n]$ for $n$ odd.  We have
\[ [r,f^n]=r(f^nrf^{-n})^{-1} = r^2. \]
Thus $r^2$ lies in the normal closure of $f^n$ in $\Mod(S_g)$.  By the well-suited curve criterion of Lemma~\ref{wsccb} (and using the assumption $g \geq 3$) the mapping class $r^2$, hence $f^n$, is a normal generator for $\Mod(S_g)$.
\end{proof}

\subsection*{Torelli groups and congruence subgroups} Having found an abundance of pseudo-Anosov normal generators for the mapping class group, we now consider the question of which proper normal subgroups of $\Mod(S_g)$ arise as the normal closure of a single pseudo-Anosov mapping class.

Specifically, our next goal is to prove Theorem~\ref{level}, which states that all level $m$ congruence subgroups $\Mod(S_g)[m]$ arise as the normal closure of a pseudo-Anosov mapping class, and moreover that this mapping class can be chosen to have arbitrarily large stretch factor.  The theorem covers the Torelli group as the case where $m=0$.  The construction in the proof is a variant of a construction of Leininger and the second author \cite[Proof of Proposition 5.1]{LM}.

The proof of Theorem~\ref{level} requires the following technical lemma, whose statement and proof are well known to experts, but are not easily found in the literature.  

\begin{lemma}
\label{k}
Let $A = \{a_1,\dots,a_m\}$ and $B = \{b_1,\dots,b_n\}$ be two multicurves that fill $S_g$.  The stretch factors of the mapping classes
\[
f_k = \big( T_{a_1}^k \cdots T_{a_m} \big)
\big( T_{b_1} \cdots T_{b_n} \big)^{-1}
\]
tend to infinity as $k$ does.
\end{lemma}

Since the $T_{a_i}$ all commute with each other, since the $T_{b_i}$ all commute with each other, and since the stretch factor of a mapping class is the same as that of its inverse, the lemma holds equally well if we attach the exponent $k$ to any $T_{a_i}$ or attach the exponent $-k$ to any $T_{b_i}$.

\begin{proof}[Proof of Lemma~\ref{k}]

Consider the intersection graph $\Gamma(A,B)$ with vertices $A \cup B$ and $i(a_i,b_j)$ edges for each pair $\{a_i,b_j\}$ (this is really a multigraph).  Since $A$ and $B$ fill $S_g$, this graph is connected.  Let $D$ be half the maximum path-length distance between two $A$-vertices of $\Gamma(A,B)$.

Let $A_k$ denote the multicurve obtained from $A$ by replacing $a_1$ with $k$ parallel copies of $a_1$, denoted $a_{1,1},\dots,a_{1,k}$.  Let $T_{A_k}$ and $T_B$ be the products of the Dehn twists about the curves in $A_k$ and $B$, respectively.  We thus have
\[
f_k = T_{A_k} T_B^{-1}.
\]
Let $\Gamma(A_k,B)$ be the intersection graph for $A_k$ and $B$.  This graph has $m+n+k-1$ vertices and is the blowup of $\Gamma(A,B)$ obtained by replacing the vertex $a_1$ with $k$ vertices $a_{1,1},\dots,a_{1,k}$.

Let $N$ be the $(n+k-1) \times m$ matrix whose rows correspond to the curves of $A_k$, whose columns correspond to the curves of $B$, and whose entries are given by the corresponding intersection numbers.  The matrix $NN^T$ is a square matrix with rows and columns corresponding to the curves of $A_k$.  Each entry is the number of paths of length 2 between the corresponding vertices of $\Gamma(A_k,B)$.  Similarly, the entries of any power $(NN^T)^\ell$ are the numbers of paths of length $2\ell$ between corresponding vertices of $\Gamma(A_k,B)$.

Similar to the proof of Proposition~\ref{prop:arbi}, the stretch factor of $f_k$ is bounded below by $2+\mu$, where $\mu$ is the Perron--Frobenius eigenvalue of $NN^T$.  Thus it suffices to show that the latter tends to infinity with $k$.

Consider the matrix $\left(NN^T\right)^{D+1}$.  Since $D$ does not depend on $k$ and since the Perron--Frobenius eigenvalue of $\left(NN^T\right)^{D+1}$ is the $(D+1)$st power of the Perron--Frobenius eigenvalue of $NN^T$ it further suffices to show that the Perron--Frobenius eigenvalue of $\left(NN^T\right)^{D+1}$ tends to infinity with $k$.

The Perron--Frobenius eigenvalue of a Perron--Frobenius matrix is bounded below by the minimum row sum of the matrix.  Fix a vertex $v$ of $\Gamma(A_k,B)$ corresponding to a component of $A_k$.  The corresponding row sum of $NN^T$ is the number of paths of length $2D+2$ in $\Gamma(A_k,B)$ starting at $v$.

By the definition of $D$ and the definition of the blowup $\Gamma(A_k,B)$, there is a path from $v$ to $a_{1,1}$ with length $L \leq 2D$.  There are at least $k$ ways to extend this path to a path of length $L+2 \leq 2D+2$, since $a_{1,1}$ is connected to some $b_j$ and this $b_j$ is connected to each of the blown-up vertices $a_{1,1},\dots,a_{1,k}$.  Thus the minimum row sum of $\left(NN^T\right)^{D+1}$ is at least $k$, and the lemma follows.
\end{proof}

The proof of Lemma~\ref{k} can be modified to prove a much more general result, namely, that for $A = \{a_1,\dots,a_m\}$ and $B = \{b_1,\dots,b_n\}$ as in the lemma,  the stretch factors of the mapping classes
\[
\Big( T_{a_1}^{p_1} \cdots T_{a_m}^{p_m} \Big)
\Big( T_{b_1}^{q_1} \cdots T_{b_n}^{q_n} \Big)^{-1}
\]
tend to infinity as the maximum of $\{p_1,\dots,p_m,q_1,\dots,q_n\}$ does.  A special case of this fact was used (and proved) in the proof of Proposition~\ref{prop:arbi} above.  The proof of the more general fact is essentially the same, but is notationally unwieldy.  The statement of Lemma~\ref{k} suffices for the proof of Theorem~\ref{level}.

\begin{figure}
    \centering
        \includegraphics[width=.85\textwidth]{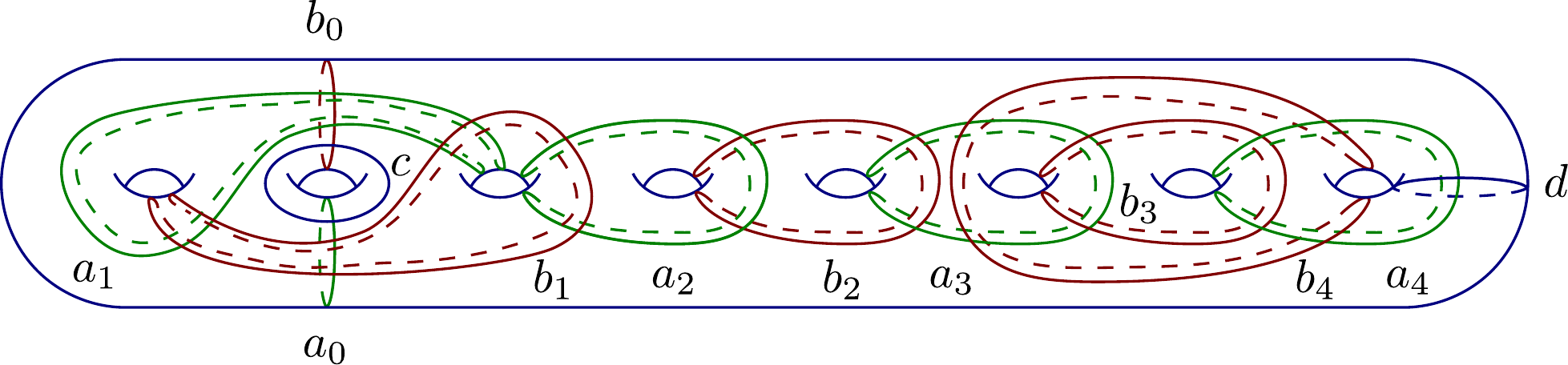}
        \caption{The curves used in the proof of Theorem~\ref{level} in the case $g=8$}
    \label{fig:lm}
\end{figure}

\begin{proof}[Proof of Theorem~\ref{level}]

As in the statement, let $g \geq 3$ be fixed.  Let
\[
a_0,\dots,a_{\lceil g/2 \rceil},b_0,\dots,b_{\lfloor g/2 \rfloor},d
\]
be the curves in $S_g$ as shown in Figure~\ref{fig:lm} for the case $g=8$.  The curve $d$ is disjoint from each $a_i$ when $g$ is odd and it is disjoint from each $b_i$ when $g$ is even.  Consider the mapping class
\[
f_{m,k} = \Big(T_{b_0}^{-1}\left(T_d^m\right)^{(-1)^{g+1}}T_{a_0}\Big)\Big(T_{a_1}^k \cdots T_{a_{\lceil g/2 \rceil}}\Big)\Big(T_{b_1} \cdots T_{b_{\lfloor g/2 \rfloor}}\Big)^{-1}.
\]
This mapping class is pseudo-Anosov.  Indeed when $g$ is odd it is conjugate by $T_{b_0}^{-1}$ to a pseudo-Anosov mapping class arising from the Thurston construction and when $g$ is even it is conjugate by $T_{b_0}^{-1}T_d^{-m}$ to a pseudo-Anosov mapping class arising from the Thurston construction.  

For each $m$ and $k$, there is a curve $c$ so that $c$ and $f_{m,k}(c)$ form a bounding pair of genus 1, an example of which is shown in Figure~\ref{fig:bp}.  The curve $c$ is shown in Figure~\ref{fig:lm}.  As in the proof of Proposition~\ref{etorelli} the normal closure of $T_cT_{f_{m,k}(c)}^{-1}$ in $\Mod(S_g)$ is $\I(S_g)$.  Since $T_cT_{f_{m,k}(c)}^{-1} = [T_c,f_{m,k}]$ we see that the normal closure of $f_{m,k}$ in $\Mod(S_g)$ contains $\I(S_g)$.

The normal closure of $f_{m,k}$ is thus completely determined by its image under the symplectic representation $\Psi : \Mod(S_g) \to \Sp_{2g}(\Z)$.  Since each $T_{a_i}$ and each $T_{b_i}$ lies in $\I(S_g)$ we have that $\Psi(f_{m,k})$ is equal to $\Psi(T_d^m)$.  The latter is the $m$th power of a transvection about a primitive element of $\Z^{2g}$.   Thus by work of Mennicke \cite[Section 10]{Mennicke} its normal closure in $\Sp_{2g}(\Z)$ is the level $m$ congruence subgroup $\Sp_{2g}(\Z)[m]$, which is the kernel of the reduction homomorphism $\Sp_{2g}(\Z) \to \Sp_{2g}(\Z/m\Z)$ (it is essential here that $g > 1$).  The preimage under $\Psi$ of $\Sp_{2g}(\Z)[m]$ is the level $m$ congruence subgroup $\Mod(S_g)[m]$, as desired.

By Lemma~\ref{k} the stretch factor of $f_{m,k}$ tends to infinity as $k$ does (in order to apply Lemma~\ref{k} directly we should replace $T_d^m$ with $T_{d_1} \cdots T_{d_m}$ where each $d_i$ is parallel to $d$).  This completes the proof.
\end{proof}


\section{Application: linear groups}
\label{sec:other}

In this section we use the results we have proven so far about normal generators for mapping class groups to produce many examples of normal generators for certain linear groups. Using the standard surjective representation $\Psi : \Mod(S_g) \to \Sp_{2g}(\Z)$, we can apply our results about mapping class groups to show that certain elements of $\Sp_{2g}(\Z)$ are normal generators for $\Sp_{2g}(\Z)$.

The normal closure of $\Sp_{2g}(\Z)$ in $\SL_{2g}(\Z)$ is $\SL_{2g}(\Z)$.  Thus, any normal generator for $\Sp_{2g}(\Z)$ can also be interpreted as a normal generator for $\SL_{2g}(\Z)$.

To begin, we obtain many examples of normal generators for $\Sp_{2g}(\Z)$ by a direct application of Theorem~\ref{main:periodic}.  That is, if an element of $\Sp_{2g}(\Z)$ is the image of a nontrivial periodic element that is not a hyperelliptic involution then it is a normal generator for $\Sp_{2g}(\Z)$.

What is more, there are periodic elements of $\Sp_{2g}(\Z)$ that are not images of periodic elements but are images of elements that satisfy the well-suited curve criterion.  For instance consider the block matrix
\[
M = \begin{bmatrix}  A & 0 \\ 0 & I_{2g-2}\end{bmatrix}
\]
where $A$ is the $2 \times 2$ matrix 
\[
\begin{bmatrix}  0 & 1 \\ -1 & 0 \end{bmatrix}
\]
and $I_n$ denotes the $n \times n$ identity matrix.  Here and throughout all matrices in $\Sp_{2g}(\Z)$ are written with respect to a standard symplectic basis $(x_1,y_1,\dots,x_g,y_g)$ for $\Z^{2g}$.

The periodic matrix $M$ is the image of the handle rotation $r$ indicated in the left-hand side of Figure~\ref{fig:swap}.  The handle rotation $r$ has infinite order, but it satisfies Lemma~\ref{wscca}.  Thus $r$ is a normal generator for $\Mod(S_g)$ and so $M$ is a normal generator for $\Sp_{2g}(\Z)$.

Similarly, consider the block matrix
\[
N = \begin{bmatrix}  B & 0 \\ 0 & I_{2g-4}\end{bmatrix}
\]
where $B$ is the $4 \times 4$ block matrix 
\[
\begin{bmatrix}  0 & I_2 \\ I_2 & 0 \end{bmatrix}
\]
The matrix $N$ is the image of the handle swap $s$ indicated in the right-hand side of Figure~\ref{fig:swap}.  Again $s$ has infinite order, but since it satisfies Lemma~\ref{wsccb} it is a normal generator for $\Mod(S_g)$, implying that $N$ is a normal generator for $\Sp_{2g}(\Z)$.

We can also use our pseudo-Anosov examples from Section~\ref{sec:long} to find examples of irreducible (in fact, primitive) normal generators for $\Sp_{2g}(\Z)$.  Indeed, if we take the curves $(b_1,a_1,\dots,b_g,a_g)$ in Figure~\ref{fig:thurston} as a basis for $H_1(S_g;\Z)$ (with appropriate orientations) then the mapping class
\[
\left(T_{a_1} \cdots T_{a_g}\right)^{-1} \left(T_{b_1} \cdots T_{b_g}\right)
\]
has primitive image in $\Sp_{2g}(\Z)$.  This image is a normal generator since the mapping class is. 

One might hope to apply Theorem~\ref{main:periodic} and the surjectivity of the map $\Mod(S_g) \to \Sp_{2g}(\Z)$ in order to show that every nontrivial, noncentral finite-order element of $\Sp_{2g}(\Z)$ is a normal generator for $\Sp_{2g}(\Z)$.  This fails, however, because there are finite-order elements of $\Sp_{2g}(\Z)$ so that each element of the preimage in $\Mod(S_g)$ has infinite order.  For instance consider the elements of $\Sp_{2g}(\Z)$ given by the block matrices
\[
M_k = \begin{bmatrix}  -I_{2k} & 0 \\ 0 & I_{2g-2k}\end{bmatrix}
\]
where $I_n$ is the $n \times n$ identity matrix.  For $0 < k < g$ the matrix $M_k$ does not have periodic preimage in $\Mod(S_g)$ (for $k=0$ the identity lies in the preimage and for $k=g$ the hyperelliptic involution does).  Moreover the normal closure of $M_k$ cannot be the entire group $\Sp_{2g}(\Z)$ since $M_k$ lies in the level 2 congruence subgroup, which is a proper normal subgroup of $\Sp_{2g}(\Z)$.  We suspect these are the only counterexamples.

\begin{figure}
    \centering
        \includegraphics[scale=.4]{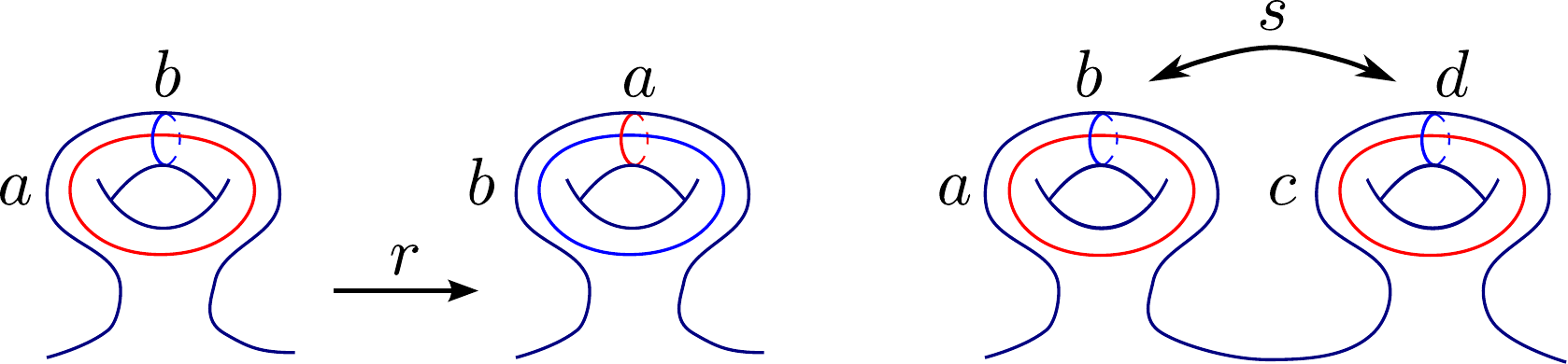}
        \caption{A handle rotation and a handle swap}
    \label{fig:swap}
\end{figure}

Putman has pointed out to us another approach to showing that elements of $\Sp_{2g}(\Z)$ are normal generators.  By the Margulis normal subgroup theorem the normal closure of a nontrivial, noncentral element $M$ of $\Sp_{2g}(\Z)$ has finite index in $\Sp_{2g}(\Z)$.  Since $\Sp_{2g}(\Z)$ enjoys the congruence subgroup property, this normal closure contains a level $m$ congruence subgroup.  Therefore, the problem of showing that $M$ is a normal generator reduces to the problem of showing that the image of $M$ in the finite group $\Sp_{2g}(\Z/m\Z)$ is a normal generator.


\section{Stretch factors and intersection numbers}
\label{sec:int}

The goal of the remaining five sections is to prove Theorem~\ref{main:pa}, which states that any pseudo-Anosov mapping class with stretch factor less than $\sqrt{2}$ is a normal generator.  In this section we take the first step by relating the hypothesis about small stretch factor to the existence of a curve that has small intersection number with its image.   After making this connection, we outline the rest of the proof of Theorem~\ref{main:pa}.  We begin with some preliminaries about singular Euclidean metrics for pseudo-Anosov mapping classes.

\subsection*{Metrics for pseudo-Anosov mapping classes} For a pseudo-Anosov mapping class $f$ with stretch factor $\lambda$ there is a representative $\phi$ and two transverse measured foliations $\F_u$ and $\F_s$ on $S_g$ with
\[
\phi(\F_u) = \lambda \F_u \quad \text{and} \quad \phi(\F_s) = \lambda^{-1} \F_s
\]
(see \cite[Section 13.2.3]{Primer}).  The measured foliations $\F_u$ and $\F_s$ determine a singular Euclidean metric on $S_g$.  If we insist that this metric has unit area then there is a one parameter family of choices for the pair $(\F_u,\F_s)$ and any choice suffices for our purposes; we refer to any one of these as an \emph{$f$-metric} on $S_g$.  An $f$-metric is also an $f^k$-metric for all nonzero $k$.

A curve in $S_g$ has a well-defined length in a given $f$-metric (by a curve we again mean a homotopy class of curves) \cite[Proposition 5.7]{FLP}.  The set of lengths is discrete in $\R$ and so there is a shortest such length, although the curve of shortest length may not be unique.

\subsection*{Stretch factors versus intersection numbers} The following fundamental tool was introduced by Farb, Leininger, and the second author of this paper \cite[Proposition 2.6]{FLM}.

\begin{proposition}
\label{flm}
Let $g \geq 1$ and let $f \in \Mod(S_g)$ be pseudo-Anosov.  Let $c$ be a curve in $S_g$ that is shortest with respect to some $f$-metric.  If the stretch factor of $f$ is less than or equal to $\sqrt[k]{n/2}$ then $i(c,f^k(c)) < n$.
\end{proposition}

In the paper by Farb, Leininger, and the second author, the statement of Proposition~\ref{flm} is less specific, saying only that if the stretch factor of $f$ is less than $n/2$ then there exists a curve $c$ with $i(c,f(c)) <  n$.  However, the proof there, plus the facts that the stretch factor of $f^k$ is the $k$th power of the stretch factor for $f$ and that an $f$-metric is also an $f^k$-metric, proves the statement given here, and we will need this more general statement.  Specifically, we will apply the proposition in the cases $1 \leq k \leq 3$.

\subsection*{Outline of the proof of Theorem~\ref{main:pa}} We now explain how Proposition~\ref{flm} will be used in the proof of Theorem~\ref{main:pa}, and we give an overview of the remaining four sections of the paper.

Let $f$ be a pseudo-Anosov mapping class.  For us, the first and most important consequence of Proposition~\ref{flm} comes from the $k=1$ case: if the stretch factor of $f$ is less than $\sqrt{2}$ then there is a curve $c$ with
\[
i(c,f(c)) \leq 2
\]
(here we are using the fact that $\sqrt{2} < 3/2$).  If $c$ is separating, then we may apply Lemma~\ref{wsccsep} to conclude that $f$ is a normal generator.  Thus we may way assume that $c$ is nonseparating.  The proof of Theorem~\ref{main:pa} then proceeds by considering three cases, as in the following table.
\medskip

\renewcommand{\arraystretch}{1.35}
\begin{center}
 \begin{tabular}{|c|c|c|c|c|}
   \hline
   Case & $i(c,f(c))$ & $[c]=[f(c)] \mod 2$ \\ \hline 
   1 & 0 or 1 & \xmark \\ \hline
   2 & 2 & \xmark \\ \hline
   3 & 0 or 2 & \cmark \\ \hline
 \end{tabular}
\end{center}

\medskip

\noindent (Of course mod 2 homologous curve cannot have intersection number 1.)  For Case 1, our first well-suited curve criteria, Lemmas~\ref{wscca} and~\ref{wsccb}, immediately apply.  For Case 2, we will need an extension,  Lemma~\ref{chen} in Section~\ref{sec:general}.  In that section we also give the full version of the well-suited curve criterion, a necessary and sufficient condition for a mapping class to be a normal generator.  This full version, Proposition~\ref{prop:wscc}, implies Lemmas~\ref{wscca}, \ref{wsccb}, and~\ref{chen}.

Case 3 is by far the most difficult.  Here we should not expect a simple argument to exist, since the equality $[c] = [f(c)] \mod 2$ holds, for example, for every element of the congruence subgroup $\Mod(S_g)[2]$ and every curve $c$.  

Unlike in Cases 1 and 2, it will not be enough to consider only $c$ and $f(c)$;  we will need to consider $f^2(c)$ in addition to $c$ and $f(c)$.  Thus the first step in Case 3 is to observe the following further consequence of Proposition~\ref{flm}, applied in the case $k=2$:
\[
i(c,f^2(c)) \leq 2.
\]
In Section~\ref{sec:configs}, we enumerate all possibilities for the configuration of the triple $(c,f(c),f^2(c))$, given that $c$ is nonseparating, that $f$ preserves the mod 2 homology class of $c$, that $i(c,f(c)) \leq 2$, and that $i(c,f^2(c)) \leq 2$ (actually we enumerate a larger class of triples).

In Section~\ref{sec:gp} we introduce an additional criterion for a mapping class to be a normal generator, namely, the existence of a good pair of curves.  This new criterion, Lemma~\ref{good pair}, is a generalization of our well-suited curve criterion for separating curves, Lemma~\ref{wsccsep}.  

We prove Theorem~\ref{main:pa} in Section~\ref{sec:pA}.  In treating Case 3 we will show that many of the configurations from Section~\ref{sec:configs} admit a good pair of curves.  We deal with the configurations that do not admit a good pair of curves on an ad hoc basis.  In one subcase we apply the $k=3$ case of  Proposition~\ref{flm} in order to constrain the possible configurations of the quadruple $(c,f(c),f^2(c),f^3(c))$.


\section{The well-suited curve criterion: general version}
\label{sec:general}

In Section~\ref{sec:criterion} we gave two special cases of the well-suited curve criterion for nonseparating curves, Lemmas~\ref{wscca}  and~\ref{wsccb}.  Using these special cases we were able to prove our main theorem about periodic normal generators and to construct pseudo-Anosov normal generators that answer Long's question and resolve Ivanov's conjecture.

In this section we give a more general version of the well-suited curve criterion for nonseparating curves, Chen's Lemma~\ref{chen} below.  Then, as discussed in Section~\ref{sec:int}, we use this lemma to give an application that will be used in Case 2 of the proof of Theorem~\ref{main:pa}, namely, Lemma~\ref{lemma:nonsep lantern}.  Specifically this lemma gives a well-suited curve criterion for the situation where $i(c,f(c))=2$ and $[c] \neq [f(c)] \mod 2$.  At the end of the section we give the general version of the well-suited curve criterion, Proposition~\ref{prop:wscc}, and explain how to use it to recover Chen's lemma.

\subsection*{Chen's lemma} The next lemma is one generalization of our well-suited curve criteria for nonseparating curves, Lemmas~\ref{wscca} and~\ref{wsccb}.  In those lemmas, we require that $c$ and $f(c)$ have particularly simple arrangements, and in particular they have intersection number at most 1.  In the next lemma the intersection number $i(c,f(c))$ can be arbitrarily large, as long as there is a curve $d$ that forms a (specific) simple arrangement with both.

The formulation of the well-suited curve criterion in Lemma~\ref{chen} was suggested to us by Lei Chen.  Our original formulation, Proposition~\ref{prop:wscc} below, is more general.  The full generality of Proposition~\ref{prop:wscc} is not needed for our application, Lemma~\ref{lemma:nonsep lantern} below.  Lemma~\ref{chen} is perhaps the simplest version of the well-suited curve criterion that suffices for the application.

\begin{lemma}
\label{chen}
Let $g \geq 0$ and let $f \in \Mod(S_g)$.  Suppose that there are nonseparating curves $c$ and $d$ in $S_g$ so that $i(c,d)=1$ and $i(f(c),d) = 0$.  Then the normal closure of $f$ contains the commutator subgroup of $\Mod(S_g)$.
\end{lemma}

If we instead have $i(c,d)=0$ and $i(f(c),d) = 1$ then we may apply Lemma~\ref{chen} to $f^{-1}$ in order to deduce that the normal closure of $f^{-1}$, hence $f$, contains the commutator subgroup of $\Mod(S_g)$.

\begin{proof}[Proof of Lemma~\ref{chen}]

The commutator $h = [T_c,f] = T_cT_{f(c)}^{-1}$ lies in the normal closure of $f$.  By the hypotheses on intersection number, we have
\[
i(d,h(d)) = i(d,T_cT_{f(c)}^{-1}(d)) = i(d,T_c(d)) = 1.
\] 
By Lemma~\ref{wscca}, the normal closure of $h$, and so also the normal closure of $f$, contains the commutator subgroup of $\Mod(S_g)$.
\end{proof}

\begin{figure}[h]
\centering
\includegraphics[scale=.325]{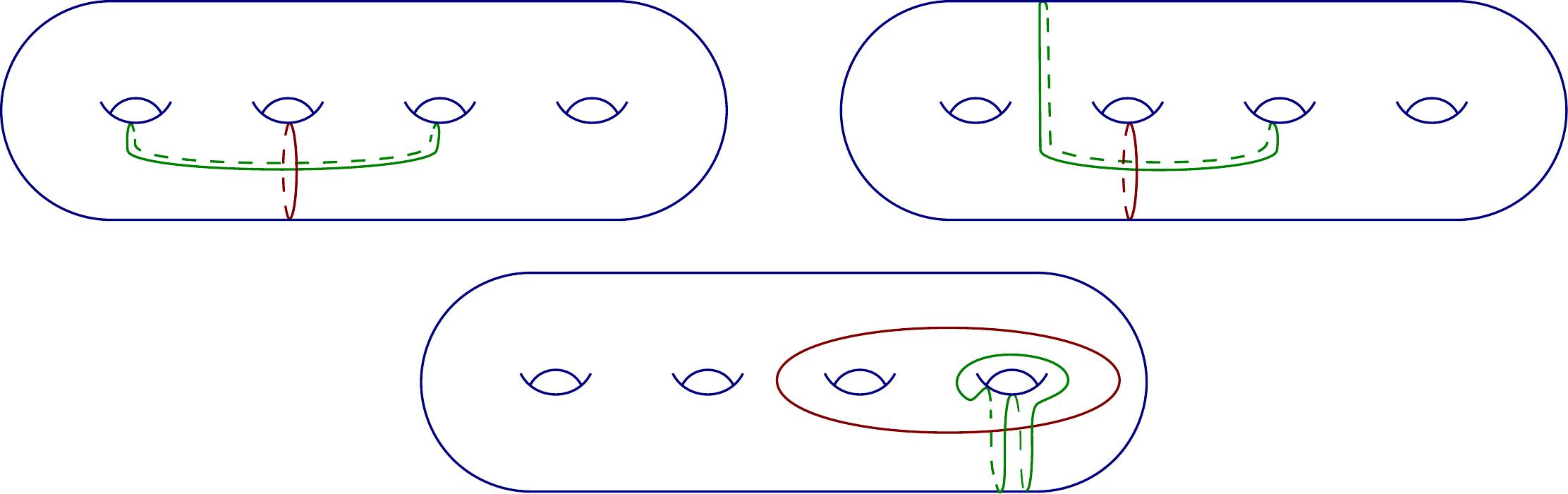}
\caption{The three possible configurations of curves $a$ and $b$ with $i(a,b) = 2$ and $[a] \neq [b] \mod 2$}
\label{fig:lantern}
\end{figure}

For the proof of Theorem~\ref{main:pa} we will require the following specific application of Lemma~\ref{chen}.

\begin{lemma}
\label{lemma:nonsep lantern}
Let $g \geq 0$ and let $f \in \Mod(S_g)$.  Suppose that there is a nonseparating curve $c$ in $S_g$ so that $i(c,f(c)) = 2$ and so that $[c] \neq [f(c)] \mod 2$.  Then the normal closure of $f$ contains the commutator subgroup of $\Mod(S_g)$.
\end{lemma}

\begin{proof}

Up to homeomorphism of $S_g$, there are only three configurations for $c$ and $f(c)$, namely, the configurations shown in Figure~\ref{fig:lantern}.  In each case there is a curve $d$ with $i(c,d)=1$ and $i(f(c),d)=0$.  The lemma thus follows from Lemma~\ref{chen}.
\end{proof}

\subsection*{The general well-suited curve criterion} Our next proposition, Proposition~\ref{prop:wscc}, is the most general statement of the well-suited curve criterion that we will give.  It implies Lemma~\ref{chen}, which in turn implies Lemmas~\ref{wscca} and~\ref{wsccb}.   Not only is Proposition~\ref{prop:wscc} the most general statement that we give, it is in a sense the most general statement possible, since it gives a necessary and sufficient condition for a mapping class to be a normal generator.

In order to state Proposition~\ref{prop:wscc}, we require a definition.  Let $\N(S_g)$ denote the graph of nonseparating curves for $S_g$, that is, the subgraph of $\C(S_g)$ spanned by the vertices corresponding to nonseparating curves.  Let $f$ be an element of $\Mod(S_g)$.  Let $\N_f(S_g)$ denote the abstract graph whose vertex set is the same as that of $\N(S_g)$ and whose edges correspond to pairs of vertices $\{c,h(c)\}$ where $h$ is conjugate to $f$; we call $\N_f(S_g)$ the \emph{graph of curves} for $f$. 

\begin{proposition}
\label{prop:wscc}
Let $g \geq 0$.  Then $f$ contains the commutator subgroup of $\Mod(S_g)$ if and only if $\N_f(S_g)$ is connected.
\end{proposition}

\begin{proof}

First assume that $\N_f(S_g)$ is connected.  We would like to show that the normal closure of $f$ contains the commutator subgroup of $\Mod(S_g)$.  Let $c$ and $d$ be nonseparating curves in $S_g$ with $i(c,d)=1$.  By Lemma~\ref{wscca} it is enough to show that $T_cT_d^{-1}$ lies in the normal closure of $f$.  As $\N_f(S_g)$ is connected there is a path $c = c_0,\dots,c_n=d$  between $c$ and $d$.  For each $i$ there is a conjugate $f_i$ of $f$ with $f_i(c_i)=c_{i+1}$.  It follows that $T_{c_i}T_{c_{i+1}}^{-1} = [T_{c_i},f_i]$ lies in the normal closure of $f$.  But then the product
\[
\left(T_{c_0}T_{c_1}^{-1}\right)\left(T_{c_1}T_{c_2}^{-1}\right) \cdots \left(T_{c_{n-1}}T_{c_n}^{-1}\right)
\]
lies in the normal closure of $f$.  This product is equal to $T_{c_0}T_{c_n}^{-1}$, or $T_cT_d^{-1}$, as desired.

For the other direction, assume that the normal closure of $f$ contains the commutator subgroup of $\Mod(S_g)$.  Let $c$ and $d$ be nonseparating curves in $S_g$.  We would like to show that there is a path from $c$ to $d$ in $\N_f(S_g)$.  There is an element $h$ of the commutator subgroup of $\Mod(S_g)$ with $h(c)=d$.  Indeed, for $g \geq 3$ this follows from the fact that the commutator subgroup of $\Mod(S_g)$ is equal to $\Mod(S_g)$ and the fact that $\Mod(S_g)$ acts transitively on the vertices of $\N_f(S_g)$.  For $g$ equal to 1 or 2 we can first choose an element $h_0$ of $\Mod(S_g)$ taking $c$ to $d$, and take $h$ to be any element of the form $h=T_d^kh_0$ that lies in the commutator subgroup.

As the normal closure of $f$ contains the commutator subgroup of $\Mod(S_g)$ we have that $h = f_n \cdots f_1$ where each $f_i$ is a conjugate of $f$.  Let $c_0 = c$ and for $1 \leq i \leq n$ let
\[
c_i = f_i(c_{i-1}) = f_i \cdots f_1(c_0).
\]
Note $c_n = h(c_0)=d$.  Since each $f_i$ is a conjugate of $f$ and $f_i(c_{i-1})=c_i$ it follows that each pair of vertices $\{c_{i-1},c_i\}$ spans an edge in $\N_f(S_g)$.  So the sequence of vertices $c_0,\dots,c_n$ gives a path in $\N_f(S_g)$ from $c$ to $d$.  Since $c$ and $d$ were arbitrary, $\N_f(S_g)$ is connected.  
\end{proof}

While the reverse implication of Proposition~\ref{prop:wscc} is the one that is relevant to our applications in this paper, the forward implication also can be applied to prove the connectivity of curve graphs associated to mapping classes.  For instance, in Theorem~\ref{translation} we gave examples of normal generators for $\Mod(S_g)$ with arbitrarily large translation lengths on the curve graph $\C(S_g)$.  The edges of the corresponding curve graph $\N_f(S_g)$ thus only connect vertices that are very far apart in $\C(S_g)$.  Perhaps surprisingly, Proposition~\ref{prop:wscc} implies that $\N_f(S_g)$ is connected.  

\subsection*{Chen's lemma via the Putman trick} The well-suited curve criterion of Proposition~\ref{prop:wscc} is only as useful as our ability to show that a graph of curves $\N_f(S_g)$ is  connected.  Putman introduced an effective method for proving the connectedness of such graphs \cite[Lemma 2.1]{Putman}.  We give here a specialized version which follows immediately from Putman's lemma and suffices for our purposes. 

\begin{lemma}
\label{putman}
Suppose $G$ is a group and $X$ is a graph on which $G$ acts.  Suppose $G$ acts transitively on the vertices and that $G$ is generated by the set $\{g_i\}$.  If there is an edge $e$ of $X$ with the property that each $g_i$ fixes at least one vertex of $e$ then $X$ is connected.
\end{lemma}

We can use this lemma and Proposition~\ref{prop:wscc} to prove Lemma~\ref{chen} as follows.  Suppose $f$ is an element of $\Mod(S_g)$ and $c$ and $d$ are nonseparating curves in $S_g$ so that $i(c,d)=1$ and $i(f(c),d) = 0$.  We can choose a generating set for $\Mod(S_g)$ consisting of Dehn twists about curves $c_1,\dots,c_{2g+1}$, where $c_1 = c$, $c_2 = d$, and all other $c_i$ are disjoint from $c$.  Let $e$ be the edge of $\N_f(S_g)$ corresponding to $\{c,f(c)\}$.  Since $T_{c_2}(f(c)) = f(c)$ and $T_{c_i}(c) = c$ for all $i \neq 2$, it follows from Lemma~\ref{putman} that $\N_f(S_g)$ is connected, and so Chen's lemma follows from Proposition~\ref{prop:wscc}.


\section{Curve configurations for mod 2 homologous curves}
\label{sec:configs}

In this section we will provide one of the main tools for dealing with Case 3 of the proof of Theorem~\ref{main:pa}, as outlined in Section~\ref{sec:int}. Specifically, we produce a list of configurations of curves $(\ga,\de,\ep)$.   This list  contains all triples of curves $(c, f(c),f^2(c))$ in $S_g$ where $i(c,f(c))=i(f(c),f^2(c)) \leq 2$ and $i(c,f^2(c)) \leq 2$ and where all three curves are homologous mod 2.  

In the remainder of this section (and this section only) we will use the word ``curve'' to mean a simple closed curve, as opposed to a homotopy class of curves.  To emphasize this point, we will use Greek letters to denote curves in this section.  For a surface $S$ and a curve $\alpha$ we will write $S \cut \alpha$ for the surface obtained from $S$ by cutting along $\alpha$.

The larger set of triples of curves that we will classify is the set of triples of curves $(\ga,\de,\ep)$ in any $S_g$ with the following properties:
\begin{itemize}
\item[$\cdot$] all three curves are pairwise homologous mod 2,
\item[$\cdot$] all three curves are pairwise homotopically distinct and in minimal position,
\item[$\cdot$] all three curves have pairwise intersection number at most 2,
\item[$\cdot$] $\ga \cup \de$ and $\de \cup \ep$ have the same number of complementary regions,
\item[$\cdot$] $i(\ga,\de)=i(\de,\ep)$, and
\item[$\cdot$] $|\hat\imath|(\ga,\de)=|\hat\imath|(\de,\ep)$.
\end{itemize}
(We give a more concise description of these triples below.)   The triples $(c,f(c),f^2(c))$ described at the start of the section satisfy the properties listed here.  The set of triples $(\ga,\de,\ep)$, however, is a priori a larger, since we do not make the assumption that the pair $(\gamma,\delta)$ is equivalent to the pair $(\delta,\epsilon)$ under a homeomorphism of $S_g$. 

The main goal of this section is to show that any triple $(\ga,\de,\ep)$ that satisfies this list of properties is (up to stabilization, defined below) one of the triples shown in Figures~\ref{fig:I}, \ref{fig:II}, \ref{fig:III}, or~\ref{fig:IV}.  This classification is proved as Propositions~\ref{prop:I}, \ref{prop:II}, \ref{prop:III}, and~\ref{prop:IV} below.

The strategy for classifying such triples $(\ga,\de,\ep)$ is as follows.  First, there are only four configurations for the pair $(\ga,\de)$ in $S_g$, up to changing the genus of the complementary regions.  The four corresponding ``templates'' are shown in Figure~\ref{fig:configs}.  The goal is then to find all candidate third curves $\ep$, using the hypotheses on the curves.  To do this, we first draw on each of the templates the finitely many candidates for $\ep$ up to homeomorphism.  We then enumerate all the ways of adding handles to each such configuration so that the resulting configuration has the desired properties.  Up to changing the genus of the complementary regions, there are finitely many such ways to add the handles.  

\subsection*{Four configurations for pairs of mod 2 homologous curves} We first give names to the various types of configurations of nonseparating curves $\ga$ and $\de$ with $\ga$ and $\de$ in minimal position, with $|\ga \cap \de| \leq 2$, with $\ga$ and $\de$ homotopically distinct, and with $[\ga]=[\de] \mod 2$.  There are four types, as in the following chart.

\medskip

\renewcommand{\arraystretch}{1.35}
\begin{center}
 \begin{tabular}{|c|c|c|c|c|}
   \hline
   Type & $|\ga \cap \de|$ &  $|\hat\imath|(\ga,\de)$ & $[\ga]=[\de]$ & $[\ga]=[\de] \mod 2$ \\ \hline 
   I & 0 &  0 & \cmark & \cmark \\ \hline
   II & 2 & 0 & \cmark & \cmark \\ \hline
   III & 2 & 0 & \xmark & \cmark \\ \hline
   IV & 2 & 2 & \xmark & \cmark \\
   \hline
 \end{tabular}
\end{center}

\medskip

We emphasize that the type of a pair of curves is independent of the order of the two curves.

In what follows it will be convenient to depict the different types of configurations by drawing the surface obtained by cutting $S_g$ along the curve $\de$.  The cut surface $S_g \cut \delta$ is connected and has two boundary components, and on this surface the curve $\ga$ either becomes a curve (type I) or a pair of arcs (types II, III, and IV).  See Figure~\ref{fig:configs} for sample configurations of the four types.  The picture for type II is drawn in such a way as to suggest an analogy with types III and IV.  

In Figure~\ref{fig:configs} the identification of the two components of the boundary of $S_g \cut \delta$ is indicated by the pair of black dots.  That is, the two boundary components are glued so that the black dots are identified in the closed surface.  In the pictures for types II, III, and IV the identification of the two boundary circles can be achieved by a vertical translation followed by a reflection about the plane passing through $\ga$.  In what follows we will draw configurations of pairs $(\ga,\de)$ and triples $(\ga,\de,\ep)$ using similar pictures, and in those pictures we will always use this identification of the boundary components of $S_g \cut \delta$.  There are only two choices for the identification so that $\ga$ is a union of curves in the closed surface, and the identification we have specified is the only one that makes $\ga$ a connected curve.

Given a configuration of curves on a surface, we may perform a \emph{stabilization}.  That is, we may delete the interiors of two disjoint disks that lie in the same component of the complement of the configuration and then identify the two new boundary components.  A sequence of such stabilizations is also called a stabilization.  If a pair of curves is of type I, II, III, or IV, then the stabilized configuration is of the same type.

The sample configurations in Figure~\ref{fig:configs} are \emph{minimal} in the sense that any pair of curves of a given type can be obtained from the corresponding sample configuration by stabilization.

\begin{figure}[h!]
\qquad \includegraphics[width=.9\textwidth]{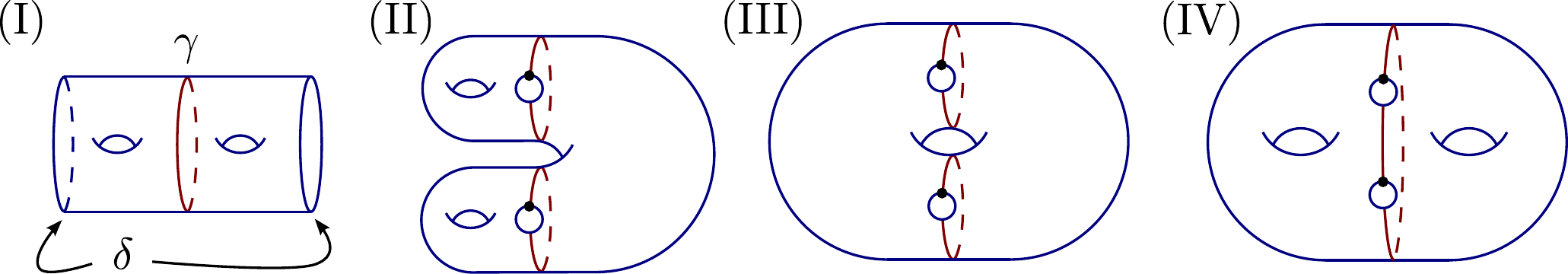}
\caption{The minimal configurations of nonseparating curves $\ga$ and $\de$ with $i(\ga,\de) \leq 2$ and $[\ga]=[\de] \mod 2$; in each picture the surface is cut along $\de$}
\label{fig:configs}
\end{figure}

\subsection*{Configurations of triples of mod 2 homologous curves} Our main goal in this section is to classify certain types of configurations of ordered triples of curves $(\ga,\de,\ep)$.  Before we embark on this classification we would to rephrase the goal in terms of the four types of pairs of curves described above.

Suppose that $(\ga,\de,\ep)$ is an ordered triple of curves where each pair of curves has type I, II, III, or IV (in particular each pair is in minimal position).  We will say that the triple of curves is of type I, II, III, or IV if all three curves are pairwise non-homotopic and both pairs $\{\ga,\de\}$ and $\{\de,\ep\}$ are of the given type.  The motivation here is that in the proof of Theorem~\ref{main:pa} we will be interested in triples of the form $(c,f(c),f^2(c))$ and in this case the pair $\{c,f(c)\}$ is of the same type as the pair $\{f(c),f^2(c)\}$ (the pair $\{c,f^2(c)\}$ might not be of the same type as the other two pairs).

It will continue to be convenient to draw our configurations on a cut surface $S_g \cut \de$.  On $S_g \cut \de$ each of the curves $\ga$ and $\ep$ becomes a curve or a pair of arcs, just as in Figure~\ref{fig:configs} (again, up to stabilization).  Moreover, since $\{\ga,\de\}$ has the same type as $\{\de,\ep\}$ (and since the type of a pair is independent of the order of the two curves) the $\ga$-curve/arcs and the $\ep$-curve/arcs on $S_g \cut \de$ are of the same type, in that they correspond to the same picture in Figure~\ref{fig:configs}.

We may now restate the main result of this section as follows:

\medskip

\begin{quote}
\emph{Every ordered triple of curves of type I, II, III, or IV is a stabilization of one of the minimal configurations given in Figures~\ref{fig:I}, \ref{fig:II}, \ref{fig:III}, or \ref{fig:IV}, respectively.
}
\end{quote}

\medskip

The four different types of triples are addressed in Propositions~\ref{prop:I},  \ref{prop:II}, \ref{prop:III}, and~\ref{prop:IV}, respectively.   Again, these classifications will be used in the proof of Theorem~\ref{main:pa} in Section~\ref{sec:pA}.  

\begin{figure}
\centering
\includegraphics[scale=.45]{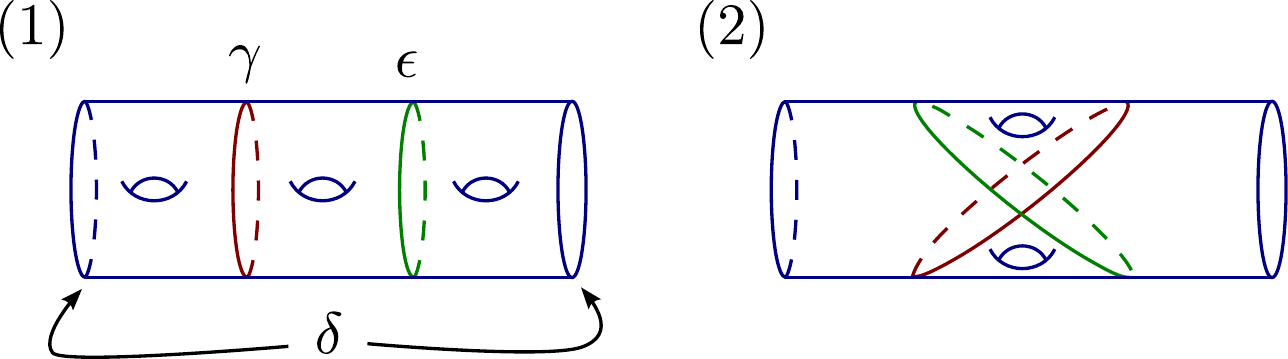}
\caption{The two minimal configurations of ordered triples of curves of type I, cut along $\delta$}
\label{fig:I}
\end{figure}

\subsection*{Triples of type I} We can right away classify triples of type I.  This case is much simpler than the cases of triples of types II, III, and IV.

\begin{proposition}
\label{prop:I}
Up to homeomorphism, every ordered triple of curves of type I in $S_g$ is a stabilization of one of the minimal configurations in Figure~\ref{fig:I}.
\end{proposition}

\begin{proof}

Let $(\ga,\de,\ep)$ be an ordered triple of curves of type I in $S_g$.  Consider the cut surface $S_g \cut \de$.  The curves $\ga$ and $\ep$ correspond to curves in $S_g \cut \de$ that separate the two boundary components from each other.  Since $\ga$ and $\ep$ are distinct and since they intersect in at most two points, it follows that the triple $(\ga,\de,\ep)$ is given by a stabilization of one of the two configurations in Figure~\ref{fig:I}.   
\end{proof}

\subsection*{Templates} In Propositions~\ref{prop:II}, \ref{prop:III}, and~\ref{prop:IV} below we will classify ordered triples of curves of type II, III, and IV.  We will do this by first listing simpler versions of each type of configuration, called templates.

Before we define templates, we first say what it means to add handles to a configuration of curves.  Let $\ga_1,\dots,\ga_k$ be a collection of curves in a surface.  To add handles to this configuration, we delete the interiors of an even number of disjoint disks in the complement of $\cup \ga_i$ and glue the resulting boundary components in pairs.  The order in which handles are added does not affect the final configuration.  

Let $(\ga,\de,\ep)$ be an ordered triple of curves of type II, III, or IV in $S_g$.  Set $S = S_1$ if the type is II or IV and set $S=S_2$ if the type is III (where $S_1$ and $S_2$ denote the closed surfaces of genus 1 and 2, respectively).  A \emph{template} for $(\ga,\de,\ep)$ is an ordered triple of curves $(\ga_0,\de_0,\ep_0)$ in $S$ so that the following conditions hold:
\medskip
\begin{enumerate}[itemsep=1.5ex]
\item $\ga_0$ and $\de_0$ are configured in $S$  as in Figure~\ref{fig:configs template}, and
\item the triple $(\ga,\de,\ep)$ is obtained from $(\ga_0,\de_0,\ep_0)$ by adding handles. 
\end{enumerate}
\medskip
It follows from the second condition that the pairwise intersection numbers and algebraic intersection numbers of $\ga_0$, $\de_0$, and $\ep_0$ are inherited from those of $\ga$, $\de$, and $\ep$.

\begin{figure}[h!]
\qquad \includegraphics[scale=.4]{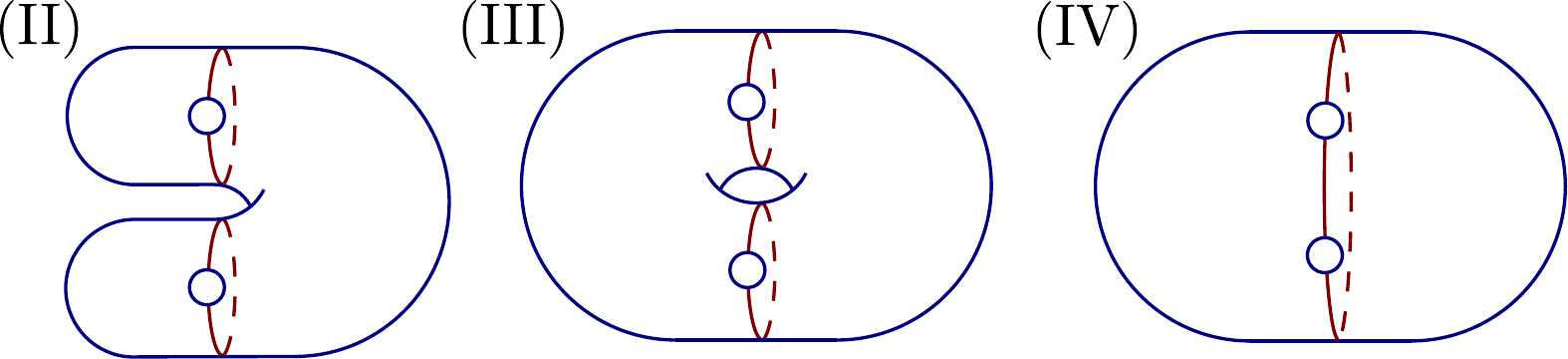}
\caption{The configurations $\{\ga_0,\de_0\}$ in the templates of type II, III, and IV; the surfaces are cut along $\de_0$ and the resulting boundary components are identified as in Figure~\ref{fig:configs}}
\label{fig:configs template}
\end{figure}

\begin{figure}
\centering
\includegraphics[width=.75\textwidth]{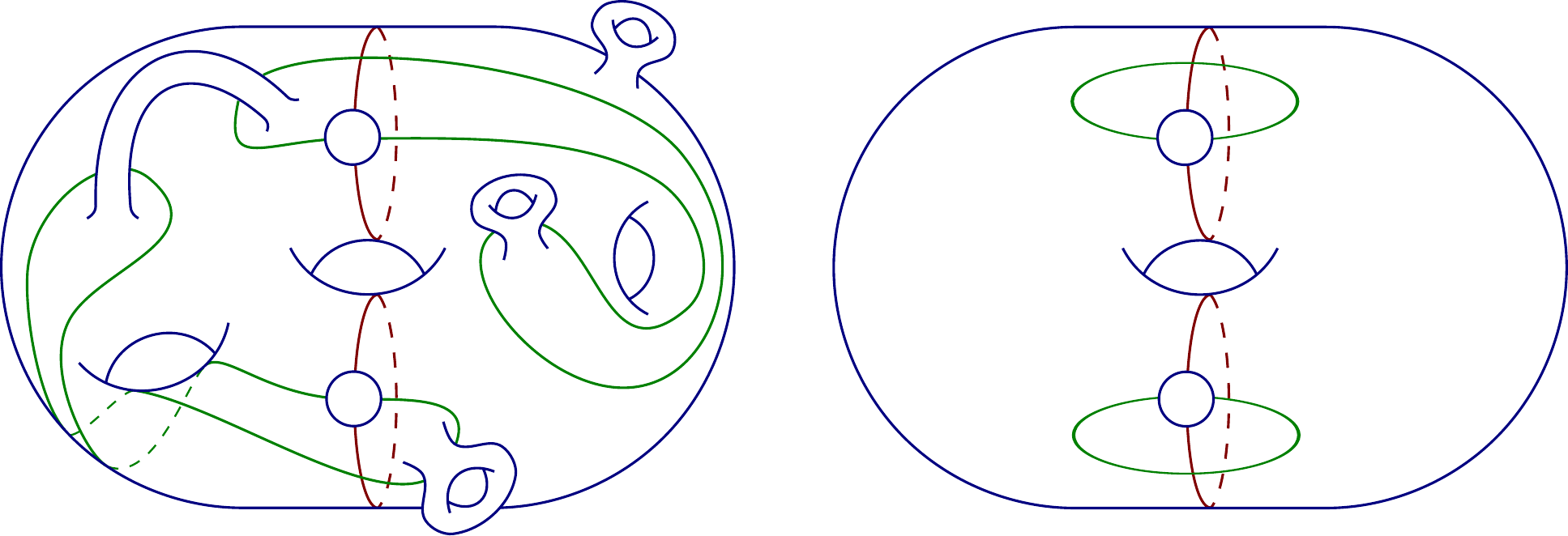}
\caption{An example of an ordered triple of curves of type III (left) and its underlying template (right)}
\label{fig:example}
\end{figure}

As one example, the configuration shown on the right-hand side of Figure~\ref{fig:example} is a template in $S_2$ for the triple of curves of type III in $S_9$ shown on the left-hand side of the figure. 

We note that the triples of curves in a template might fail to be a triple of type I, II, III, or IV in several ways: they might not be in minimal position, they might not be homotopically distinct, and the pair $\{\de_0,\ep_0\}$ might not have the same type as $\{\de,\ep\}$.  In the example of Figure~\ref{fig:example}, the triple of curves in the template fails in all three ways.

\begin{lemma}
\label{templates}
Every ordered triple of curves of type II, III, or IV in $S_g$ has a template in $S_1$, $S_2$, or $S_1$, respectively.
\end{lemma}

\begin{proof}

For concreteness we state the proof in the case of a triple of type III; the proofs for triples of type II and IV are similar.  Let $(\ga,\de,\ep)$ be an ordered triple of curves in $S_g$ of type III.  The configuration $\{\ga,\de\}$ in $S_g$ is a stabilization of the standard pair of type III shown in Figure~\ref{fig:configs}.  The curves $\ga$ and $\de$ divide $S_g$ into two components, which we will refer to as $R_1$ and $R_2$.  Each is a surface with two boundary components.  Since $\{\de,\ep\}$ is of type III, the curve $\ep$ gives rise to collections of arcs in $R_1$ and $R_2$.  Since $|\ga \cap \ep|$ and $|\de \cap \ep|$ are bounded above by 2, the intersection of $\ep$ with each $R_i$ is at most two arcs.

Fix $i \in \{1,2\}$.  If there are two arcs of $\ep$ in $R_i$ and both arcs have both their endpoints on the same component of the boundary of $R_i$, then these endpoints are unlinked on the boundary, in that they do not alternate.  Indeed, if there were two arcs of $\ep$ in $R_i$ that were linked on the boundary then we could find a curve in $S_g$ that is disjoint from $\ga$ and $\de$, intersects one of these arcs of $\ep$ in one point, and is disjoint from $\ep$ otherwise.  This contradicts the assumption that $\ep$ is homologous to $\ga$ and $\de$ mod 2.  

Consider a regular neighborhood of the union of the arcs of $\ep$ in $R_1$ with the boundary of $R_1$.  By the previous paragraph, this neighborhood has genus 0.  It follows that there is a collection of curves in $R_1$ so that when we cut $R_1$ along these curves we obtain a connected surface of genus 0.  By the same argument there is a similar collection of curves in $R_2$.  Cutting $S_g$ along all of these curves and gluing disks to the resulting boundary components, we obtain a surface of genus 2.  The curves arising from $\ga$ and $\de$ are configured in $S_2$ like the curves $\ga_0$ and $\de_0$ in Figure~\ref{fig:configs template}.  By construction the configuration of curves in $S_2$ is a template for the original triple $(\ga,\de,\ep)$ in $S_g$.  
\end{proof}

In the process of classifying triples of curves of types II, III, and IV, we will start in each case by classifying all templates of that type.  The next lemma gives a condition that a configuration of curves must satisfy in order to be a template.  

\begin{lemma}
\label{prune sep}
Suppose $\ga_0$, $\de_0$, and $\ep_0$ are curves giving a template for an ordered triple of curves $(\ga,\de,\ep)$ in $S_g$ of type II, III, or IV.  Then $\ep_0$ is a nonseparating curve in the underlying surface of the template.
\end{lemma}

\begin{proof}

Suppose for contradiction that $\ep_0$ is a separating curve in the template.  Because $\ep$ is nonseparating in $S_g$ it must be that the ordered triple in $S_g$ is obtained from the template by adding at least one handle to the template that connects the two sides of $\ep_0$.  The three curves of the template cut the underlying surface into regions.  Say that the added handle connects regions $R_1$ and $R_2$.  Since $\de_0$ is nonseparating, there exists a path $\alpha$ from $R_1$ to $R_2$ that avoids $\de_0$.    The concatenation of $\alpha$ with a path along the added handle thus gives rise to a curve in $S_g$ that intersects $\ep_0$ in a single point and is disjoint from $\de_0$.  This contradicts the assumption that $\ep$ and $\de$ are homologous mod 2.  
\end{proof}

\subsection*{Linking} Let $(\ga,\de,\ep)$ be an ordered triple of curves in $S_g$ of type II, III, or IV.  Since $|\ga \cap \de|= |\de \cap \ep|=2$, each of $\ga \cap \de$ and $\ep \cap \de$ is a pair of points on $\de$.  Because there are no triple intersections, these four points are distinct and so the two pairs can be either linked or unlinked.  As such we may say that $\ga$ and $\ep$ are \emph{linked} or \emph{unlinked} along $\de$.  In our classification of triples of type II, III, and IV we will use this notion to distinguish between various cases.  

\begin{figure}
\centering
\includegraphics[width=.9\textwidth]{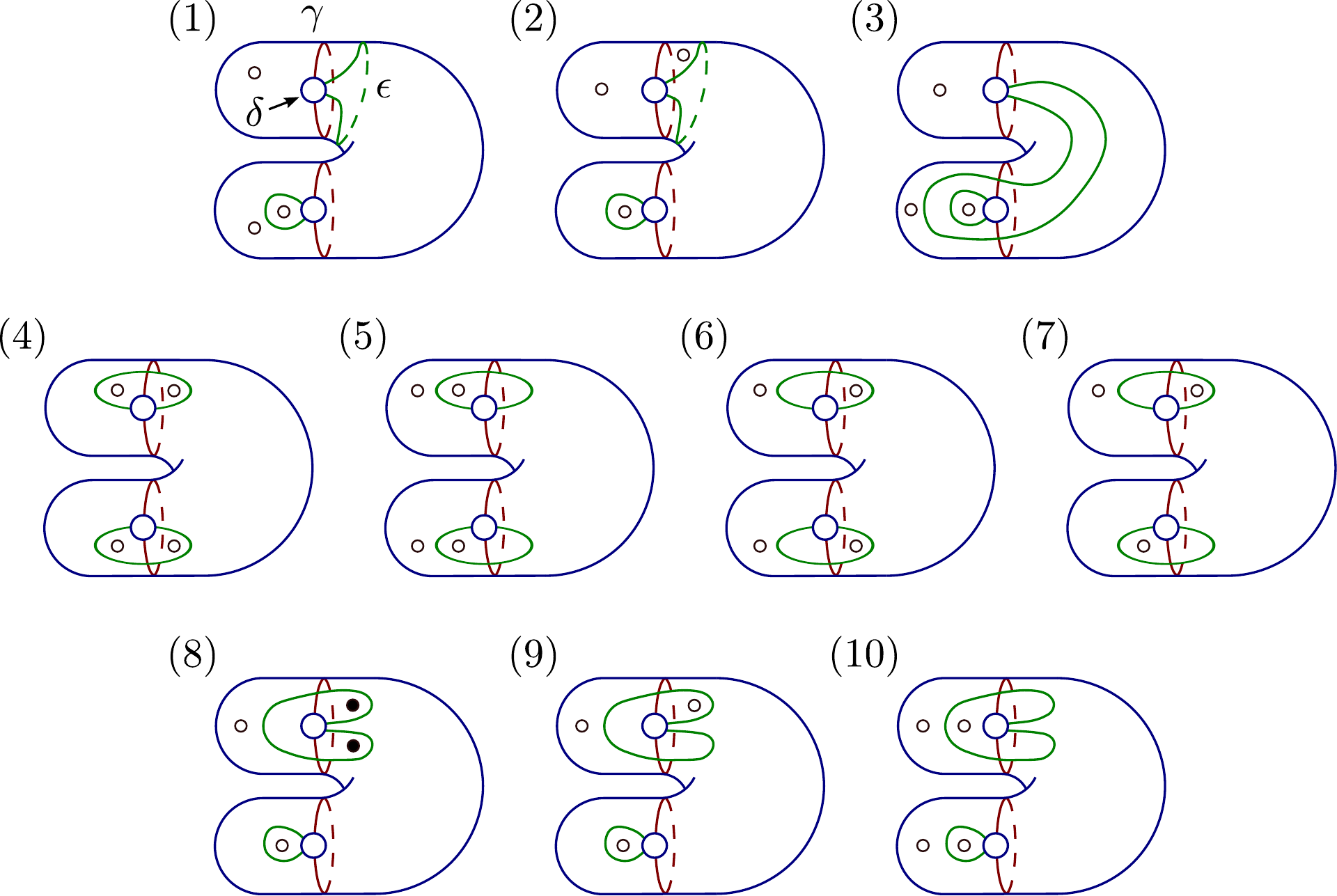}
\caption{The ten minimal configurations of triples of type II; a white dot indicates a handle attached within a single region and a pair of black dots indicates a handle connecting two distinct regions
 }
\label{fig:II}
\end{figure}

\subsection*{Triples of type II} Having established the preliminaries about templates, we now begin the process of classifying ordered triples of types II, III, and IV.  We first classify ordered triples of type II.  

\begin{proposition}
\label{prop:II}
Up to homeomorphism, every ordered triple of curves of type II in $S_g$ is a stabilization of one of the minimal configurations in Figure~\ref{fig:II}.
\end{proposition}

\begin{proof}

The proof has two main phases:
\begin{enumerate}
\item classify all possible templates in $S_1$ for a triple of type II, and
\item build all minimal configurations of type II by adding handles to the templates.
\end{enumerate}
It follows from Lemma~\ref{templates} that this two-phase process recovers all triples of type II.

\medskip

\noindent \emph{Phase 1.} Suppose that $(\ga_0,\de_0,\ep_0)$ is a template in $S_1$ for a triple of curves $(\ga,\de,\ep)$ in $S_g$.  We will show that the template $(\ga_0,\de_0,\ep_0)$ in $S_1$ is, up to homeomorphism, one of the four templates shown in Figure~\ref{fig:II templates}.  

Consider the cut surface $S_1 \cut \de_0$.  We call the two resulting boundary components $\de_1$ and $\de_2$.   On the cut surface, $\ga_0$ and $\ep_0$ become pairs of arcs $\{\ga_1,\ga_2$\} and $\{\ep_1,\ep_2\}$.  The arcs $\ga_1$ and $\ga_2$ are as shown in configuration II in Figure~\ref{fig:configs template}.  What remains is to determine how $\ep_1$ and $\ep_2$ lie in the cut surface.

Since $(\ga,\de,\ep)$ is of type II, each of the $\ga_i$-arcs and $\ep_i$-arcs connects some $\de_i$ to itself.  The labels of the $\ga_i$-arcs and $\ep_i$-arcs are inherited from the labels of the $\de_i$ on which they have their endpoints.

The arcs $\ga_1$ and $\ga_2$ are both separating and they divide the cut surface $S_1 \cut \de_0$ into one annulus $R_0$ and two disks, $R_1$ and $R_2$.  We choose the names $R_1$ and $R_2$ according to which of $\de_1$ and $\de_2$ it intersects.

We may record the intersection numbers between the arcs $\ga_i$ and $\ep_i$ with a $2 \times 2$ matrix.  We will arrange these matrices so that the rows correspond to the $\ga_i$ and the columns correspond to the $\ep_i$.  Since $|\ga_0 \cap \ep_0|$ is 0 or 2, the total sum of the entries in the matrix must be 0 or 2.  Further, since each $\ga_i$ is separating, each off-diagonal entry of the matrix must be even.  

\begin{figure}
\centering
\includegraphics[width=.9\textwidth]{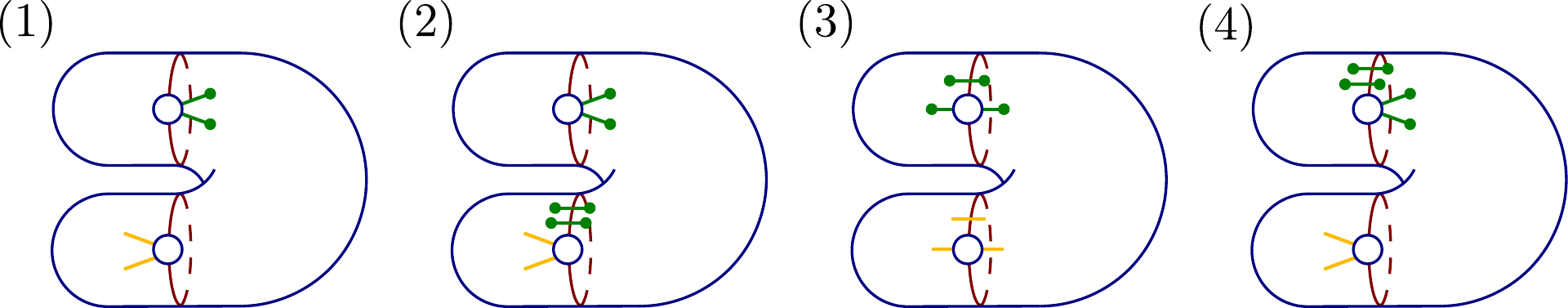}
\caption{The partial templates for configuration II; the partial arcs with dots belong to $\ep_1$ and the other arcs to $\ep_2$}
\label{fig:partial II}
\end{figure}

Up to renumbering the two arcs in each pair, there are four such matrices:
\[
 \begin{bmatrix}  0 & 0 \\ 0 & 0 \end{bmatrix}, \ 
 \begin{bmatrix}  0 & 0 \\ 2 & 0 \end{bmatrix},   \ \begin{bmatrix}  1 & 0 \\ 0 & 1 \end{bmatrix}, \ 
 \text{and} \ 
 \begin{bmatrix}  2 & 0 \\ 0 & 0 \end{bmatrix}.
\]
From these matrices, we can draw corresponding partial templates as in Figure~\ref{fig:partial II}, each unique up to homeomorphism of $S_1$.  To do this, we must first observe that each of the four matrices determines whether or not $\ga_0$ and $\ep_0$ are linked along $\de_0$: if the sum of the entries in each column is odd then they are linked and otherwise they are unlinked.  This is because the parity of the sum of the entries in the $i$th column is the number of times $\ep_i$ crosses $\ga_0$.  We treat the partial templates in turn.

\begin{figure}
\centering
\includegraphics[width=.9\textwidth]{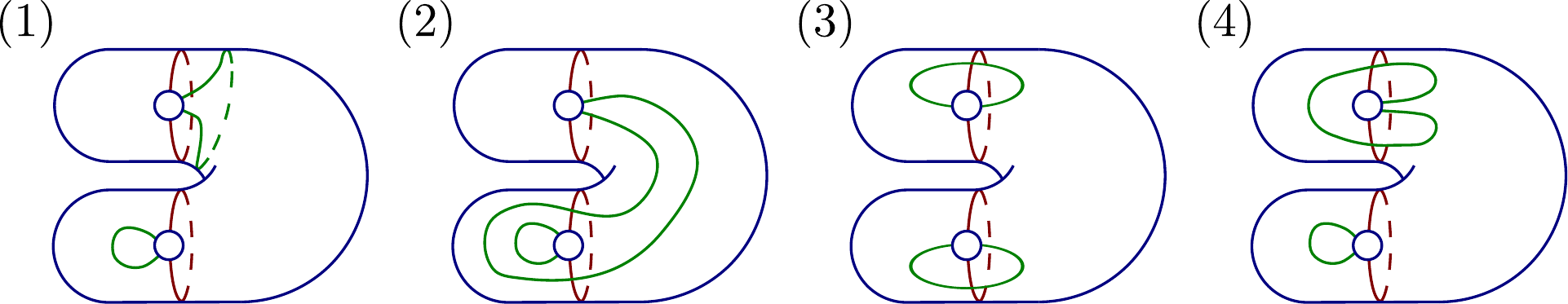}
\caption{The four templates for triples of type II}
\label{fig:II templates}
\end{figure}

\medskip

\noindent \emph{Partial template 1.} Since $R_2$ is a disk, there is a unique choice for the arc $\ep_2$ up to homeomorphism.  There are two choices for the arc $\ep_1$ in $R_0$.  One choice of $\ep_1$ gives rise to a separating curve $\ep_0$, which violates Lemma~\ref{prune sep}.  The other choice of $\ep_1$ gives template 1 in Figure~\ref{fig:II templates}.  

\medskip

\noindent \emph{Partial template 2.} There is a unique choice for $\ep_2$ and a unique choice for the intersection of $\ep_1$ with $R_2$.  Up to homeomorphism, there are two choices for the intersection of $\ep_1$ with $R_0$, determined by the induced pairing of the two points of $\ep_1 \cap \de_1$ with the two points of $\ep_1 \cap \ga_2$.  One of the possibilities leads to a separating curve $\ep_0$, violating Lemma~\ref{prune sep}.  The other leads to template 2 in Figure~\ref{fig:II templates}.

\medskip

\noindent \emph{Partial template 3.} The arcs $\ep_1$ and $\ep_2$ must bound disks with $\de_1$ and $\de_2$, respectively.  The boundaries of these two disks specify two arcs of $\de_1$ and $\de_2$ (the arcs of $\de_1$ and $\de_2$ contained in the disks).  Up to homeomorphism, there are then two possibilities: these two arcs are either equal in (the uncut) $S_1$ or not.  If they are equal, then the resulting curve $\ep_0$ is separating, violating Lemma~\ref{prune sep}.  The other possibility leads to template 3 in Figure~\ref{fig:II templates}.  

\medskip

\noindent \emph{Partial template 4.} There is a unique choice of $\ep_2$ and a unique choice for the intersection of $\ep_1$ with $R_1$.  Up to homeomorphism, there are two possibilities for the intersection of $\ep_1$ with $R_0$: the two sub-arcs are either nested or not.  When they are nested the resulting curve $\ep$ is separating.  When they are not nested, we obtain template 4 in Figure~\ref{fig:II templates}.

\medskip

Since all of the templates found are shown in Figure~\ref{fig:II templates}, this completes the first phase of the proof.

\medskip

\noindent \emph{Phase 2.} As mentioned above, Lemma~\ref{templates} gives that every minimal configuration of type II on a surface $S_g$ arises from one of the templates by adding handles.  Before addressing each template, we establish three constraints on the process of adding handles in order to obtain a minimal configuration of type II: the minimality condition, the checkerboard condition, and the type II condition.

We can record the set of handles required to pass from a given template to a given type II configuration in $S_g$ with a multiset (i.e. unordered list) of unordered pairs $\{R_i,R_j\}$ of not-necessarily-distinct regions in the complement of $\ga_0 \cup \de_0 \cup \ep_0$.

There is a partial order on configurations of type II where one configuration is less than another if the latter is obtained from the former by stabilization.  The desired configurations are the ones that are minimal in this partial order.

For a template and multiset of handles of pairs of regions as above to specify a minimal configuration, it must satisfy the following condition:
\begin{enumerate}
\item[] \emph{Minimality condition.} Each pair $\{R_i,R_j\}$ appears at most once, and the pair $\{R_i,R_i\}$ may appear only if $R_i$ is contained in a bigon or an annulus bounded by two of the curves.
\end{enumerate}
The first part of the minimality condition forces our multiset to be a set; we will refer to the multiset as the handle set in what follows.  Since the handle set must be finite (there are finitely many pairs of regions), we have thus reduced our problem of finding all minimal configurations to a finite check for each template.  

There are certain pairs of regions that may not appear in the handle set.  The curves $\ga_0$, $\de_0$ and $\ep_0$ are all homologous mod 2 in the templates.  We may not add any handles destroying this property.  For each pair of curves we may color the complementary regions by two colors so that two regions adjacent along an arc of one of the curves have different colors.  The resulting restriction on the handle set is as follows.
\begin{enumerate}
\item[] \emph{Checkerboard condition.} The handle set may not include any pair of regions that have different colors in any of the three colorings.
\end{enumerate}
Alternatively, any path connecting a pair of regions in the handle set must cross all three curves $\ga_0$, $\de_0$, and $\ep_0$ the same number of times mod 2.  The checkerboard condition eliminates many possibilities in the finite check.  

One requirement for $(\ga,\de,\ep)$ to be a triple of type II is that each of $\ga \cup \de$ and $\de \cup \ep$ must cut $S_g$ into exactly three regions.  It is already the case in each the templates that $\ga_0 \cup \de_0$ and $\de_0 \cup \ep_0$ cut $S_1$ into exactly three regions.  We thus obtain one additional condition, as follows.
\begin{enumerate}
\item[] \emph{Type II condition.} The handle set may not include any pair of regions (determined by the triple) that lie in two distinct regions determined by either of the pairs $\{\ga_0 , \de_0\}$ or $\{\de_0 , \ep_0\}$.  
\end{enumerate}

\medskip

We are now ready to determine all minimal configurations of type II obtained from the four templates.  The task for each template is to find all handle sets that respect the three constraints above and result in curves lying in minimal position.  In particular we must add handles to eliminate bigons and annuli formed by any pair of curves in the triple.

\medskip

\noindent \emph{Templates 1, 2, and 3.}  In these cases all pairs of distinct regions fail either the checkerboard or the type II condition.  Therefore, in adding handles to these templates we may only use stabilizations, subject to the minimality condition.  It thus remains to determine which subsets of the set of regions give rise to configurations of curves that are in minimal position and are minimal in the partial order.  The resulting minimal configurations are configurations 1--7 in Figure~\ref{fig:II}.  

\medskip

\noindent \emph{Template 4.}  In this case there is exactly one distinct pair of regions that satisfies both the checkerboard and type II conditions, namely, the two regions marked by black dots in configuration 8 in Figure~\ref{fig:II}.  This pair of regions may be in the handle set or not.  After making this choice, it remains to determine which stabilizations satisfy the minimality condition and result in configurations of curves that are in minimal position.  The resulting minimal configurations are configurations 8--10 of Figure~\ref{fig:II}.

\medskip

All of the minimal configurations found are shown in Figure~\ref{fig:II} and so the proposition is proven.
\end{proof}

\subsection*{Triples of type III}  Our next goal is to classify ordered triples of curves of type III.  We require one further lemma about templates of type III.

\begin{lemma}
\label{prune sepnonsep}
Suppose $\ga_0$, $\de_0$, and $\ep_0$ are curves in $S_2$ giving a template for an ordered triple of curves $(\ga,\de,\ep)$ in $S_g$ of type III.  Then the two arcs of $\ep_0$ in the cut surface $S_2 \cut \de_0$ are both separating or both nonseparating.
\end{lemma}

\begin{proof}

Consider the cut surface $S_2 \cut \de_0$, and suppose that one of the $\ep_0$-arcs in $S_2 \cut \de_0$ is separating and one is nonseparating.  In $S_g$ the configuration $\{\de,\ep\}$ is a stabilization of the configuration III in Figure~\ref{fig:configs}.  This means that after adding handles to the template on $S_2$, the two arcs of $\ep$ must be nonseparating in the cut surface $S_g \cut \de$, and their union must separate $S_g$.

The only way to make the separating $\ep_0$-arc in $S_2$ into a nonseparating arc in $S_g \cut \delta$ would be for one of the added handles to connect its complementary regions.  In the surface with the added handles the union of the two $\ep$-arcs is then nonseparating in $S_g \cut \delta$.  Thus $(\ga_0,\de_0,\ep_0)$ is not a template for $(\ga,\de,\ep)$.
\end{proof}

\begin{figure}
\centering
\includegraphics[width=.9\textwidth]{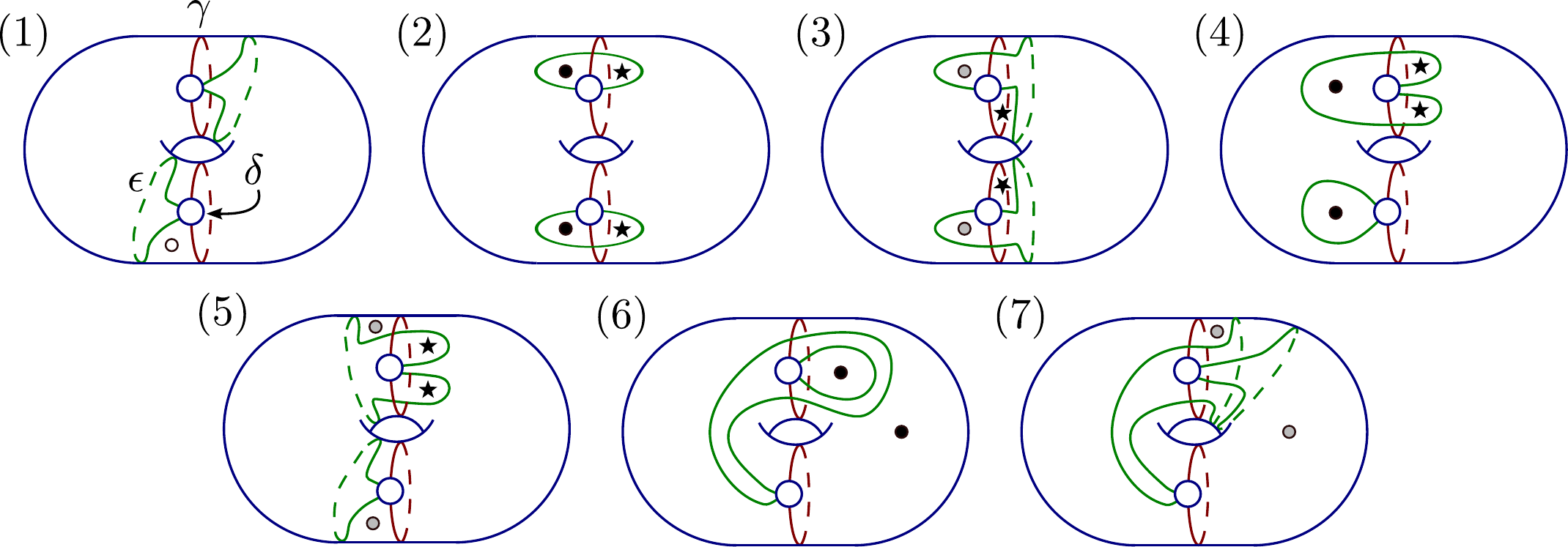}
\caption{The sixteen minimal configurations of triples of type III; a white dot indicates a handle attached within a single region, a pair of black dots indicates a handle connecting two distinct regions, a pair of gray dots indicates two possibilities, namely, one handle attached within each of the two regions or a single handle connecting the two regions, and a pair of stars indicates two possibilities, namely, either no handle is added or a handle is added connecting the two regions}
\label{fig:III}
\end{figure}

\begin{proposition}
\label{prop:III}
Up to homeomorphism, every ordered triple of curves of type III in $S_g$ is a stabilization of one of the minimal configurations in Figure~\ref{fig:III}.
\end{proposition}

\begin{proof}

We complete the proof in the same two phases as in the proof of Proposition~\ref{prop:II}, first listing all possible templates in $S_2$ for a triple of type III and then building all minimal configurations of type III from these templates.

\medskip

\noindent \emph{Phase 1.} Suppose that $(\ga_0,\de_0,\ep_0)$ is a template in $S_2$ for a triple of curves $(\ga,\de,\ep)$ in $S_g$.  We will show that the template $(\ga_0,\de_0,\ep_0)$ in $S_2$ is, up to homeomorphism, one of the seven templates shown in Figure~\ref{fig:III templates}. 

The curves $\ga_0$ and $\de_0$ are configured in $S_2$ as in the type III configuration in Figure~\ref{fig:configs template}.  Consider the cut surface $S_2 \cut \de_0$, as in the figure.  We call the two resulting boundary components $\de_1$ and $\de_2$.

In $S_2 \cut \de_0$ both $\ga_0$ and $\ep_0$ become pair of arcs $\{\ga_1,\ga_2$\} and $\{\ep_1,\ep_2\}$.  Each connects some $\de_i$ to itself.  The labels of the $\ga_i$-arcs and the $\ep_i$-arcs are inherited from the labels of the $\de_i$ on which they have their endpoints. The arcs $\ga_1$ and $\ga_2$ divide $S_2 \cut \de_0$ into two annuli, $R_1$ and $R_2$.  

\begin{figure}
\centering
\includegraphics[width=.9\textwidth]{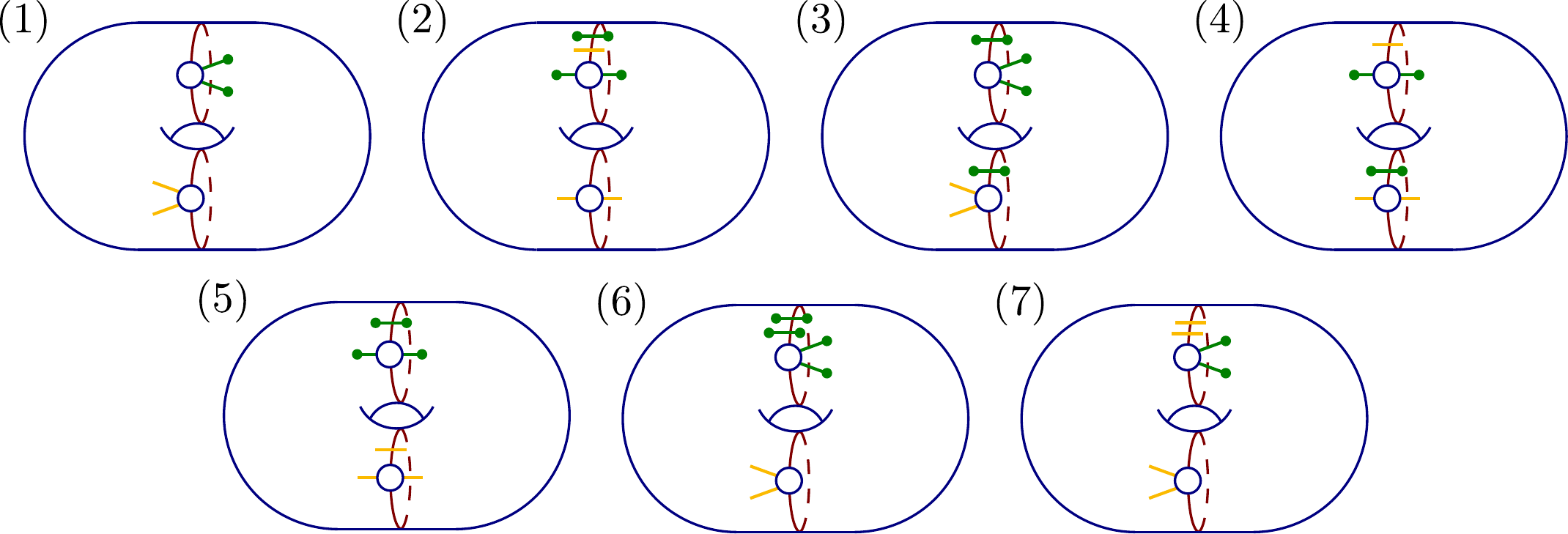}
\caption{The partial templates for configuration III; the partial arcs with a dot belong to $\ep_1$ and the other arcs to $\ep_2$}
\label{fig:partial III}
\end{figure}

We may record the intersection numbers between the $\ga_i$-arcs and $\ep_i$-arcs with a $2 \times 2$ matrix, where the rows and columns correspond to the $\ga_i$ and the $\ep_i$, respectively.  The total sum of the entries in the matrix must be 0 or 2.  Up to renumbering the arcs in each pair, there are seven such matrices:
\[
 \begin{bmatrix}  0 & 0 \\ 0 & 0 \end{bmatrix}, \ 
 \begin{bmatrix}  1 & 1 \\ 0 & 0 \end{bmatrix}, \ 
 \begin{bmatrix}  1 & 0 \\ 1 & 0 \end{bmatrix}, \ 
 \begin{bmatrix}  0 & 1 \\ 1 & 0 \end{bmatrix}, \ 
 \begin{bmatrix}  1 & 0 \\ 0 & 1 \end{bmatrix}, \ 
 \begin{bmatrix}  2 & 0 \\ 0 & 0 \end{bmatrix},   \ \text{and} \ 
 \begin{bmatrix}  0 & 2 \\ 0 & 0 \end{bmatrix}.
\]
The partial templates corresponding to these matrices are shown in Figure~\ref{fig:partial III}.  As in the proof of Proposition~\ref{prop:II}, $\ga_0$ and $\ep_0$ are linked or unlinked along $\de_0$ according to the parities of the column sums.  In the pictures, $R_1$ and $R_2$ are the left- and right-hand sides of each partial template, respectively.  We treat the partial templates in turn.  

\begin{figure}
\centering
\includegraphics[width=.9\textwidth]{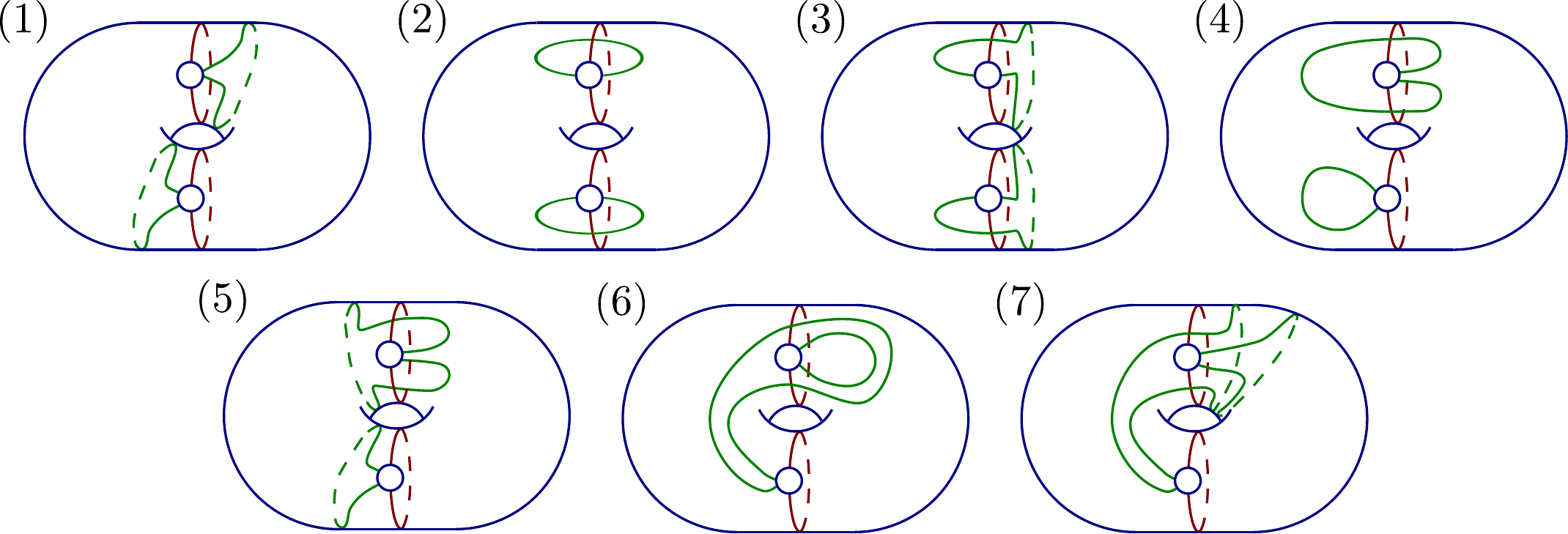}
\caption{The seven templates for triples of type III}
\label{fig:III templates}
\end{figure}

\medskip

\noindent \emph{Partial template 1.} Given two points on a single boundary component of an annulus, there are two choices of arcs in the annulus connecting those points up to homeomorphism fixing the boundary.  Therefore, in each annulus $R_i$ there are two choices of arcs that complete the partial template.  This gives four possibilities.

In each $R_i$ one of the choices of arcs is separating in the cut surface $S_2 \cut \de_0$ and one is nonseparating.  By Lemma~\ref{prune sepnonsep} we must either choose the two separating arcs or the two nonseparating arcs to form a template.  If we choose the two separating arcs, then $\ep_0$ is separating, violating Lemma~\ref{prune sep}.  Thus the only possibility is where $\ep_1$ and $\ep_2$ are nonseparating.  This choice gives template 1 in Figure~\ref{fig:III templates}.  

\medskip

\noindent \emph{Partial templates 2 and 3.} Neither of these partial templates can be completed to be a template.  To see this we consider any of the four curves obtained as the boundary of a neighborhood of $\ga_0 \cup \de_0$.  The algebraic intersection number of any of these curves with $\ep_0$ is odd.  On the other hand, the algebraic intersection number of any of these curves with $\ga_0$ and $\de_0$ is 0.  This contradicts the assumption that $\ep_0$ is homologous to $\ga_0$ and $\de_0$ mod 2.

\medskip

\noindent \emph{Partial template 4.} In this case there is a unique choice of $\ep_0$ up to homeomorphism, and it is a separating curve.  By Lemma~\ref{prune sep} there is no template arising from this partial template.

\medskip

\noindent \emph{Partial template 5.} In order to complete $\ep_1$ from this partial template, we need to choose two arcs, one in each annulus $R_1$ and $R_2$.  This gives four possibilities.  Up to homeomorphism there are only two possibilities for the resulting $\ep_1$, one that is separating in the cut surface $S_2 \cut \de_0$ and one that is nonseparating.  There are similarly four choices for $\ep_2$, two separating and two nonseparating in $S_2 \cut \de_0$.  By  Lemma~\ref{prune sepnonsep} we are left with four choices for the pair $\{\ep_1,\ep_2\}$.  Two of these choices result in a separating $\ep_0$, which violates Lemma~\ref{prune sep}.  The other two choices are templates 2 and 3 in Figure~\ref{fig:III templates}.

\medskip

\noindent \emph{Partial template 6.} The intersection of $\ep_1$ with $R_2$ is a pair of arcs.  Up to homeomorphism, there are two possibilities: they can be nested or not nested.  There are then two choices for the intersection of $\ep_1$ with $R_1$.  This gives four total choices for $\ep_1$.  There are two choices for $\ep_2$, one separating and one not, but by Lemma~\ref{prune sepnonsep} the choice of $\ep_2$ is completely determined by the choice of $\ep_1$.  Thus there are four total candidates for $\ep_0$.  The two candidates corresponding to the nested arcs in $R_2$ are separating, and by Lemma~\ref{prune sep} these are ruled out.  The other two are templates 4 and 5 in Figure~\ref{fig:III templates}.

\medskip

\noindent \emph{Partial template 7.} Up to homeomorphism there are two choices for the intersection of $\ep_2$ with $R_1$, corresponding to the two ways to pair the two points of $\ep_2 \cap \de_2$ with the two points of $\ep_2 \cap \ga_1$.  As usual there are two choices for the intersection of $\ep_2$ with $R_2$.  This gives four choices for $\ep_2$.  For one of the four choices there is no choice of $\ep_1$ that is disjoint from $\ep_2$.  For each of the other three choices for $\ep_2$, the choice of $\ep_1$ is determined by Lemma~\ref{prune sepnonsep}.  One of the resulting candidates for $\ep_0$ is a separating curve, violating Lemma~\ref{prune sep}.  The other two are templates 6 and 7 in Figure~\ref{fig:III templates}.

\medskip

Since all of the templates found are shown in Figure~\ref{fig:III templates}, this completes the first phase of the proof.

\medskip

\noindent \emph{Phase 2.} By Lemma~\ref{templates} we may obtain the minimal configurations of triples of type III by adding handles to the seven templates.  As in the proof of Proposition~\ref{prop:II} the corresponding handle set is subject to the minimality and checkerboard conditions.  In place of the type II condition, we will give an analogous type II condition, as follows.

Two curves $\{\de,\ep\}$ in $S_g$ that intersect in two points and have algebraic intersection number 0 form a pair of type III if and only if $\de \cup \ep$ has two complementary regions, each with two boundary components.  In some of the templates $\de_0 \cup \ep_0$ has three complementary regions.  In these cases, we must add a handle connecting the two non-adjacent regions complementary to $\{\de,\ep\}$, so that the resulting $\{\de,\ep\}$ is of type III.   

\begin{enumerate}
\item[] \emph{Type III condition.} If $\de_0 \cup \ep_0$ has three complementary regions in $S_2$, the handle set must include a pair of regions (for the triple) that are contained in non-adjacent regions of $S_2$ determined by the pair $\{\de_0 , \ep_0\}$. 
\end{enumerate}

For each of the seven templates we will follow three steps to produce all possible minimal configurations of triples of type III.  Each step may produce multiple configurations.
\begin{enumerate}
\item[\emph{Step 1.}] Check the number of components of the complement of $\de_0 \cup \ep_0$; if there are three components, add a single pair to the handle set as in the type III condition while respecting the checkerboard condition.
\item[\emph{Step 2.}] Check in each of the resulting configurations if the $\ep_0$-curve and the $\ga_0$-curve are homotopically distinct in minimal position; if not add pairs to the handle set as above, respecting the checkerboard and minimality conditions.
\item[\emph{Step 3.}] Check in each of the resulting configurations if there are any other pairs of regions that satisfy the checkerboard condition and the minimality condition.  If so we may add these handles to create a new minimal configuration.
\end{enumerate}
This first step is required in order to obtain a triple of type III.  Note that this step is never a stabilization.  The second step might be a stabilization, but it is also required in order to obtain a triple of type III.  The third step is also never a stabilization, but it is required in order to obtain all minimal configurations.  In practice, at most one handle is added in the third step.

We are now ready to determine all minimal configurations of type III obtained from the seven templates.

\medskip

\noindent \emph{Template 1.} In this case the complement of $\de_0 \cup \ep_0$ already has exactly two components, and so the type III condition is satisfied.  However, there is an annulus bounded by $\ga_0$ and $\ep_0$.  Every handle connecting distinct regions fails the checkerboard condition.  Therefore, the only possibility is to add handles within regions.

The annulus between $\ga_0$ and $\ep_0$ is divided into two squares by $\de_0$.  We need to add a handle to one of the two squares, but the two choices give homeomorphic configurations.  The result is configuration 1 in Figure~\ref{fig:III}.

\medskip

\noindent \emph{Templates 2, 4, and 6.} In these cases the complement of $\de_0 \cup \ep_0$ has three components. Up to homeomorphism, there is a unique pair of regions that satisfies the checkerboard condition and the type III condition. These pairs of regions are marked by black dots in configurations 2, 4, and 6 in Figure~\ref{fig:III}.  Adding each of these handles results in a minimal configuration.  For templates 2 and 4, it is possible to further add a single handle between the two regions as in Step 3.  The resulting minimal configurations are given by pictures 2, 4, and 6 in Figure~\ref{fig:III}.  Each of the pictures 2 and 4 really represents two minimal configurations; the stars indicate the additional handle that can be added in Step 3.

\medskip

\noindent \emph{Templates 3, 5, and 7.} In these cases the complement of $\de_0 \cup \ep_0$ already has exactly two components, satisfying the type III condition.  However the $\de_0$ and $\ep_0$ curves are not in minimal position; in fact they form two bigons.  Up to homeomorphism, there are two handle sets that we may use as in Step 2: in each case we may attach a handle that connects the pair of regions indicated by gray dots in pictures 3, 5, and 7 of Figure~\ref{fig:III}, or we may pair each region marked by a gray dot with itself.  As such, we obtain two minimal configurations from each template.  As in Step 3, the two minimal configurations produced from templates 3 and 5 so far allow one additional handle that still results in a minimal configuration.  The corresponding regions are marked by stars in the figure.  Thus, pictures 3 and 5 represent four different minimal configurations each.

\medskip

Since all of the minimal configurations we have constructed are shown in Figure~\ref{fig:III}, this completes the proof of the proposition.  
\end{proof}

\begin{figure}
\centering
\includegraphics[width=.9\textwidth]{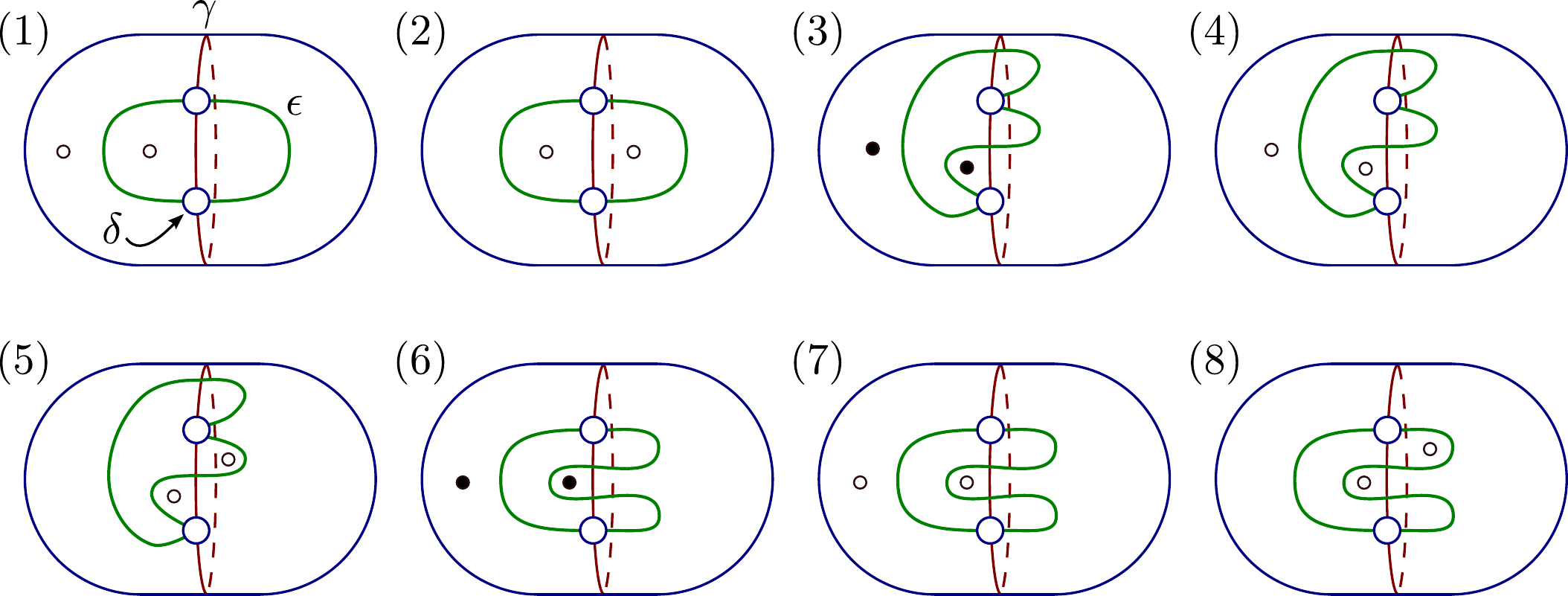}
\caption{The eight minimal configurations of triples of type IV; a white dot indicates a handle attached within a single region and a pair of black dots indicates a handle connecting two distinct regions}
\label{fig:IV}
\end{figure}

\subsection*{Triples of type IV}  Finally we classify ordered triples of type IV.

\begin{proposition}
\label{prop:IV}
Up to homeomorphism, every ordered triple of curves of type IV in $S_g$ is a stabilization of one of the minimal configurations in Figure~\ref{fig:IV}.
\end{proposition}

\begin{proof}

We complete the proof in the same two phases used in the proofs of Propositions~\ref{prop:II} and~\ref{prop:III}, first listing all possible templates in $S_1$ for a triple of type IV, and then building all minimal configurations of type IV from these templates.  

\medskip

\noindent \emph{Phase 1.} Suppose that $(\ga_0,\de_0,\ep_0)$ is a template in $S_1$ for a triple of curves $(\ga,\de,\ep)$ in $S_g$. We will show that up to homeomorphism this template is one of the three templates shown in Figure~\ref{fig:IV templates}.  

\begin{figure}
\centering
\includegraphics[width=.75\textwidth]{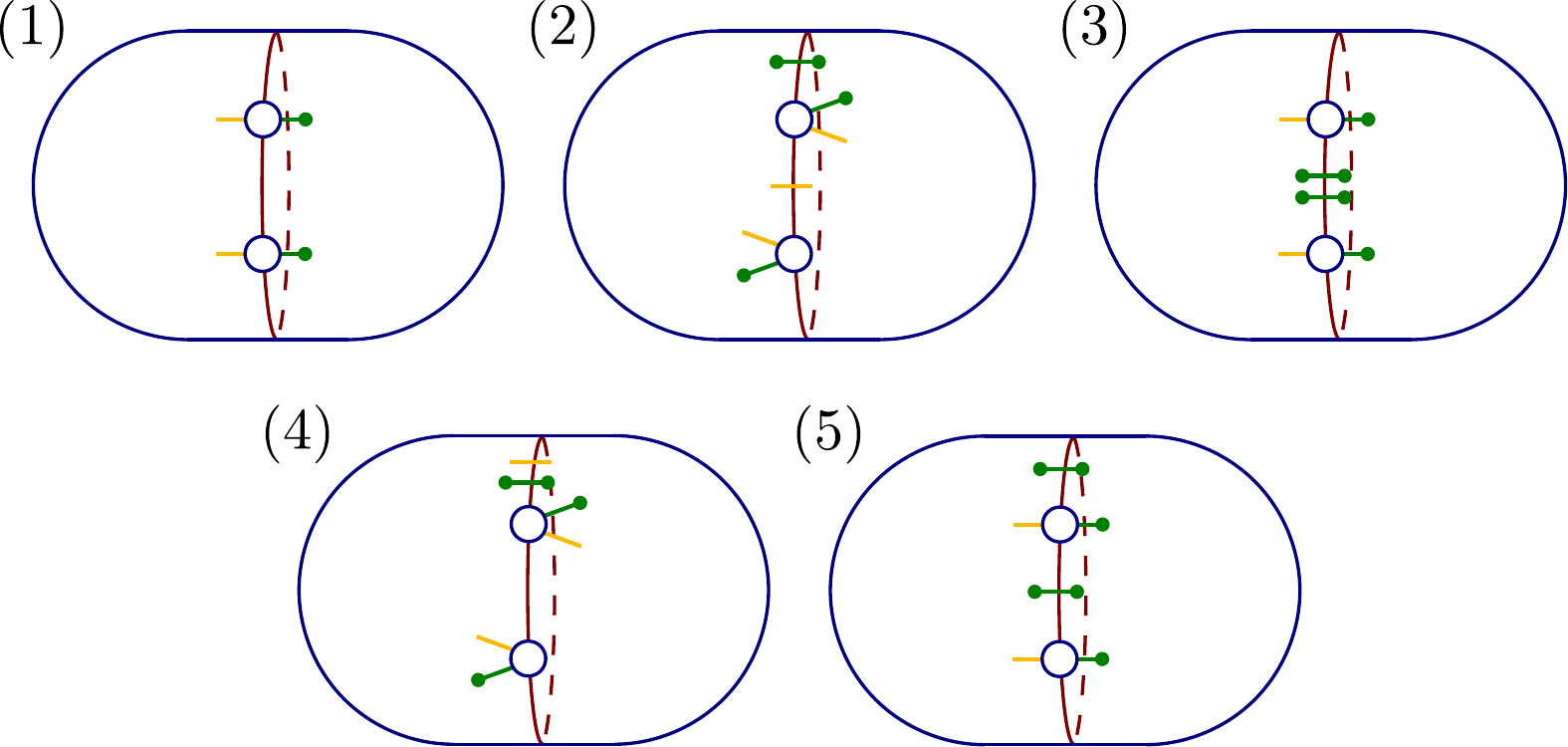}
\caption{The partial templates for configuration IV; the partial arcs with a dot belong to $\ep_1$ and the other arcs belong to $\ep_2$}
\label{fig:partial IV}
\end{figure}

The curves $\ga_0$ and $\de_0$ are configured in $S_1$ as in Figure~\ref{fig:configs template}.  Consider the cut surface $S_1 \cut \delta_0$ and call the resulting boundary components $\de_1$ and $\de_2$.   In the cut surface, $\ga_0$ and $\ep_0$ each become a pair of arcs.  By recording the pairwise intersections, we obtain a $2 \times 2$ matrix where the sum of the entries is 0 or 2.  Up to symmetry there are five matrices:
\[
 \begin{bmatrix}  0 & 0 \\ 0 & 0 \end{bmatrix}, \ 
 \begin{bmatrix}  1 & 0 \\ 0 & 1 \end{bmatrix}, \
 \begin{bmatrix}  0 & 0 \\ 2 & 0 \end{bmatrix}, \ 
 \begin{bmatrix}  1 & 1 \\ 0 & 0 \end{bmatrix}, \   \text{and} \ 
 \begin{bmatrix}  1 & 0 \\ 1 & 0 \end{bmatrix}.
\]
Each matrix again determines whether $\ga_0$ and $\ep_0$ are linked or unlinked along $\de_0$.  The corresponding partial templates are shown in Figure~\ref{fig:partial IV}.  To complete the first phase we must find all ways of completing the partial templates to templates.  We treat the five cases in turn.  Our work is simplified by the fact that the $\ga_0$-arcs divide the $S_1 \cut \delta_0$ into two disks $R_1$ and $R_2$ and the fact that there is a unique arc connecting two points in the boundary of a disk, up to homeomorphism.

\begin{figure}
\centering
\includegraphics[width=.75\textwidth]{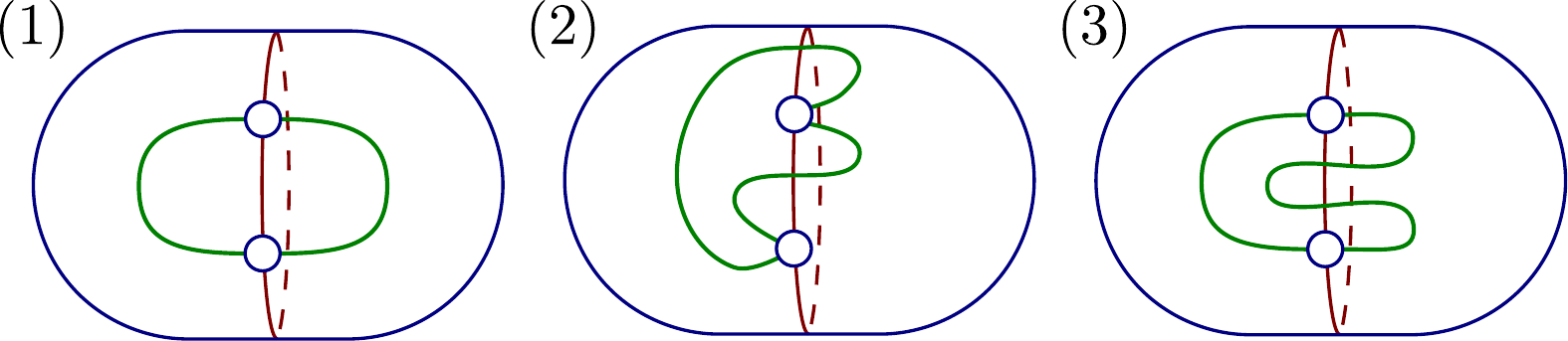}
\caption{The three templates for triples of type IV}
\label{fig:IV templates}
\end{figure}  

\medskip

\noindent \emph{Partial templates 1, 2, and 3.} In each of these cases there is a unique way to complete the partial templates.  The resulting templates are shown in Figure~\ref{fig:IV templates}. 

\medskip

\noindent \emph{Partial templates 4 and 5.} In these cases there is no way to complete the partial templates to templates.  This is because the partial templates indicate that the two $\ep$-arcs in each $R_1$ should be linked along the boundaries of the $R_1$, but this is impossible since $R_1$ is a disk and the arcs must be disjoint.  

\medskip

Since all of the templates found are shown in Figure~\ref{fig:IV templates}, this completes the first phase of the proof.

\medskip

\noindent \emph{Phase 2.} By Lemma~\ref{templates} we may obtain the minimal configurations of triples of type IV by adding handles to the three templates.  

We follow the same steps as in the proof of Proposition~\ref{prop:III} (the classification of type III triples) in order to build the minimal configurations of type IV from the templates.  As in the proofs of Propositions~\ref{prop:II} and~\ref{prop:III}, the added handles must satisfy the minimality condition and the checkerboard condition.  Also, we must add handles so that in the resulting configuration the three curves are pairwise in minimal position and homotopically distinct.

One simplifying feature in the present situation is that the pair $\{\de_0,\ep_0\}$ in the three templates is automatically of type IV.  Indeed, any two curves that intersect in two points and have nonzero algebraic intersection are in minimal position.  Therefore we may skip Step 1 from the proof of Proposition~\ref{prop:III} in each of the three cases.  In other words, there is no need for a type IV condition, analagous to the type II condition and the type II conditions used in the proof of Propositions~\ref{prop:II} and~\ref{prop:III}.

\medskip

\noindent \emph{Template 1.} In this case $\ga_0$ and $\ep_0$ are disjoint and homotopic, so they bound two annuli in $S_1$.  Each annulus is cut into two squares by $\de_0$.  Up to homeomorphism, there are two ways to add handles to the template that satisfy the minimality and checkerboard conditions and that result in homotopically distinct $\ga_0$- and $\ep_0$-curves.  The two resulting configurations are configurations 1 and 2 in Figure~\ref{fig:IV}.

\medskip

\noindent \emph{Template 2.} In this case $\ga_0$ and $\ep_0$ are homotopic in $S_1$ and they intersect in two points.  As such, they bound two bigons.  Up to homeomorphism, there are three ways to add handles in order to obtain a minimal configuration, resulting in configurations 3, 4, and 5 in Figure~\ref{fig:IV}.  

\medskip

\noindent \emph{Template 3.} Again $\ga_0$ and $\ep_0$ are homotopic in $S_1$ and intersect in two points and so they bound two bigons.  Up to homeomorphism, there are three ways to add handles to obtain a minimal configuration, resulting in configurations 6, 7, and 8 in Figure~\ref{fig:IV}. 

\medskip

Since all of the minimal configurations we have constructed are shown in Figure~\ref{fig:IV}, this completes the proof of the proposition.
\end{proof}


\section{A well-suited curve criterion for separating subsurfaces}
\label{sec:gp}

In this section we give our main tool for addressing Case 3 of the proof of Theorem~\ref{main:pa}, as outlined in Section~\ref{sec:int}.  One approach to addressing Case 3 would be to show that the curve graph $\N_f(S_g)$ from Section~\ref{sec:general} is connected and to apply Proposition~\ref{prop:wscc}.  However, the only edges of $\N_f(S_g)$ that we have in hand are the ones in the orbit of the edge $\{c,f(c)\}$.  Since $[c] = [f(c)] \mod 2$ there is no path of such edges connecting $c$ to any other curve in a different mod 2 homology class.

Our solution is to develop an analogue of our well-suited curve criterion for separating curves, Lemma~\ref{wsccsep}.  In that lemma we also have a curve $d$ whose homology class (the trivial one) is preserved.  However, we are able to use the fact that $i(d,f(d)) \leq 2$ in order to find a different curve $a$ that satisfies one of our well-suited curve criteria for nonseparating curves, Lemma~\ref{wsccb}.  We will follow a similar approach here.

Our first goal is to state our generalization of Lemma~\ref{wsccsep}, namely, Lemma~\ref{good pair}.  Then we will state two specific consequences, Lemmas~\ref{gpbp} and~\ref{gp}, that will be applied directly in the proof of Theorem~\ref{main:pa}.

\subsection*{The good pair lemma} Suppose that $D$ is a subsurface of $S_g$, not necessarily connected.  Let $a$ be a curve in $S_g$, and say that $a$ lies in the region $R$ complementary to $D$.  We say that the homology class $[a] \in H_1(S_g;\R)$ is \emph{local} with respect to $D$ if it is not represented by any curve lying in a complementary region to $D$ that is distinct from $R$.  When there is no confusion, we will refer to a homology class as simply being ``local,'' the subsurface $D$ being understood.

\begin{lemma}
\label{good pair}
Let $g \geq 0$ and let $f \in \Mod(S_g)$.  Let $D$ be a separating subsurface of $S_g$.  Suppose that $a$ and $b$ are nonseparating curves in $S_g$ whose homology classes are local with respect to $D$.   If $a$ and $b$ lie in distinct complementary regions of $D$ and lie in the same complementary region of $f(D)$ then the normal closure of $f$ contains the commutator subgroup of $\Mod(S_g)$.  
\end{lemma}

When there exists a subsurface $D$ and curves $a$ and $b$ as in the statement of Lemma~\ref{good pair} we say that $a$ and $b$ form a \emph{good pair of curves} for the mapping class $f$ (the subsurface $D$ being understood).  

Lemma~\ref{wsccsep} is in fact a special case of Lemma~\ref{good pair}: the subsurface $D$ there is the annular neighborhood of the separating curve $d$ and it is easy to find the appropriate curves $a$ and $b$ in this case.  Our proof of Lemma~\ref{good pair} is a straightforward generalization of our proof of Lemma~\ref{wsccsep}.

There is a symmetry in the statement of Lemma~\ref{good pair}: if $a$ and $b$ lie in different complementary regions of $D$ but in the same complementary region of $f(D)$ then we may apply the lemma to $f^{-1}$ instead of $f$.  Of course $f$ is a normal generator if and only if $f^{-1}$ is.  

\begin{proof}[Proof of Lemma~\ref{good pair}]

By hypothesis, $a$ and $b$ lie in distinct complementary regions of $D$.  Therefore the curves $f(a)$ and $f(b)$ lie in distinct complementary regions of $f(D)$.  Also by hypothesis, $a$ and $b$ lie in the same complementary region of $f(D)$.  Combining the last two sentences, it must be that either $a$ and $f(a)$ lie in distinct complementary regions of $f(D)$ or $b$ and $f(b)$ lie in distinct complementary regions of $f(D)$ (or both statements hold).  Without loss of generality, suppose the former holds.  Since $[a]$ is local with respect to $D$, it follows that $[f(a)]$ is local with respect to $f(D)$, and in particular that $[a] \neq [f(a)]$.  Clearly $i(a,f(a))=0$.  An application of Lemma~\ref{wsccb} completes the proof.
\end{proof}

The next lemma gives two ways of finding curves whose homology classes are local in the sense of Lemma~\ref{good pair}.  In fact, such curves are easy to find.

\begin{lemma}
\label{local}
Let $g \geq 0$, and let $D$ be a separating subsurface of $S_g$.
\begin{enumerate}
\item If $D$ is connected, then every nonseparating curve in $S_g$ in the complement of $D$ represents a homology class that is local.  
\item Every curve in the complement of $D$ that does not separate the complement of $D$ represents a homology class that is local.
\end{enumerate}
\end{lemma}

\begin{proof}

Let $a$ be a nonseparating curve in $S_g$ that lies in the complementary region $R$ of $D$.  For both statements it suffices to find a curve $b$ so that $i(a,b)=1$ and so that $b$ is contained in the smallest subsurface containing $D \cup R$.  If $a$ is nonseparating in $R$, then the curve $b$ can be taken to lie in $R$.  The second statement follows.  If $a$ is separating in $R$, then since $a$ is nonseparating in $S_g$ it must be that $a$ induces a nontrivial partition of the components of the boundary of $R$.  It then follows that we can find the desired curve $b$ in $D \cup R$.  The first statement follows.
\end{proof}

\subsection*{Applications for iterates of curves}  We now explain how Lemma~\ref{good pair} will be used in our proof of Theorem~\ref{main:pa}.  As mentioned above, we will consider a mapping class $f$ and a curve $c$.  The role of $D$ will be played by a neighborhood of $c \cup f(c)$, so that $f(D)$ corresponds to a neighborhood of $f(c) \cup f^2(c)$.  If $c$ and $f(c)$ form a pair of type I, then $c \cup f(c)$ is a bounding pair and $D$ is not connected.  On the other hand, if $c$ and $f(c)$ form a pair of type II, III, or IV, then $i(c,f(c))$ is nonzero and so $D$ is connected.

As in Section~\ref{sec:configs}, we may cut $S_g$ along $f(c)$ in order to obtain a surface with two boundary components.  The curves $c$ and $f^2(c)$ each correspond to a collection of arcs on the cut surface.  Finding a good pair of curves $a$ and $b$ as in the statement of Lemma~\ref{good pair}, with $D=c \cup f(c)$, then reduces to finding a pair of nonseparating curves $a$ and $b$ on the cut surface so that:
\begin{enumerate}
\item $a$ and $b$ represent homology classes of $S_g$ that are local with respect to $D$,
\item $a$ and $b$ lie on different sides of the $c$-arcs, and
\item  $a$ and $b$ lie on the same side of the $f^2(c)$-arcs
\end{enumerate}
(the term ``sides'' here is perhaps an abuse of terminology; what we precisely mean is that $a$ and $b$ lie in different components of the complement of $c$ in the cut surface and in the same component of the complement of $f^2(c)$ in the cut surface).  If $i(c,f(c))=0$ then by Lemma~\ref{local} the first condition is satisfied whenever $a$ and $b$ are not parallel to a component of the boundary of the cut surface.  Otherwise by Lemma~\ref{local} the first condition is automatically satisfied.

We summarize the above discussion with the following two lemmas.  The first addresses the case $i(c,f(c)) = 0$ and the second addresses the case $i(c,f(c)) \neq 0$.  In both cases the hypotheses of the lemma force the genus $g$ to be at least 3, which is why we can conclude that $f$ is a normal generator and not just that its normal closure contains the commutator subgroup.  At the same time, we could just as well assume here that $g$ is at least 3, since this will be the case when we apply these lemmas in the proof of Theorem~\ref{main:pa}.  

\begin{lemma}
\label{gpbp}
Let $f \in \Mod(S_g)$.  Suppose there is a curve $c$ so that $i(c,f(c))=0$ and so that $[c] = [f(c)]$.  Suppose further there are curves $a$ and $b$ in $S_g$ so that on the surface obtained by cutting $S_g$ along $f(c)$ we have that
\begin{enumerate}
\item $a$ and $b$ lie on different sides of $c$,
\item $a$ and $b$ lie on the same side of $f^2(c)$, and
\item $a$ and $b$ are nonseparating curves in the surface obtained by further cutting along $c$.
\end{enumerate}
Then $a$ and $b$ form a good pair for $f$, and so $f$ is a normal generator for $\Mod(S_g)$.
\end{lemma}

\begin{lemma}
\label{gp}
Let $f \in \Mod(S_g)$.  Suppose there is a curve $c$ and nonseparating curves $a$ and $b$ in $S_g$ so that $i(c,f(c)) > 0$ and so that on the surface obtained by cutting $S_g$ along $f(c)$ we have that 
\begin{enumerate}
\item $a$ and $b$ lie on different sides of $c$ and
\item $a$ and $b$ lie on the same side of $f^2(c)$.
\end{enumerate}
Then $a$ and $b$ form a good pair for $f$, and so $f$ is a normal generator for $\Mod(S_g)$.
\end{lemma}

In Figure~\ref{fig:good pair examples} we give two examples of good pairs as in Lemma~\ref{gp}.

\begin{figure}
\centering
\includegraphics[scale=.3]{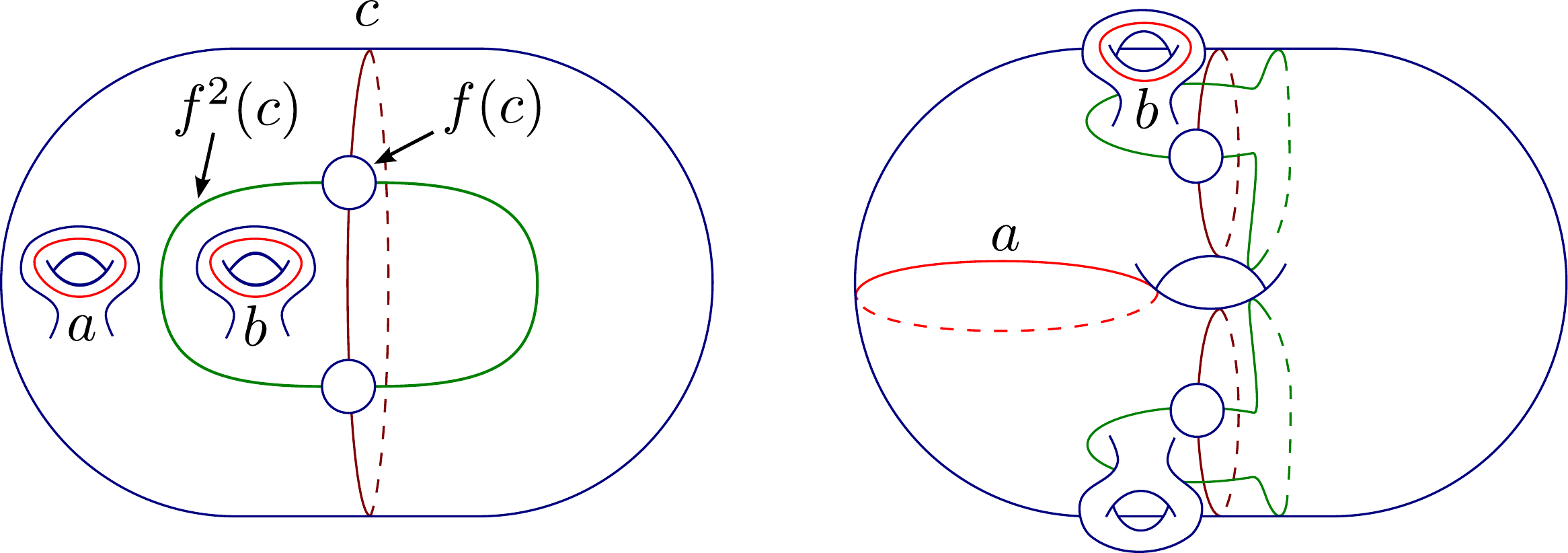}
\caption{In each case the curves $a$ and $b$ form a pair of good curves for $f$, with the role of $D$ being played by a neighborhood of $c \cup f(c)$}
\label{fig:good pair examples}
\end{figure}


\section{Application: pseudo-Anosov mapping classes with small stretch factor}
\label{sec:pA}

In this section we use the tools developed in Sections~\ref{sec:int}--\ref{sec:gp} to prove Theorem~\ref{main:pa}, which states that a pseudo-Anosov element of $\Mod(S_g)$ with stretch factor less than $\sqrt{2}$ is a normal generator for $\Mod(S_g)$.  

In Section~\ref{sec:configs} we introduced terminology for various types of pairs and ordered triples of curves, namely, types I, II, III, and IV.  In this section we say that a pair or a triple of (homotopy classes of) curves is of type I, II, III, or IV if there are representatives of the curves that form a configuration of that type.  The representatives chosen do not affect the designation into types. 

\begin{proof}[Proof of Theorem~\ref{main:pa}]

As discussed in the introduction, we may assume that $g \geq 3$ since there does not exist a pseudo-Anosov mapping class in $\Mod(S_g)$ with stretch factor less than $\sqrt{2}$ when $g < 3$.  Let $f \in \Mod(S_g)$ be a pseudo-Anosov mapping class with stretch factor less than $\sqrt{2}$.  We would like to show that $f$ is a normal generator for $\Mod(S_g)$.  We will apply the various well-suited curve criteria we have developed.  Some of these lemmas only imply that the normal closure of $f$ contains the commutator subgroup of $f$, but since $g \geq 3$ this implies that $f$ is a normal generator.

We follow the plan we established in Section~\ref{sec:int}.  Let $c$ be a shortest curve in $S_g$ with respect to some $f$-metric.  Since the stretch factor of $f$ is less than $\sqrt{2} < 3/2$ it follows from Proposition~\ref{flm} that $i(c,f(c)) \leq 2$.  If $c$ is separating then $f$ is a normal generator by our well-suited curve criterion for separating curves, Lemma~\ref{wsccsep}.  We may thus assume for the remainder of the proof that $c$ is nonseparating.  

Suppose first that $[c] \neq [f(c)] \mod 2$.  There are three possibilities for $i(c,f(c))$: 0, 1, or 2.  Applying Lemmas~\ref{wsccb}, \ref{wscca}, and \ref{lemma:nonsep lantern} to the three cases, respectively, we conclude that $f$ is a normal generator for $\Mod(S_g)$.

We may now assume for the remainder of the proof that $[c] = [f(c)] \mod 2$.  This is Case 3 from Section~\ref{sec:int}.  We have that $i(c,f(c)) = i(f(c),f^2(c))$ and $i(c,f^2(c))$ are both equal to 0 or 2.  Again, this follows from the $(k,n)=(2,4)$ case of Proposition~\ref{flm} plus the fact that $i(c,f^2(c))$ is even (since $f$, hence $f^2$ preserves $[c] \mod 2$).

It follows from the previous paragraph that $\{c,f(c)\}$, $\{f(c),f^2(c)\}$, and $\{c,f^2(c)\}$ are each of type I, II, III, or IV.  Since $\{c,f(c)\}$ and $\{f(c),f^2(c)\}$ are of the same type (they differ by $f$),  and since $c$, $f(c)$, and $f^2(c)$ are all distinct ($f$ is pseudo-Anosov), the ordered triple $(c,f(c),f^2(c))$ is of type I, II, III, or IV.  We treat the four possibilities in turn.

\medskip

\noindent \emph{Type I.} By Proposition~\ref{prop:I}, the configuration $(c,f(c),f^2(c))$ is a stabilization of one of the two configurations in Figure~\ref{fig:I}.  In the first configuration there is a good pair of curves as in Lemma~\ref{gpbp}, and so $f$ is a normal generator.

For the second configuration, there may not be a good pair of curves.  Consider then the quadruple of curves $(c,f(c),f^2(c),f^3(c))$.  Since $\sqrt{2} < \sqrt[3]{3}$, the stretch factor of $f$ is less than $\sqrt[3]{3}$ and so by the $(k,n)=(3,6)$ case Proposition~\ref{flm} and the fact that $[c] = [f^3(c)] \mod 2$ we have $i(c,f^3(c)) \leq 4$.  

The curves $c$ and $f^2(c)$ form a pair of type II.  As such, there are two separating curves $d_1$ and $d_2$ obtained as a boundary component of a neighborhood of $c \cup f^2(c)$.  We will show that at least one of $i(d_1,f(d_1))$ or $i(d_2,f(d_2))$ is at most 2.  It will then follow from Lemma~\ref{wsccsep} that $f$ is a normal generator.

The curves $c$ and $f^2(c)$ cut each other into two arcs each, and we may think of $d_1$ and $d_2$ as being formed from these four arcs.  Two of the arcs make up $d_1$ and the other two make up $d_2$.  Similarly, $f(d_1)$ and $f(d_2)$ are made of the four corresponding arcs of $f(c)$ and $f^3(c)$.  We may therefore understand how $f(d_1)$ and $f(d_2)$ intersect $d_1$ and $d_2$ by understanding how $f(c)$ and $f^3(c)$ intersect $c$ and $f^2(c)$.  Specifically,
\begin{align*}
i(f(d_1),&d_1) + i(f(d_1),d_2) + i(f(d_2),d_1) + i(f(d_2),d_2) \\ &\leq  i(c,f(c)) + i(f^2(c),f(c)) + i(c,f^3(c)) + i(f^2(c),f^3(c)).
\end{align*}
Let us examine the right-hand side.  By assumption, \[i(c,f(c)) = i(f(c),f^2(c)) = i(f^2(c),f^3(c)) = 0.\]  We already said that $i(c,f^3(c)) \leq 4$ and so the left-hand side is at most 4.  Thus $i(f(d_i),d_i) < 4$ for some $i$.   Since $d_i$ is separating it must be that $i(f(d_i),d_i) \leq 2$ as desired.

\medskip

\noindent \emph{Type II.} By Proposition~\ref{prop:I}, the configuration $(c,f(c),f^2(c))$ is a stabilization of one of the ten configurations in Figure~\ref{fig:II}.  For configurations 1--5 there is a good pair of curves as in Lemma~\ref{gp}, and so $f$ is a normal generator.  For configurations 6--10, we consider either of the two separating curves obtained as a component of the boundary of  a neighborhood of $c \cup f(c)$ (these are $\ga \cup \de$ in the pictures); call it $d$.  Its image is one of the two separating curves $e_1$ and $e_2$ obtained as a component of the boundary of $f(c) \cup f^2(c)$ (these are $\de \cup \ep$ in the pictures).  Each of $i(d,e_1)$ and $i(d,e_2)$ is either 0 or 2.  It thus follows from Lemma~\ref{wsccsep} that $f$ is a normal generator.  
(The argument for configurations 6--10 also applies to configurations 1--5.)

\medskip

\noindent \emph{Type III.} By Proposition~\ref{prop:III}, the configuration $(c,f(c),f^2(c))$ is a stabilization of one of the  sixteen configurations in Figure~\ref{fig:III}. There are seven pictures in the figure.  Some of the pictures correspond to more than one minimal configuration, but our arguments will apply uniformly to all the different minimal configurations represented by the same picture.

In pictures 1--5 we find a good pair of curves and so by Lemma~\ref{gp} we conclude that $f$ is a normal generator.  For the configurations depicted in pictures 6 and 7, consider the four curves lying on the boundary of a neighborhood of (representatives of) $c$ and $f(c)$ (these are $\gamma$ and $\delta$ in the pictures).    These four curves must map under $f$ to the four curves lying on the boundary of a neighborhood of $f(c)$ and $f^2(c)$ (these are $\de$ and $\ep$ in the pictures).  In each of the pictures, one of the latter four curves is the curve $d$ surrounding one of the black or gray dots. (Note that since $d$ is the image of a nonseparating curve it must be itself nonseparating; hence the configuration where the gray dot represents a handle connecting the region to itself does not arise here.)  The curve $d$ is disjoint from the first four curves and not homologous to any of them (we can find a curve that intersects each of the first four curves in a single point and does not intersect $d$).  By Lemma~\ref{wsccb}, we have that $f$ is a normal generator in each case.

\medskip

\noindent \emph{Type IV.} By Proposition~\ref{prop:IV}, the configuration $(c,f(c),f^2(c))$ is a stabilization of one of the eight configurations in Figure~\ref{fig:IV}. For configurations 1 and 2 there is a good pair of curves and so by Lemma~\ref{gp} we have that $f$ is a normal generator.  

\begin{figure}[h]
\centering
\includegraphics[scale=.32]{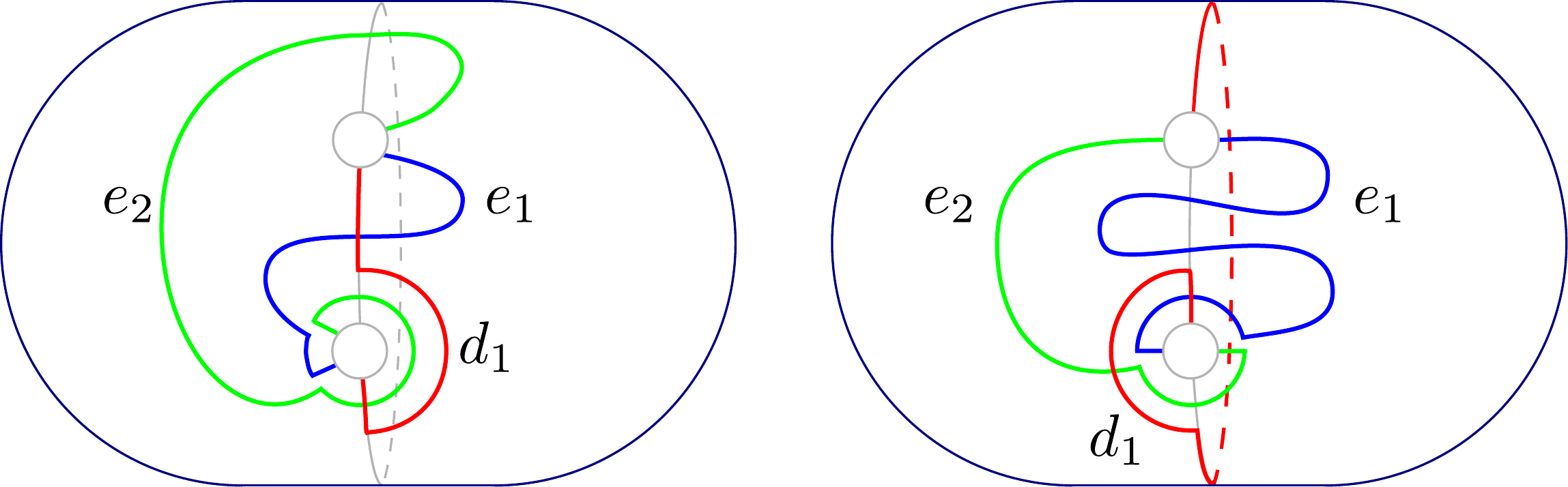}
\caption{The curves $d_1$, $e_1$ and $e_2$ for configurations 3--5 (left) and 6--8 (right)}
\label{fig:s}
\end{figure}

Next consider configurations 3--5.  The curves $c$ and $f(c)$ ($\ga$ and $\de$ in the picture) intersect in two points, and hence each cuts the other into two arcs.  There are four curves in the surface comprised of one of the two arcs of $c$ and one of the two arcs of $f(c)$.  Two of these four curves make a left turn onto $f(c)$ from $c$, call them $d_1$ and $d_2$; see the left-hand side of Figure~\ref{fig:s}.  We can similarly make two curves $e_1$ and $e_2$ in $f(c) \cup f^2(c)$ that turn left onto $f^2(c)$.  The former pair of curves must map to the latter.  Since $i(d_1,e_1) = i(d_1,e_2) = 1$ and since $f(d_1)$ is either $e_1$ or $e_2$, it follows from Lemma~\ref{wscca} that $f$ is a normal generator.

For configurations 6--8 the argument works in the same way.  The analogous curves $d_1$, $e_1$, and $e_2$ are shown in the right-hand side of Figure~\ref{fig:s}.  Thus $f$ is a normal generator in this case as well.   This completes the proof of the theorem.
\end{proof}

\bibliographystyle{plain}
\bibliography{normal}

\end{document}